\documentclass[11pt]{amsart}
\usepackage{amsmath,amsfonts,amsbsy,amsgen,amscd,mathrsfs,amssymb,amsthm,mathtools,enumerate,bbm}
\usepackage[colorlinks=true,citecolor=blue,linkcolor=blue]{hyperref}
\usepackage[top=1.28in, left=1.15in, bottom=1.15in, right=1.15in]{geometry}

\usepackage{hhline}

\usepackage{accents}

\usepackage{pgfplots}
\usepackage{tikz-cd}
\usetikzlibrary{arrows,automata}

\numberwithin{equation}{section}
\newtheorem{theorem}{Theorem}[section]

\newtheorem{conjecture}[theorem]{Conjecture}
\newtheorem{lemma}[theorem]{Lemma}
\newtheorem{proposition}[theorem]{Proposition}
\theoremstyle{definition}
\newtheorem{assumption}[theorem]{Assumption}
\newtheorem{definition}[theorem]{Definition}

\newtheorem{problems}[theorem]{Problems}

\newtheorem{remark}[theorem]{Remark}

\makeatletter\renewenvironment{proof}[1][\proofname] {\par\pushQED{\qed}\normalfont\topsep6\p@\@plus6\p@\relax\trivlist\item[\hskip\labelsep\bfseries#1\@addpunct{.}]\ignorespaces}{\popQED\endtrivlist}

\newcommand\al{\alpha}
\newcommand\be{\beta}
\newcommand\dd{\mathrm d}
\newcommand\De{\Delta}
\newcommand\de{\delta}
\newcommand\deq{\stackrel{\mathrm{distr.}}{=}}
\newcommand\eps{\varepsilon}
\newcommand\ga{\gamma}

\newcommand\ka{\kappa}

\newcommand\la{\lambda}
\newcommand{\om}{\omega}

\newcommand\Si{\Sigma}
\newcommand\si{\sigma}

\newcommand\ze{\zeta}
\renewcommand\d{~\mathrm d}
\renewcommand\phi{\varphi}

\newcommand\bs{\boldsymbol}
\newcommand\mbb{\mathbb}
\newcommand\mbf{\mathbf}
\newcommand\mc{\mathcal}
\newcommand\mf{\mathfrak}
\newcommand\mr{\mathrm}
\newcommand\ms{\mathscr}
\newcommand\msf{\mathsf}

\begin{document}

\title{Moment Intermittency in the PAM with Asymptotically Singular Noise}
\author{Pierre Yves Gaudreau Lamarre, Promit Ghosal, and Yuchen Liao}
\address{}
\email{}
\subjclass[2010]{}
\maketitle

\begin{abstract}
Let $\xi$ be a singular Gaussian noise on $\mbb R^d$
that is either white, fractional, or with the Riesz covariance kernel; in particular,
there exists a scaling parameter $\om>0$ such that
$c^{\om/2}\xi(c\cdot)$ is equal in distribution to $\xi$ for all $c>0$.
Let $(\xi_\eps)_{\eps>0}$ be a
sequence of smooth mollifications such that $\xi_\eps\to\xi$ as $\eps\to0$.
We study the asymptotics of the moments
of the parabolic Anderson model (PAM) with noise $\xi_\eps$ as $\eps\to0$,
both for large (i.e., $t\to\infty$) and fixed times $t$.
This approach makes it possible to study the moments of
the PAM with regular and singular noises in a unified fashion,
as well as interpolate between the two settings.
As corollaries of our main results, we obtain the following:

\noindent\textbf{(1)} When $\xi$ is subcritical (i.e., $0<\om<2$),
our results extend the known large-time moment
and tail asymptotics for the Stratonovich PAM with noise $\xi$.
Our method of proof clarifies
the role of the maximizers of the variational problems (known as Hartree ground states) that appear in these moment asymptotics 
in describing the geometry of intermittency. We take this opportunity to prove
the existence and study the properties of the Hartree ground
state with a fractional kernel, which we believe is of independent interest.

\noindent \textbf{(2)} When $\xi$ is critical or supercritical (i.e., $\om=2$ or $\om>2$), our results provide
a new interpretation of the moment blowup phenomenon observed in the Stratonovich PAM with noise $\xi$.
That is, we uncover that the latter is related to an intermittency effect that occurs in the
PAM with noise $\xi_\eps$ as $\eps\to0$ for {\it fixed finite times} $t>0$.
\end{abstract}

\section{Introduction}

In this paper, we are interested in the continuous parabolic Anderson model (PAM)
with a time-independent Gaussian noise. That is, the partial differential equation with random coefficients
\begin{align}
\label{Equation: PAM}
\begin{cases}\partial_t u(t,x)=\big(\ka\De+\xi(x)\big)u(t,x)\\
u(0,x)=1
\end{cases}
,\qquad t\geq0,~x\in\mbb R^d,
\end{align}
where $d\in\mbb N$ and $\ka>0$ are fixed, $\De$ is the Laplacian operator,
and $\xi:\mbb R^d\to\mbb R$ is a centered stationary Gaussian process. A
fundamental problem involving the PAM consists of understanding the
occurrence of intermittency at large times. Informally, the latter refers to the observation
that the random field $x\mapsto u(t,x)$ tends to develop tall and narrow peak-like formations
as $t\to\infty$. We refer to \cite[Chapter 1]{KonigBook} and references therein
for a recent overview.

As explained in \cite{Molchanov},
arguably the most elementary rigorous indication that intermittency occurs in
\eqref{Equation: PAM} is a growth property of its moments: More specifically,
suppose that there exists a scale function $A$ such that $A(t)\to\infty$ as $t\to\infty$
and such that the limits
\[\ell_p:=\lim_{t\to\infty}\frac{\log\langle u(t,x)^p\rangle}{A(t)},\]
which are called Lyapunov exponents,
exist and are finite for some values of $p>0$.

\begin{remark}
Here, as is common in a subset of the PAM literature, we use $\langle\cdot\rangle$
to denote the expectation with respect to $\xi$.
Also, since $x\mapsto\xi(x)$ is stationary, so is $x\mapsto u(t,x)$;
hence $\ell_p$ does not depend on the choice of $x$.
\end{remark}

\begin{definition}
\label{Definition: Intermittency}
We say that \eqref{Equation: PAM} is intermittent
if there exists a subset of $(0,\infty)$---which contains at least two points---on which
the function $p\mapsto\ell_p/p$ is strictly increasing.
\end{definition}

We refer to \cite[Pages 210--212]{Molchanov} for an explanation of
the geometric significance of this condition.
The classical theory of the continuous PAM concerns the case where
$\xi$ is a regular random process (e.g., H\"older continuous).
In this setting, the intermittency phenomenon is very well understood,
as shown in the works of Carmona, G\"artner, K\"onig, and Molchanov \cite{CarmonaMolchanov,GartnerKonigMolchanov,GartnerKonig}.
In contrast to that, in recent years there has been a growing interest in extending this theory to
singular noises. In this paper, our aim is to contribute to the understanding of intermittency in the PAM with a
singular Gaussian noise---where the product $\xi\cdot u(t,\cdot)$ is interpreted in the Stratonovich sense---via a study of the moments.
We discuss our approach and state our main results in Section \ref{Section: Main Results}.

In the remainder of this section,
we introduce the singular noises that we are interested in (Section \ref{Section: PAM with Singular Noise}),
we briefly recall the current state of the art concerning the moments
of the regular and singular PAM (Section \ref{Section: Moments of the Singular PAM}),
we provide some motivation for our investigations and a brief overview of our results (Section \ref{Section: Two Problems}),
we discuss open questions that arise from our results (Section \ref{Section: Discussion}),
and we give an outline of the organization of this paper (Section \ref{Section: Organization}).

\subsection{Stratonovich PAM with Singular Noise}
\label{Section: PAM with Singular Noise}

We make the following assumption on $\xi$:
\begin{assumption}
\label{Assumption: Noise}
Informally, we view $\xi$ as a centered Gaussian process on $\mbb R^d$ with covariance
\[\langle\xi(x)\xi(y)\rangle=\ga(x-y),\qquad x,y\in\mbb R^d,\]
where $\ga$ is one of the following three kernels:
Letting $\si>0$ be a fixed constant,
\begin{enumerate}
\item ({\bf White}) $\ga(x)=\si^2\de_0(x)$, where $\de_0$ denotes the Dirac delta distribution.
\item ({\bf Riesz}) $\ga(x)=\si^2|x|^{-\om}$ for some $0<\om<d$.
\item ({\bf Fractional}) $\ga(x)=\si^2\prod_{i=1}^d|x_i|^{-\om_i}$ for some $\om_i\in(0,1)$.
\end{enumerate}
Since $\ga$ is either a Schwartz distribution or has a singularity at zero,
$\xi$ must be defined rigorously as a random Schwartz distribution. More specifically,
$\xi$ is a centered Gaussian process on $C_0^\infty(\mbb R^d)$ (the smooth and compactly
supported functions) with covariance
\[\langle\xi(f)\xi(g)\rangle=\iint_{(\mbb R^d)^2}f(x)\ga(x-y)g(y)\d x\dd y,\qquad f,g\in C_0^\infty(\mbb R^d).\]
\end{assumption}

\begin{remark}
\label{Remark: Scaling Parameter}
The three covariance kernels listed in Assumption \ref{Assumption: Noise} all satisfy
\[\ga(c x)=c^{-\om}\ga(x),\qquad c>0,~x\in\mbb R^d\]
for some scaling exponent $\om>0$
($\om=\sum_{i=1}^d\om_i$ in the case of fractional noise and $\om=d$ in the case of white noise).
$\om$ can be viewed as a parameter that quantifies the extent to which
$\xi$ is singular, i.e., a larger $\om$ corresponds to a more singular covariance/noise.
\end{remark}

The main difficulty in dealing with the PAM with singular noise is that $\xi$ is not
actually a function that can be evaluated pointwise. Thus, extending the results on the 
intermittency of the PAM with regular noise to the singular setting poses serious technical challenges. In fact, the
very definition of the singular PAM is nontrivial, as the pointwise product $\xi(x)u(t,x)$
in \eqref{Equation: PAM} requires a careful justification.
One way to get around this issue,
which leads to the Stratonovich solution,
is to proceed as follows:

\begin{definition}
\label{Definition: Mollifier}
For every $\eps>0$, denote the Gaussian kernel with variance $\eps^2$ as
\[p_\eps(x):=\frac{\mr e^{-|x|^2/2\eps^2}}{(2\pi \eps^2)^{d/2}},\qquad x\in\mbb R^d,\]
noting that $p_\eps(x)=\eps^{-d}p_1(x/\eps)$. We define the mollified noise
$\xi_\eps:=\xi*p_{\eps/\sqrt 2}$, where $*$ denotes the convolution. In particular,
$\xi_\eps$ is a centered Gaussian process on $\mathbb R^d$ with covariance
\[\langle\xi_\eps(x)\xi_\eps(y)\rangle=\ga_\eps(x-y),\qquad x,y\in\mbb R^d,\]
where $\ga_\eps:=\ga*p_\eps$.
\end{definition}
With this in hand, for every $\eps>0$, we may now consider the smoothed PAM
\begin{align}
\label{Equation: Smoothed PAM}
\begin{cases}\partial_t u_\eps(t,x)=\big(\ka\De+\xi_\eps(x)\big)u_\eps(t,x)\\
u_\eps(0,x)=1
\end{cases}
,\qquad t\geq0,~x\in\mbb R^d.
\end{align}
Since $\xi_\eps$ has smooth sample paths, the solution of \eqref{Equation: Smoothed PAM} is classically well defined.
Then, one hopes that the solution of \eqref{Equation: PAM} can be constructed by
taking the limit
\begin{align}
\label{Equation: Smoothed PAM Limit}
u(t,x):=\lim_{\eps\to0}u_\eps(t,x).
\end{align}

On the one hand, if the noise is not too singular---that is, if the condition $0<\om<2$ holds---then the limit
\eqref{Equation: Smoothed PAM Limit} is nontrivial
and gives rise to a meaningful notion of solution
(e.g., \cite[Section 5]{HuHuangNualartTindel}). On the other hand, when $\om\geq2$,
the limit \eqref{Equation: Smoothed PAM Limit} blows up. In some such
cases, the divergence of $u_\eps(t,\cdot)$ can be compensated by subtracting diverging
renormalization constants from the noise,
thus allowing to salvage a meaningful notion of nontrivial solution \cite{ChenDeyaOuyangTindel2,ChenRosen,GubinelliImkellerPerkowski,Hairer,HairerLabbe,HairerLabbe2}.
However, in some cases, the noise is so singular that it is not clear that any meaningful notion of
Stratonovich solution can be salvaged (e.g., the white noise with
$\om=d\geq4$; see \cite[Section 1.1, (PAM)]{Hairer}).

\begin{remark}
The procedure described in the above paragraph works for more general
mollifiers than the Gaussian kernel used in Definition \ref{Definition: Mollifier}.
For simplicity and definiteness, in this paper we use the Gaussian kernel.
\end{remark}

\begin{remark}
There are alternate ways of defining a solution of the singular PAM, such as the
Skorokhod solution (e.g., \cite{ChenSkorokhod,ChenDeyaOuyangTindel} and \cite[Section 3]{HuHuangNualartTindel}).
The present paper deals exclusively with the Stratonovich solution.
\end{remark}

\subsection{Moments of the PAM - Known Results}
\label{Section: Moments of the Singular PAM}

In order to motivate our investigations, we review the known results regarding
the moments of the continuous PAM with Gaussian noise.

\subsubsection{Regular PAM}
\label{Section: State of the art - Reg}

Consider $u_1(t,\cdot)$ (i.e., \eqref{Equation: Smoothed PAM} in the case $\eps=1$),
which is the PAM with smooth Gaussian noise $\xi_1$.
In this case, \cite[Theorem 1 and Section 4.1]{GartnerKonig} states that
\begin{align}
\label{Equation: GK 2nd Order}
\log\langle u_1(t,x)^p\rangle=\left(\frac{p^2\ga_1(0)}{2}\right)t^2-p^{3/2}\chi\,t^{3/2}\big(1+o(1)\big)\qquad \text{as }t\to\infty
\end{align}
for every $x\in\mbb R^d$ and $p\geq1$.
In \eqref{Equation: GK 2nd Order}, we define 
\begin{align}
\label{Equation: GK Variational}
\chi:=\inf_{\substack{f\in H^1(\mbb R^d)\\\|f\|_2=1}}\left(\ka\int_{\mbb R^d}|\nabla f(x)|_2^2\d x+\frac14\iint_{(\mbb R^d)^2}f(x)^2\big((x-y)^\top\Si(x-y)\big)f(y)^2\d x\dd y\right),
\end{align}
where $H^1(\mbb R^d)$ denotes the order 1 Sobolev Hilbert
space, $\|\cdot\|_p$ denotes the Lebesgue $L^p$ norm,
and $-\Si$ is the Hessian matrix of the covariance kernel $\ga_1(x)$ evaluated at $x=0$.
In particular, intermittency in the sense of Definition \ref{Definition: Intermittency}
occurs with scale function $A(t)=t^2$ and Lyapunov exponents $\ell_p=\frac{p^2\ga_1(0)}{2}$
for all $p\geq1$.

The first order in the asymptotic \eqref{Equation: GK 2nd Order} (i.e., the term $\ell_pt^2$)
was proved by Carmona and Molchanov \cite{CarmonaMolchanov} for all $p\in\mbb N$.
Later, G\"artner and K\"onig \cite{GartnerKonig} extended the result to every $p\geq1$ and 
also provided the second-order asymptotic $p^{3/2}\chi\,t^{3/2}$.
As explained in
\cite[Section 0.4.3]{GartnerKonig}, the derivation of this second order---and more specifically the minimizers of $\chi$---provide a precise description of the geometry of intermittent peaks in $u_1(t,\cdot)$ for large $t$.

\begin{remark}
The setting considered in \cite{CarmonaMolchanov,GartnerKonig} is more general
than stationary Gaussian noises with a covariance of the form $\ga_1=\ga*p_1$ for a
$\ga$ satisfying Assumption \ref{Assumption: Noise}. Since the latter is the
setting that is of interest to us in this paper, we do not state the moment asymptotics
of the PAM with regular noise on $\mbb R^d$ in its full generality.
\end{remark}

\subsubsection{Singular PAM - Subcritical Regime}
\label{Section: State of the art - Sub}

If $\xi$ is as in Assumption \ref{Assumption: Noise}
with scaling parameter $0<\om<2$, then we say that it is subcritical.
In this regime, it is known that intermittency in the sense of Definition \ref{Definition: Intermittency} occurs.
More specifically, define the variational constant
\begin{align}
\label{Equation: Sub Variational}
\mbf M:=\sup_{f\in H^1(\mbb R^d),~\|f\|_2=1}\left(\frac12\iint_{(\mbb R^d)^2}f(x)^2\ga(x-y)f(y)^2\d x\dd y-\ka\int_{\mbb R^d}|\nabla f(x)|_2^2\d x\right).
\end{align}
Arguing as in \cite[Lemma A.2 with $\al_0=0$]{ChenHuSongXing},
it can be shown that $\mbf M\in(0,\infty)$.
In the case of Riesz and Fractional noise,
\cite[Theorem 1.1, $\be_0=\rho=0$
and $\al=2$]{ChenHuSongSong} states that
\begin{align}
\label{Equation: Sub Known}
\lim_{t\to\infty}\frac{\log\langle u(t,x)^p\rangle}{t^{(4-\om)/(2-\om)}}=p^{(4-\om)/(2-\om)}\mbf M=\ell_p
\end{align}
for all $x\in\mbb R^d$ and $p\in\{1\}\cup[2,\infty)$.
For white noise with $d=\om=1$, \cite[(2)]{Mansmann}
proves \eqref{Equation: Sub Known} for $p=1$, and \cite[Theorem 6.9 (ii)]{HuHuangNualartTindel}
shows that there exists constants $0<C<C'$ such that
\begin{align}
\label{Equation: Sub Known - WN}
C(pt)^{(4-\om)/(2-\om)}\leq\log\langle u(t,x)^p\rangle\leq C'(pt)^{(4-\om)/(2-\om)}
\end{align}
for $p\in\mbb N$ and $t>0$.

\begin{remark}
While we did not find such a statement in the literature,
we expect that the general methodology developed in \cite{ChenHuSongSong,ChenHuSongXing} could be suitably adapted
to prove \eqref{Equation: Sub Known} for $p\geq2$ in the case of one-dimensional white noise.
That being said, if $p\in(0,1)\cup(1,2)$, then \eqref{Equation: Sub Known} appears to be out of reach of
the currently-known techniques for all three noises considered in Assumption \ref{Assumption: Noise};
see \cite[(6.2)]{ChenHuSongSong}.
\end{remark}

\subsubsection{Singular PAM - Critical Regime}
\label{Section: State of the art - Crt}

If $\xi$ is as in Assumption \ref{Assumption: Noise}
with $\om=2$, then we say that it is critical.
In this regime, it is expected that there exist moment blowup thresholds $t_0(p)\in(0,\infty)$ such that
\begin{align}
\label{Equation: Moment Blowup Threshold}
\langle u(t,x)^p\rangle
\begin{cases}
<\infty&\text{if }t<t_0(p)\\
=\infty&\text{if }t>t_0(p)
\end{cases}
\end{align}
for all $p>0$; see, e.g., \cite[Theorem 1.7]{AllezChouk},
\cite{ChenDeyaOuyangTindel,GuXu}, \cite[(1.26)]{ChenRosen},
and \cite[Corollary 1.2]{Matsuda}.
In particular, if there is intermittency in this regime, then it does not follow from Definition \ref{Definition: Intermittency}; the Lyapunov exponents do not exist when the moments blow up in finite time.

We refer to \cite[Contribution 1 and Remarks 1.1 and 1.2]{ChenDeyaOuyangTindel}
for a related result concerning the Skorokhod PAM when $p\in\{1\}\cup[2,\infty)$,
and to \cite[Page 5]{ChenDeyaOuyangTindel}
for the mention of a conjecture regarding the value of $t_0(p)$ in the Stratonovich setting considered in this paper.
In both cases, it is known/conjectured that $t_0(p)$ is related to the best constant
\begin{multline}
\label{Equation: GNS}
\mc G:=\inf\bigg\{C>0:\iint_{(\mbb R^d)^2}f(x)^2\ga(x-y)f(y)^2\d x\dd y\\\leq C\int_{\mbb R^d}|\nabla f(x)|_2^2\d x\quad\forall f\in H^1(\mbb R^d)\text{ s.t. }\|f\|_2=1\bigg\};
\end{multline}
see \cite[(1.9), (3.21), Notation 3.12, and Theorem 3.14]{ChenDeyaOuyangTindel}.

\begin{remark}
The fact that $\mc G$ is finite for the covariance kernels in
Assumption \ref{Assumption: Noise} when $\om=2$ is established in
\cite[(C.1) for $p=d=2$]{ChenBook} for white noise,
Proposition \ref{Proposition: Riesz Best Constant} below for Riesz noise,
and \cite[Remark 3.11]{ChenDeyaOuyangTindel} for fractional noise.
\end{remark}

\subsubsection{Singular PAM - Supercritical Regime}
\label{Section: State of the art - Sup}

If $\xi$ is as in Assumption \ref{Assumption: Noise} with
$\om>2$, then we say that it is supercritical.
In this regime, it is expected that $\langle u(t,x)^p\rangle=\infty$ for all $p,t>0$;
see, e.g., \cite[Remark 1.7]{GuXu} and \cite[Theorem 2 for $d=3$]{Labbe}.
Thus, intermittency cannot be established by studying Lyapunov exponents in this regime either.

\subsection{Motivational Problems and Outline of Results}
\label{Section: Two Problems}

The results discussed above
raise a number of interesting problems. For instance:

\begin{problems}
\label{Problem: 1}
In the subcritical regime:
\begin{itemize}
\item Can the maximizers of the variational problem $\mbf M$ in \eqref{Equation: Sub Variational}
be argued to describe the geometry of intermittent peaks in the singular PAM, in similar fashion to the
minimizers of $\chi$ in \eqref{Equation: GK 2nd Order} for the smooth PAM (e.g., the heuristic derivation in \cite[Section 0.4.3]{GartnerKonig})?
\item If the answer to the above is affirmative, then how/why does the contribution of the geometry of peaks
(i.e., the variational constant) undergo a transition from the second- to the first-order
asymptotic when going from a regular to a singular noise?
\item Can the exact asymptotic \eqref{Equation: Sub Known} be established for every $p>0$
and every noise in Assumption \ref{Assumption: Noise}
(in particular, including $p\in(0,1)\cup(1,2)$)?
\end{itemize}
\end{problems}

\begin{problems}
\label{Problem: 2}
In the critical and supercritical regimes:
\begin{itemize}
\item What is the geometric significance (if any) of the blowup of the moments
(either in finite time
in the critical regime or for all times in the supercritical regime)?
\item What is the geometric significance (if any) of the
variational constant $\mc G$ in \eqref{Equation: GNS}
that is expected to characterize the moment blowup thresholds $t_0(p)$ in \eqref{Equation: Moment Blowup Threshold}?
\end{itemize}
\end{problems}

Our main results, which we state in full in Section \ref{Section: Main Results},
shed new light on all of these problems. More specifically, our
results fit into three main categories:

Firstly, the main theorems proved in this paper are
{\bf Theorems \ref{Theorem: Sub}, \ref{Theorem: Crt},
and \ref{Theorem: Sup}}. The latter provide moment asymptotics for the PAM with a noise that becomes
increasing singular (i.e., $\xi_\eps$ as $\eps\to0$) in the subcritical, critical, and supercritical regimes respectively.
The setting that we consider in this paper,
which is explained in details in Section \ref{Section: Asymptotically Singular}, allows us to treat the
PAM with smooth and singular noises in a unified
way, as well as interpolate between the two settings.
In particular,
our method of proof is mostly inspired by the methodology developed in
\cite{GartnerKonig} for the smooth setting, and thus
provides the same geometric interpretation of intermittency (see Sections \ref{Section: Moment Heuristics Details} and \ref{Section: Motivation for Maximizers} for the details).

Secondly, we have three corollaries of our main theorems, which improve known
results and provide new insights on the behavior of the moments of the PAM with singular noise:
\begin{itemize}
\item As a corollary of Theorem \ref{Theorem: Sub},
in {\bf Theorem \ref{Theorem: Subcritical Improvement}} we extend
\eqref{Equation: Sub Known} and \eqref{Equation: Sub Known - WN}
by establishing precise moment asymptotics
in the subcritical regime {\it for all $p>0$},
which also allow us to obtain precise tail asymptotics for the solution of the PAM.
Combining this with the fact that
our moment asymptotics in Theorem \ref{Theorem: Sub}
interpolate between the classical and singular asymptotics in
\eqref{Equation: GK 2nd Order} and
\eqref{Equation: Sub Known}/\eqref{Equation: Sub Known - WN}, as well as the
geometric interpretation of our asymptotics in Sections
\ref{Section: Moment Heuristics Details} and \ref{Section: Motivation for Maximizers},
this provides a complete answer to Problems \ref{Problem: 1} (see Section \ref{Section: Main Results Sub}
for more details).
\item As corollaries of Theorems \ref{Theorem: Crt} and \ref{Theorem: Sup},
in {\bf Theorems \ref{Theorem: Intermittency for Finite t Crt} and \ref{Theorem: Intermittency for Finite t Sup}}
we make progress on Problems \ref{Problem: 2} by uncovering that
the moment blowup phenomenon discussed in Sections \ref{Section: State of the art - Crt}
and \ref{Section: State of the art - Sup} is related
to a {\it finite time} intermittency effect, which
we can detect in the moment asymptotics of $u_\eps(t,\cdot)$ as $\eps\to0$ with $t>0$ fixed.
Moreover, we explain how the best constant $\mc G$ in \eqref{Equation: GNS} characterizes
the occurrence of this intermittency phenomenon in the critical regime.
To the best of our knowledge, these are the first results that establish the existence of an intermittency
phenomenon in the PAM with an asymptotically singular noise and a fixed finite time.
\end{itemize}

Thirdly, in {\bf Theorem \ref{Theorem: Fractional Minimizers}},
we prove the existence and study the properties of the maximizers of the variational problem
$\mbf M$ in \eqref{Equation: Sub Variational} in the case of the fractional kernel $\ga(x)=\si^2\prod_{i=1}^d|x_i|^{-\om_i}$; the relevance of these maximizers to intermittency is discussed in Section \ref{Section: Motivation for Maximizers}.
Maximizers of variational problems
of the form $\mbf M$ for general $\ga$ are known in the mathematical physics literature as Hartree ground states
and are important objects in quantum mechanics; see, e.g., \cite{FrohlichLenzmann} and references therein. (In
that interpretation, the kernel $\ga$ represents an interaction potential between distinct particles.)
Thus, the existence and properties of these maximizers have been studied for various choices of
$\ga$, including the white and Riesz noises (see Theorems \ref{Theorem: White Hartree} and \ref{Theorem: Riesz Hartree}).
To the best of our knowledge, the standard results in this theory concern the case where
$\ga$ is symmetric decreasing and has integrable singularities (i.e., the measure of 
level sets $\{x\in\mbb R^d:\ga(x)>t\}$ is finite); see, e.g., \cite[Theorem 2]{FrohlichLenzmann}.
Neither of these properties hold for the fractional covariance kernel. Thus, we believe that Theorem \ref{Theorem: Fractional Minimizers}
may be of independent interest, as it studies the Hartree ground state problem in a situation where
there is less symmetry and integrability than the usual setting.

\subsection{Discussion}
\label{Section: Discussion}

\subsubsection{Extensions}

Many of the moment asymptotics stated in
Sections \ref{Section: State of the art - Sub}--\ref{Section: State of the art - Sup}
admit a generalization if
\begin{itemize}
\item $\xi$ is replaced by a possibly time-dependent noise of the form
\[\langle\xi(s,x)\xi(t,y)\rangle=|s-t|^{\om_0}\ga(x-y)\]
for some $\om_0\in[0,1)$; and/or
\item the term $\xi(x)\cdot u(t,x)$ or $\xi(t,x)\cdot u(t,x)$ in the definition
of the PAM is interpreted as a Wick product, which leads to the Skorokhod solution; and/or
\item the Laplacian operator $\De$ is replaced by 
a fractional Laplacian of the form $-(-\De)^{s}$ for some $0<s<1$.
\end{itemize}
See, for instance, \cite{ChenEM3,ChenEM2,ChenEM1,ChenDeyaOuyangTindel,ChenHuSongSong,ChenRosen}. It is thus natural to ask the following:

\begin{problems}
\label{Problems: Open 1}
Is it possible to prove analogues of Theorems \ref{Theorem: Sub},
\ref{Theorem: Crt}, and \ref{Theorem: Sup}
for the PAM with a time-independent noise and/or Skorokhod solution and/or fractional Laplacian?
\end{problems}

Among other things, an affirmative answer to these problems could improve the known precise moment asymptotics in \cite[Theorem 1.1]{ChenHuSongSong} to all $p>0$ in the subcritical case,
and it could show that the moment blowup phenomenon proved in \cite[Contribution 1]{ChenDeyaOuyangTindel} for the
critical Skorokhod/time-dependent case is also related to an intermittency effect
that occurs in the mollified PAM when $\eps\to0$ for fixed times.

However, many specifics of the techniques employed in this paper cannot be directly extended to the
time-independent noise, Skorokhod, or fractional Laplacian settings. Most notably, our frequent
reliance on the spectral expansion \eqref{Equation: Spectral Expansion}, which is exclusive to the
Stratonovich and time-independent settings. Thus, we expect that solving
Problems \ref{Problems: Open 1} would require a number of new ideas (such as replacing spectral
expansions with the hypercontractivity trick developed in \cite{Le}), which we leave open for future works.

\subsubsection{Intermittency in the Critical and Supercritical Regimes}

While Theorems \ref{Theorem: Intermittency for Finite t Crt} and
\ref{Theorem: Intermittency for Finite t Sup} shed new light on
the relationship between intermittency and the moment blowup phenomenon
in the critical and supercritical regimes, they do not appear to say
anything definitive about the occurrence (or not) of intermittency
in the usual large-time setting.
At most, these results suggest that {\it if} intemittency occurs
in the critical or supercritical PAM at large times, then it is likely explained by
a different mechanism than the PAM with smooth or subcritical noise.

More specifically, the phase transition in
Theorem \ref{Theorem: Sub}
shows that the large-time intermittency in the subcritical PAM comes
from the same mechanism as the PAM with smooth noise:
As illustrated in the transition from Figure
\ref{Figure: 2 Orders} to Figure \ref{Figure: Coalesce 2},
as we make $\eps$ increasingly small compared to $t$, the geometry of peaks in
the noise and the PAM at large times undergoes a transition from the classical asymptotics,
and  eventually stabilizes to a size and shape independent of $\eps$.

In contrast, Theorems \ref{Theorem: Intermittency for Finite t Crt} and \ref{Theorem: Intermittency for Finite t Sup}
show that the intermittency in the PAM with smooth noise
cannot transform into an intermittency effect that occurs at large times
in the critical or supercritical PAM. Indeed, if we define the renormalized subcritical/critical PAM as
\[u(t,x)=\lim_{\eps\to0}u_\eps(t,x)\mr e^{-tc_\eps}\]
(for an appropriate choice of diverging renormalization constants $c_\eps$), then the tall and narrow
peaks that generate the moment asymptotics in Theorem \ref{Theorem: Intermittency for Finite t Crt} and \ref{Theorem: Intermittency for Finite t Sup} disappear once we take the $\eps\to0$ limit, and thus have no
contribution in the geometry of $u(t,x)$ as $t\to\infty$.

As a final remark, we note the work of K\"onig, Perkowski, and Van Zuijlen \cite{KonigPerkowskiZuijlen},
who studied the almost-sure asymptotics of the critical Stratonovich PAM in the case of white noise.
In particular, \cite[Theorem 1.1 (b)]{KonigPerkowskiZuijlen} states that, from the point of view of
almost-sure logarithmic asymptotics of the PAM, an intermittency effect cannot
be observed (at least on boxes near the origin of side-length $\sim t^a$ for $0<a<1$). This is in sharp contrast to the case of smooth noise, as explained in
\cite[Section 1.6]{GartnerKonigMolchanov}. Therefore, it appears that the occurrence
of an intermittency phenomenon at large times (as well as its exact nature if it does occur) in the critical and supercritical
PAM remains open.

\subsection{Organization of this Paper}
\label{Section: Organization}

The remainder of this paper is organized as follows:
In Section \ref{Section: Main Results}, we explain the
approach developed in this paper
to tackle Problems \ref{Problem: 1} and \ref{Problem: 2},
and then we state our main results.
In Section \ref{Section: Moment Heuristics},
we introduce some notations, present a heuristic derivation
of our main theorems, and discuss the geometric significance of the
latter.
Finally, in Sections \ref{Section: Moments Part 1}--\ref{Section: Variational},
we prove our main results.

\section{Main Results}
\label{Section: Main Results}

Throughout this section (and for the remainder of this paper), we assume that $\xi$
is a singular noise that satisfies Assumption \ref{Assumption: Noise}.
We also recall that $\langle\cdot\rangle$ denotes the expectation with respect
to $\xi$, that $\om$ is the scaling parameter introduced in Remark \ref{Remark: Scaling Parameter},
and that $\xi_\eps$ and $u_\eps(t,x)$ ($\eps>0$) are the mollified noise and
PAM introduced in Definition \ref{Definition: Mollifier} and \eqref{Equation: Smoothed PAM}.

This section is organized as follows. In Section \ref{Section: Asymptotically Singular},
we discuss the general philosophy behind the approach used in this paper.
In Sections \ref{Section: Main Results Sub}--\ref{Section: Main Results Sup},
we provide precise statements of the results outlined in Section \ref{Section: Two Problems}.
Then, in Section \ref{Section: Roadmap}, we provide an index of where the proofs
of each of our main results are located.

\subsection{Our Approach: Asymptotically Singular Noise}
\label{Section: Asymptotically Singular}

%
%
%
Instead of studying $u_1(t,x)$ as $t\to\infty$ (i.e., \eqref{Equation: GK 2nd Order})
or studying $u_\eps(t,x)$ by first sending $\eps\to0$ and then $t\to0$ (i.e.,
Sections \ref{Section: State of the art - Sub}--\ref{Section: State of the art - Sup}),
we consider the setting where both
$\eps$ and $t$ can vary simultaneously:

\begin{definition}
We introduce the master parameter
\[\msf m\in[0,\infty).\]
We assume that $\msf e=\msf e(\msf m)$ and $\msf t=\msf t(\msf m)$
are positive functions of this parameter
such that
\begin{align}
\label{Equation: e bounded and t away from zero}
\sup_{\msf m\geq0}\msf e\leq 1\qquad\text{and}\qquad\inf_{\msf m\geq0}\msf t>0.
\end{align}
The function $\msf e$ represents the dependence of $\eps$ on the master parameter $\msf m$,
and $\msf t$ represents the dependence of $t$ on $\msf m$.
\end{definition}

\begin{remark}
As we have done in \eqref{Equation: e bounded and t away from zero},
in order to improve readability, throughout this paper we mostly keep the
dependence of $\msf e$ and $\msf t$ (and various other functions) on the master parameter
implicit. For instance, if we state an asymptotic of the form
\[\log\langle u_\msf e(\msf t,x)^p\rangle=F(\msf e,\msf t)+o(1)\qquad\text{as }\msf m\to\infty\]
for some functional $F$, then it should be understood as
\[\log\big\langle u_{\msf e(\msf m)}\big(\msf t(\msf m),x\big)^p\big\rangle=F\big(\msf e(\msf m),\msf t(\msf m)\big)+o(1)\qquad\text{as }\msf m\to\infty.\]
Limits of the form $\lim_{\msf m\to\infty}F(\msf e,\msf t)=l$ should be interpreted in the same way.
\end{remark}

We are mainly interested in taking $\msf e\to0$ as $\msf m\to\infty$.
In doing so, our hope is twofold:
\begin{enumerate}
\item Since $u_\eps(t,\cdot)$ has a smooth noise for every fixed $\eps>0$,
some of the techniques used for the case
$\eps=1$ in \cite{GartnerKonig} that are not applicable to the singular case remain available.
\item If $\eps$ is small, then $u_\eps(t,\cdot)$'s geometry is similar to that of
$u(t,\cdot)$ (possibly up to a renormalization).
Thus, if $\msf e\to0$ at a fast enough
rate as $\msf m\to\infty$, then the asymptotics
\[u_{\msf e}(\msf t,\cdot)=u_{\msf e(\msf m)}\big(\msf t(\msf m),\cdot\big)\]
might shed new light on the intermittency of the PAM with singular noise.
\end{enumerate}

Since intermittency is a phenomenon that typically occurs at large times,
we will in many cases consider the large-$\msf m$ asymptotics of the moments $\langle u_\msf e(\msf t,x)^p\rangle$
under the assumption that both $\msf e\to0$ and $\msf t\to\infty$ as $\msf m\to\infty$.
However, as we explain in the coming sections, one of the main insights
of this paper is that it is also interesting to consider the case where $\msf t$ remains bounded
(see Theorems \ref{Theorem: Intermittency for Finite t Crt} and \ref{Theorem: Intermittency for Finite t Sup}).

\begin{remark}
A similar approach was implemented in \cite{GLPT}, wherein the main focus was to
establish a phase transition in the
almost-sure asymptotics of the PAM with asymptotically singular noise
(see Theorems 1.7 and 1.11 therein).
The moments in the case of asymptotically singular white noise
were studied in \cite[Theorem 1.17]{GLPT}. However, the latter is
weaker and less general than the results in this paper,
it is proved using different techniques,
and it does not provide
any insight on Problems \ref{Problem: 1} and \ref{Problem: 2} (which are the main focus of this paper).
\end{remark}

We now proceed to an exposition of our results.

\subsection{Main Results Part 1 - Subcritical Regime}
\label{Section: Main Results Sub}

Consider the subcritical regime where $0<\om<2$, which we henceforth abbreviate as Sub.

\begin{assumption}
\label{Assumption: Sub}
In the Sub regime, in addition to \eqref{Equation: e bounded and t away from zero},
we always assume that
\[\lim_{\msf m\to\infty}\msf t=\infty.\]
We consider three cases of the Sub regime, depending on the behavior of $\msf e$ relative
to $\msf t$:
\begin{enumerate}[\qquad\text{$\bullet$ Case Sub-}1:]
\item $\displaystyle\lim_{\msf m\to\infty}\msf e\,\msf t^{1/(2-\om)}=\infty$.
\item $\displaystyle\lim_{\msf m\to\infty}\msf e\,\msf t^{1/(2-\om)}=\mf c$ for some constant $\mf c\in(0,\infty)$.
\item $\displaystyle\lim_{\msf m\to\infty}\msf e\,\msf t^{1/(2-\om)}=0$.
\end{enumerate}
\end{assumption}

Our first main result is as follows:

\begin{theorem}
\label{Theorem: Sub}
Suppose that Assumption \ref{Assumption: Sub} holds.
Let $p>0$ and $x\in\mbb R^d$ be fixed.\begin{itemize}
\item In case Sub-1, as $\msf m\to\infty$, one has
\begin{align}
\label{Equation: Sub-1}
\log\langle u_\msf e(\msf t,x)^p\rangle=\left(\frac{p^2\ga_1(0)}{2}\right)\msf e^{-\om}\msf t^2-p^{3/2}\chi\,\msf e^{-(2+\om)/2}\msf t^{3/2}\big(1+o(1)\big),
\end{align}
where we recall that $\chi$ is defined in \eqref{Equation: GK Variational}.
\item In case Sub-2,
\begin{align}
\label{Equation: Sub-2}
\lim_{\msf m\to\infty}\frac{\log\langle u_\msf e(\msf t,x)^p\rangle}{\msf t^{(4-\om)/(2-\om)}}=p^{(4-\om)/(2-\om)}\mbf M_{\mf c,p},
\end{align}
where we define the variational constant
\begin{align}
\label{Equation: Sub-2 Variational}
\mbf M_{\mf c,p}:=\sup_{\substack{f\in H^1(\mbb R^d)\\\|f\|_2=1}}\left(\frac12\iint_{(\mbb R^d)^2}f(x)^2\ga_{p^{1/(2-\om)}\mf c}(x-y)f(y)^2\d x\dd y-\ka\int_{\mbb R^d}|\nabla f(x)|_2^2\d x\right).
\end{align}
\item In case Sub-3,
\begin{align}
\label{Equation: Sub-3}
\lim_{\msf m\to\infty}\frac{\log\langle u_\msf e(\msf t,x)^p\rangle}{\msf t^{(4-\om)/(2-\om)}}=p^{(4-\om)/(2-\om)}\mbf M,
\end{align}
where we recall that $\mbf M$ is defined in \eqref{Equation: Sub Variational}.
\end{itemize}
\end{theorem}

\subsubsection{Interpretation of Theorem \ref{Theorem: Sub}}
\label{Section: Interpretation of Sub}

In case Sub-1, $\msf t$ blows up quickly relative to $\msf e$'s size.
Thus, it is natural to expect that $u_\msf e(\msf t,\cdot)$ should behave similarly to $u_1(\msf t,\cdot)$ as $\msf m\to\infty$.
The asymptotic \eqref{Equation: Sub-1} confirms this:
Apart from the additional factors of $\msf e^{-\om}$ and $\msf e^{-(2+\om)/2}$ in the first- and second-order terms respectively,
this is the same as the classical result \eqref{Equation: GK 2nd Order}.
(In fact, in the case where $\msf e=1$ is constant, then we
recover exactly \eqref{Equation: GK 2nd Order}, with the addition
that we also prove moment asymptotics for $0<p<1$.)

Next, by carefully examining \eqref{Equation: Sub-1}, 
it is to be expected that a transition occurs when $\msf e$ is so small
that the first- and second-order terms are the same size. This occurs
when $\msf e\approx\msf t^{-1/(2-\om)}$,
which implies that $\msf e^{-\om}\msf t^2$ and $\msf e^{-(2+\om)/2}\msf t^{3/2}$
are both on the order of $\msf t^{(4-\om)/(2-\om)}$.
In cases Sub-2 and Sub-3, we show that this transition does indeed
occur, and eventually gives rise to the subcritical asymptotics
stated earlier in \eqref{Equation: Sub Known}.

Referring back to Problems \ref{Problem: 1},
we note that Theorem \ref{Theorem: Sub} provides an
affirmative answer to the first two items raised therein:
Since our proof of Theorem \ref{Theorem: Sub}
mostly follows the general methodology developed in \cite{GartnerKonig}, we can infer that
the variational problems that appear in \eqref{Equation: Sub-2} and \eqref{Equation: Sub-3}
(and their maximizers)
have the same geometric interpretation as the constant $\chi$
that appears in the second-order term in \eqref{Equation: GK 2nd Order}.
We refer to Sections \ref{Section: Moment Heuristics Details} and
\ref{Section: Motivation for Maximizers} for a detailed geometric
interpretation of the following from the point of view of intermittency:
\begin{enumerate}
\item The maximizers of the variational problems in \eqref{Equation: Sub-2} and \eqref{Equation: Sub-3}; and
\item the coalescence phenomenon whereby the first- and second-order terms
in \eqref{Equation: Sub-1} merge into a single term when $\msf e$ is small enough
compared to $\msf t$.
\end{enumerate}

\subsubsection{Precise Asymptotics for all $p>0$}

Regarding the third item raised in Problems \ref{Problem: 1},
unlike the results stated in Section \ref{Section: State of the art - Sub}, the asymptotic \eqref{Equation: Sub-3}
holds for all $p>0$. Thus, as a corollary of \eqref{Equation: Sub-3}, we obtain
the following improvement of \cite[Theorem 1.1 with $\be_0=\rho=0$
and $\al=2$]{ChenHuSongSong}, \cite[Theorem 6.9 (ii)]{HuHuangNualartTindel}, and \cite[(2)]{Mansmann},
extending the precise asymptotics to all $p>0$:

\begin{theorem}
\label{Theorem: Subcritical Improvement}
Let $\xi$ be one of the noises in Assumption \ref{Assumption: Noise},
assuming that $0<\om<2$. For every $p>0$ and $x\in\mbb R^d$, it holds that
\[\lim_{t\to\infty}\frac{\log\langle u(t,x)^p\rangle}{t^{(4-\om)/(2-\om)}}=p^{(4-\om)/(2-\om)}\mbf M.\]
In particular, for every $\theta>0$ and $x\in\mbb R^d$, we have the tail asymptotic
\begin{multline*}
\lim_{t\to\infty}t^{-(4-\om)/(2-\om)}\log\mbf P\left[u(t,x)\geq\exp\Big(\theta t^{(4-\om)/(2-\om)}\Big)\right]\\
=-\sup_{p>0}\left(\theta p-p^{(4-\om)/(2-\om)}\mbf M\right)
=-2\theta^{(4-\om)/2}(4-\om)^{-(4-\om)/2}\left(\frac{\mbf M}{2-\om}\right)^{-(2-\om)/2}.
\end{multline*}
\end{theorem}

The moment asymptotics in Theorem \ref{Theorem: Subcritical Improvement}
are an immediate consequence of \eqref{Equation: Sub-3} and
\begin{align}
\label{Equation: Subcritical Convergence}
\lim_{\eps\to0}\langle u_\eps(t,x)^p\rangle=\langle u(t,x)^p\rangle\qquad\text{for every fixed }p,t>0
\end{align}
in the subcritical regime for the noises considered in this paper
(see, e.g., \cite{HuHuangNualartTindel}):
Indeed, \eqref{Equation: Subcritical Convergence} means that in case Sub-3, for each $\msf m$ we can choose
$\msf e(\msf m)$ small enough relative to $\msf t(\msf m)$ so that
\[\frac{\log\langle u_{\msf e}(\msf t,x)^p\rangle}{\msf t^{(4-\om)/(2-\om)}}=
\frac{\log\langle u(\msf t,x)^p\rangle}{{\msf t^{(4-\om)/(2-\om)}}}+o(1)\qquad\text{as }\msf m\to\infty.\]
Once the moment asymptotics are established, the tail asymptotics follow from standard large deviations
theory (i.e., the G\"artner-Ellis theorem; see, e.g., item (i) on page 4 of
\cite{ChenHuSongSong}).

\subsubsection{Geometry of Maximizers}
\label{Section: Sub Maximizers}

The significance of the variational constants $\mbf M_{\mf c,p}$ and $\mbf M$ in describing
the geometry of intermittency (as explained in Sections \ref{Section: Moment Heuristics Details}  and \ref{Section: Motivation for Maximizers})
motivates studying the existence and properties of its maximizers.

For the white noise, $0<\om<2$ implies that $d=1$. In this case, we have the following
well-known result (e.g., \cite[Lemma 3.1]{FrankGeisinger} and \cite[(3)]{Mansmann}):

\begin{theorem}[\cite{FrankGeisinger,Mansmann}]
\label{Theorem: White Hartree}
Suppose that $d=1$ and $\ga=\si^2\de_0$.
The set of maximizers of $\mbf M$ consists of all functions $f_\star$ of the form
\[f_\star(x)=\pm\frac{\si}{2^{3/2}\ka^{1/2}}\mr{sech}\left(\frac{\si^2}{4\ka}(x-z)\right),\]
where $z\in\mbb R$ can be any real number.
\end{theorem}

\begin{remark}
The importance of the function $\mr{sech}=1/\mr{cosh}$ in the description of
the geometry of intermittency of the PAM with one-dimensional white noise
was also pointed out in \cite{DumazLabbe}, in the context of the localization of the
associated Schr\"odinger operator on a large interval.
\end{remark}

For the Riesz noise, the minimizers are not explicitly known, but their existence and
basic properties are nevertheless well-understood and classical
(e.g., \cite[Theorem 2]{FrohlichLenzmann}):

\begin{theorem}[\cite{FrohlichLenzmann}]
\label{Theorem: Riesz Hartree}
Suppose that $\ga(x)=\si^2|x|^{-\om}$ for some $0<\om<\min\{2,d\}$.
There exists maximizers of $\mbf M$. Suppose that $f_\star:\mbb R^d\to\mbb R$ is such a maximizer.
\begin{enumerate}
\item $f_\star$ is either positive or negative.
\item $f_\star$ is smooth and decays exponentially at infinity; that is,
there exists some constants $a,b>0$ such that $|f_\star(x)|\leq a\mr e^{-b|x|}$ for all $x\in\mbb R^d$.
\item $|f_\star|$ is symmetric decreasing about a point; that is, there
exists some $z\in\mbb R^d$ and a nonincreasing function $\rho:[0,\infty)\to[0,\infty)$
such that $|f_\star(x-z)|=\rho(|x|)$ for all $x\in\mbb R^d$.
\end{enumerate}
\end{theorem}

\begin{remark}
\label{Remark: Riesz Unique}
To the best of our knowledge, the uniqueness of the maximizer $f_\star$
in Theorem \ref{Theorem: Riesz Hartree}
(up to a $\pm$ sign and/or a shift $f_\star(\cdot-z)$)
is only known when $d=3$ and $\om=1$ \cite{Lieb}. Otherwise, the uniqueness
appears to be a hard open problem (e.g. \cite[Problem 3]{FrohlichLenzmann}).
\end{remark}

Next, the existence
and basic properties of the maximizers of $\mbf M$ in the case of the fractional
noise do not appear to be known. One of the contributions of this paper is to prove the existence and study the geometry of such maximizers:

\begin{theorem}
\label{Theorem: Fractional Minimizers}
Let $\ga(x)=\si^2\prod_{i=1}^d|x_i|^{-\om_i}$, where $\om_i\in(0,1)$ and $\om=\sum_{i=1}^d\om_i\in(0,2)$.
There exists maximizers of $\mbf M$. Suppose that $f_\star:\mbb R^d\to\mbb R$ is such a maximizer.
\begin{enumerate}
\item $f_\star$ is either positive or negative.
\item $f_\star$ is smooth and decays exponentially at infinity.
\item $|f_\star|$ is symmetric decreasing along every axis about a point; that is, there exists some $z\in\mbb R^d$
such that for every fixed $1\leq i\leq d$ and $x_j\in\mbb R$ (for all $j\neq i$),
there exists a nonincreasing function $\rho:[0,\infty)\to[0,\infty)$ such that
\[\big |f_\star\big((x_1,\ldots,x_{i-1},r,x_{i+1},\ldots,x_d)-z\big)\big |=\rho(|r|)\qquad\text{for every }r\in\mbb R.\]
\end{enumerate}
\end{theorem}

\begin{remark}
The fractional covariance kernel $\ga(x)=\si^2\prod_{i=1}^d|x_i|^{-\om_i}$ is symmetric decreasing
along every axis about the point $z=0$. Thus, much like the white and Riesz noises,
the maximizers of $\mbf M$ for fractional noise inherit some of the same symmetries (up to a translation) as the
associated covariance kernel.
\end{remark}

\begin{remark}
We do not attempt to settle the question of the
uniqueness (up to change of sign or translations) of the maximizers in Theorem \ref{Theorem: Fractional Minimizers},
which we also expect to be a hard problem (see Remark \ref{Remark: Riesz Unique}).
\end{remark}

Finally, in Section \ref{Section: Motivation for Maximizers} we showcase
how the maximizers of $\mbf M_{\mf c,p}$ for $\mf c\in(0,\infty)$ describe
the transition of the geometry of intermittency from the smooth
asymptotics in \eqref{Equation: Sub-1} to the singular ones in \eqref{Equation: Sub-3}.
This can be made more formal with the following statement:

\begin{proposition}
\label{Proposition: Convergence of Minimizers}
Let $0<\om<2$.
For every $p,\mf c>0$,
the constant $\mbf M_{\mf c,p}$ is finite and positive,
and it has maximizers.
For every fixed $p>0$,
we have the limit
\begin{align}
\label{Equation: Convergence to M}
\lim_{\mf c\to0}\mbf M_{\mf c,p}=\mbf M.
\end{align}
Moreover, if
$(f_\star^{(\mf c)})_{\mf c>0}$ is a sequence such that $f_\star^{(\mf c)}$ is a maximizer of $\mbf M_{\mf c,p}$
for all $\mf c>0$,
then every vanishing sequence of $\mf c$'s has a
subsequence $(\mf c_n)_{n\in\mbb N}$ along which
\begin{align}
\label{Equation: Convergence to fstar}
\lim_{n\to\infty}\|f_\star^{(\mf c_n)}(\cdot-z_n)-f_\star\|_{H^1(\mbb R^d)}=0
\end{align}
for some $z_n\in\mbb R^d$ and $f_\star$ that is a maximizer of $\mbf M$.
\end{proposition}

\begin{remark}
The limit $f_\star$ in \eqref{Equation: Convergence to fstar} will
in general depend on the subsequence $\mf c_n$.
\end{remark}

\subsection{Main Results Part 2 - Critical Regime}
\label{Section: Main Results Crt}

Consider the critical regime $\om=2$, which we henceforth abbreviate as Crt.

\begin{assumption}
\label{Assumption: Crt}
For Crt, in addition to \eqref{Equation: e bounded and t away from zero},
we assume that one of the following two cases holds:
\begin{enumerate}[\qquad\text{$\bullet$ Case Crt-}1:]
\item $\displaystyle\lim_{\msf m\to\infty}\msf t=\infty$.
\item $\displaystyle\lim_{\msf m\to\infty}\msf e=0$ and $\displaystyle\lim_{\msf m\to\infty}\msf t=t$
for some fixed time $t\in(0,\infty)$.
\end{enumerate}
In particular, in this regime $\msf t$ does not necessarily blow up as
$\msf m\to\infty$.
\end{assumption}

Our second main result is as follows:

\begin{theorem}
\label{Theorem: Crt}
Suppose that Assumption \ref{Assumption: Crt} holds.
Let $p>0$ and $x\in\mbb R^d$ be fixed.
\begin{itemize}
\item In case Crt-1, as $\msf m\to\infty$, one has
\begin{align}
\label{Equation: Crt-1}
\log\langle u_\msf e(\msf t,x)^p\rangle=\left(\frac{p^2\ga_1(0)}{2}\right)\msf e^{-2}\msf t^2-p^{3/2}\chi\,\msf e^{-2}\msf t^{3/2}\big(1+o(1)\big),
\end{align}
where we recall that $\chi$ is defined in \eqref{Equation: GK Variational}.
\item In case Crt-2,
\begin{align}
\label{Equation: Crt-2}
\lim_{\msf m\to\infty}\frac{\log\langle u_\msf e(\msf t,x)^p\rangle}{\msf e^{-2}}=pt\,\mbf M^{\mr{crt}}_{t,p},
\end{align}
where we define the variational constant
\begin{align}
\label{Equation: Crt-2 Variational}
\mbf M^{\mr{crt}}_{t,p}:=\sup_{\substack{f\in H^1(\mbb R^d)\\\|f\|_2=1}}\left(\frac{p t}2\iint_{(\mbb R^d)^2}f(x)^2\ga_1(x-y)f(y)^2\d x\dd y-\ka\int_{\mbb R^d}|\nabla f(x)|_2^2\d x\right).
\end{align}
\end{itemize}
\end{theorem}

\subsubsection{Interpretation of Theorem \ref{Theorem: Crt}}
\label{Section: Interpretation of Ctr}

In \eqref{Equation: Crt-1},
the first order term $\msf e^{-2}\msf t^2$ always dominates the second order term
$\msf e^{-2}\msf t^{3/2}$ when $\msf t\to\infty$. Thus, no matter how quickly $\msf e$ vanishes
compared to $\msf t$'s size, we do not observe
a phase transition in the large-time asymptotics similar to Theorem \ref{Theorem: Sub} in the Crt regime.

Instead, in order for the
the first- and second-order terms in \eqref{Equation: Crt-1} to be of the same
size and coalesce, $\msf t$ must remain bounded, in which case both
terms grow on the order of $\msf e^{-2}$.
Case Crt-2 describes one aspect of this situation---when $\msf t$
converges to a constant time $t$---which yields the asymptotic
\eqref{Equation: Crt-2}. As it turns out, the latter result provides
unexpected insights into Problems \ref{Problem: 2}. In order to understand
why that is, we state some properties of the variational constant therein:

\begin{proposition}
\label{Proposition: Finite Time Intermittency}
The constant $\mbf M^{\mr{crt}}_{t,p}$ is finite for every $p>0$ and $t>0$.
Moreover, letting $\mc G$ be defined as in \eqref{Equation: GNS}, we have that
\begin{enumerate}
\item if $p<\frac{2\ka}{t\mc G}$ (equivalently $t<\frac{2\ka}{p\mc G}$),
then $\mbf M^{\mr{crt}}_{t,p}=0$; and
\item if $p>\frac{2\ka}{t\mc G}$ (equivalently $t>\frac{2\ka}{p\mc G}$) then $\mbf M^{\mr{crt}}_{t,p}$ is
positive and strictly increasing in $p$.
\end{enumerate}
\end{proposition}

In particular, \eqref{Equation: Crt-2} and Proposition \ref{Proposition: Finite Time Intermittency} (2) imply the following surprising result:

\begin{theorem}
\label{Theorem: Intermittency for Finite t Crt}
Let $\xi$ be one of the noises in Assumption \ref{Assumption: Noise}
in the critical regime $\om=2$.
For every fixed $t>0$, the random field
$x\mapsto u_\eps(t,x)$ is intermittent
as $\eps\to0$ with rate function $A(\eps)=\eps^{-2}$, in the sense that the Lyapunov exponents
\[\ell_p=pt\mbf M^{\mr{crt}}_{t,p}=\lim_{\eps\to0}\frac{\log\langle u_\eps(t,x)^p\rangle}{\eps^{-2}}\]
satisfy Definition \ref{Definition: Intermittency} for $p\in(\frac{2\ka}{t\mc G},\infty)$.
\end{theorem}

Given the first main result of \cite{ChenDeyaOuyangTindel} (see Contribution 1 therein)
regarding the moment blowup threshold for the critical Skorokhod PAM,
and the blowup threshold \cite[Corollary 1.2]{Matsuda} for the first moment of the Stratonovich solution
with white noise on $\mbb R^2$, it is natural
to conjecture the following:

\begin{conjecture}
$t_0(p)$ in \eqref{Equation: Moment Blowup Threshold}
is equal to $\frac{2\ka}{p\mc G}$ for all $p>0$.
\end{conjecture}

If true, then combining this conjecture with
Theorem \ref{Theorem: Intermittency for Finite t Crt} answers parts of Problems \ref{Problem: 2} in
the following surprising fashion:
In the critical Stratonovich PAM,
for every $p>0$,
the moment blowup threshold $t_0(p)$
corresponds to the smallest fixed time $t>0$
for which we observe an intermittency phenomenon as $\eps\to0$
in the moments $\langle u_\eps(t,x)^q\rangle$ of order $q>p$.

\subsection{Main Results Part 3 - Supercritical Regime}
\label{Section: Main Results Sup}

Consider the critical regime $\om>2$,
which we henceforth abbreviate as Sup.

\begin{assumption}
\label{Assumption: Sup}
In the Sup regime, in addition to \eqref{Equation: e bounded and t away from zero},
we assume that at least one of the following holds:
\[\lim_{\msf m\to\infty}\msf e=0\qquad\text{or}\qquad\lim_{\msf m\to\infty}\msf t=\infty.\]
\end{assumption}

\begin{theorem}
\label{Theorem: Sup}
Suppose that Assumption \ref{Assumption: Sup} holds.
Let $p>0$ and $x\in\mbb R^d$ be fixed.
As $\msf m\to\infty$, one has
\begin{align}
\label{Equation: Sup}
\log\langle u_\msf e(\msf t,x)^p\rangle=\left(\frac{p^2\ga_1(0)}{2}\right)\msf e^{-\om}\msf t^2-p^{3/2}\chi\,\msf e^{-(2+\om)/2}\msf t^{3/2}\big(1+o(1)\big),
\end{align}
where we recall that $\chi$ is defined in \eqref{Equation: GK Variational}.
\end{theorem}

\subsubsection{Interpretation of Theorem \ref{Theorem: Sup}}

Our proposed interpretation of Theorem \ref{Theorem: Sup} is a continuation of
Theorems \ref{Theorem: Crt} and \ref{Theorem: Intermittency for Finite t Crt},
and has similar implications for Problems \ref{Problem: 2}. More specifically,
Theorem \ref{Theorem: Sup} implies the following:

\begin{theorem}
\label{Theorem: Intermittency for Finite t Sup}
Let $\xi$ be one of the noises in Assumption \ref{Assumption: Noise}
in the supercritical regime $\om>2$.
For every fixed $t>0$, the random field
$x\mapsto u_\eps(t,x)$ is intermittent as $\eps\to0$ with scale function $A(\eps)=\eps^{-\om}$
and Lyapunov exponents $\ell_p=\frac{p^2t^2\ga_1(0)}{2}$.
\end{theorem}

Moreover,
the second-order term in \eqref{Equation: Sup} suggests that the local geometry of
$u_\eps(t,\cdot)$'s intermittent peaks as $\eps\to0$
is determined by the same mechanism as
the PAM with smooth Gaussian noise. This reinforces the idea that the moment blowup
phenomenon in the singular PAM can be explained by the occurrence of intermittency
in smooth approximations $u_\eps(t,\cdot)$ for fixed $t$ and small $\eps$
(with the difference that, unlike in the Crt regime, in the Sup regime there is no limitation on
the moments for which intermittency holds).

\begin{remark}
Looking at the specific form of \eqref{Equation: Sup}, it is natural to suspect that
a phase transition in the asymptotics occurs when $\msf t$ goes to
zero at the same rate or faster than $\msf e^{\om-2}$. We do
not pursue this direction in this paper, as it is not clear to us that small-time asymptotics have an interesting interpretation
in the PAM.
\end{remark}

\subsection{Proof Index}
\label{Section: Roadmap}

Theorems \ref{Theorem: Sub}, \ref{Theorem: Crt},
and \ref{Theorem: Sup} are proved in Sections \ref{Section: Moment Heuristics}--\ref{Section: Moments Part 3}.
More specifically:
\begin{itemize}
\item in Section \ref{Section: Moment Heuristics} we provide some notations, a heuristic derivation,
and a geometric interpretation of our moment asymptotics; and
\item in Sections \ref{Section: Moments Part 1}--\ref{Section: Moments Part 3}, we provide a rigorous proof
of every step of the heuristic derivation. 
\end{itemize}
Proposition \ref{Proposition: Finite Time Intermittency} is proved in Section \ref{Section: Finite Time Intermittency},
and Theorem \ref{Theorem: Fractional Minimizers} and Proposition \ref{Proposition: Convergence of Minimizers}
are proved together in Sections \ref{Section: Variational Outline}--\ref{Section: Geometry of fcoord}.
Finally, as Theorems \ref{Theorem: Subcritical Improvement}, \ref{Theorem: Intermittency for Finite t Crt}, and \ref{Theorem: Intermittency for Finite t Sup}
are immediate consequences of Theorems \ref{Theorem: Sub}, \ref{Theorem: Crt},
and \ref{Theorem: Sup} and Proposition \ref{Proposition: Finite Time Intermittency}, they are not further
discussed in the paper.

\section{Notations and Heuristics for Moment Asymptotics}
\label{Section: Moment Heuristics}

In this section, our main purpose is to provide a heuristic derivation
and geometric interpretation of
Theorems \ref{Theorem: Sub}, \ref{Theorem: Crt},
and \ref{Theorem: Sup}. We take this opportunity to
set up some notations that will be used throughout the paper.
Among other things, this allows us to formulate a
statement that combines Theorems \ref{Theorem: Sub}, \ref{Theorem: Crt},
and \ref{Theorem: Sup} into a single result, namely, Theorem \ref{Theorem: All Moments}.
We then formally prove Theorem \ref{Theorem: All Moments} in Sections
\ref{Section: Moments Part 1}--\ref{Section: Moments Part 3}.

\begin{remark}
The main source of inspiration for the method
developed in this paper is the work of G\"artner and K\"onig \cite{GartnerKonig}.
Given that some technical
estimates used in this paper are cited directly from \cite{GartnerKonig},
most of our notations are meant to mirror that of \cite{GartnerKonig}
for sake of convenience. However, the difference between
the asymptotically singular setting considered herein and that of \cite{GartnerKonig} means that some notations
have to be suitably modified. Whenever there are significant differences
between the notations, we address it in a remark.
\end{remark}

\subsection{Semigroup Theory Notations}

The main tool used in the proof of our moment asymptotics is the Feynman-Kac
formula. For this, we use the following:
\begin{itemize}
\item Given any $r>0$, we denote the centered open box with side-length $2r$ as $Q_r:=(-r,r)^d$.
\item We use $\mathbbm 1$ to denote the function equal to one everywhere, and given a set
$K$, we use $\mathbbm 1_K$ to denote the indicator function of $K$.
\item We use $(W_t)_{t\geq0}$ to denote a Brownian motion with generator $\ka\De$.
We use $\mbb E_x$ and $\mbb P_x$ to denote the expectation
and probability law of $W$ conditioned on the starting point $W_0=x$.
({\bf Remark.} We always assume that $W$ is independent of the noise $\xi$.)
\item We use $(L_t)_{t\geq0}$ to denote $W$'s normalized occupation measures; that is,
\[L_t(K):=\frac1t\int_0^t\mathbbm 1_{\{W_s\in K\}}\d s\qquad\text{for any Borel set $K\subset\mbb R$}.\]
\item For every closed set $K\subset\mbb R^d$,
we use $\ms P(K)$ to denote the set of probability measures supported in $K$.
\item We use $(\cdot,\cdot)$ to denote both the $L^2$ inner product and integration against a measure.
That is, for two functions $f,g$ and a measure $\mu$, one has
\[( f,g):=\int_{\mbb R^d}f(x)g(x)\d x\qquad\text{and}\qquad(f,\mu):=\int_{\mbb R^d}f(x)\d\mu(x).\]
\item We denote the Dirichlet form of $\ka\De$ as
\begin{align}
\label{Equation: Dirichlet Form}
\ms S(f):=
\ka\int_{\mbb R^d}|\nabla f(x)|^2\d x,\qquad f\in H^1(\mbb R^d).
\end{align}
\end{itemize}

Henceforth, given a smooth function $V$ (possibly random), we use $u^V(t,x)$ to denote the solution of
\[\begin{cases}
\partial_t u^V(t,x)=\big(\ka\De+V(x)\big)u^V(t,x)\\
u^V(0,\cdot)=\mathbbm 1
\end{cases},\qquad t\geq0,~x\in\mbb R^d.\]
In particular, with the notation introduced in previous sections, we have that
\[u_\eps(t,x)=u^{\xi_\eps}(t,x).\]
For every smooth function $V$ and $r>0$ we use $u^V_r(t,x)$ to denote the solution of 
\[\begin{cases}
\partial_t u_r^V(t,x)=\big(\ka\De+V(x)\big)u_r^V(t,x)\\
u_r^V(0,\cdot)=\mathbbm 1_{Q_r}\\
u^V_r(t,x)\mathbbm1_{\mbb R^d\setminus Q_r}(x)=0
\end{cases},\qquad t\geq0,~x\in\mbb R^d.\]
By the Feynman-Kac formula,
\begin{align}
\label{Equation: Feynman-Kac}
u^V(t,x)=\mbb E_x\left[\mr e^{t( V,L_t)}\right]
\qquad\text{and}\qquad u^V_r(t,x)=\mbb E_x\left[\mr e^{t( V,L_t)}\mathbbm 1_{\{L_t\in\ms P(Q_r)\}}\right].
\end{align}

We let $\la^V_{k}(Q_r)$ ($k\in\mbb N$) denote the eigenvalues of the operator $\ka\De+V$ on $Q_r$ with Dirichlet
boundary conditions in decreasing order, and we let $e^V_{k}(Q_r)$ denote the corresponding
orthonormal eigenfunctions. In particular, we have the spectral expansion
\begin{align}
\label{Equation: Spectral Expansion}
u^V_r(t,\cdot)=\sum_{k=1}^\infty\mr e^{t\la^V_{k}(Q_r)}\big(e^V_{k}(Q_r),\mathbbm 1\big)e^V_{k}(Q_r).
\end{align}
Though the eigenfunctions $e^V_{k}(Q_r)$ are defined on $Q_r$,
we extend their domain to all of $\mbb R^d$ by setting their value to be zero outside of $Q_r$.

\subsection{Heuristic Derivation of Theorems \ref{Theorem: Sub}, \ref{Theorem: Crt},
and \ref{Theorem: Sup}}
\label{Section: Moment Heuristics Details}

If intermittency occurs in $u^{\xi_\eps}(t,\cdot)$,
then we expect that there exists sparse localized regions of $\mbb R^d$ inside which $u^{\xi_\eps}(t,\cdot)$ takes unusually large values.
In particular, the main contribution to $\langle u^{\xi_\eps}(t,x)^p\rangle$ comes from the event where a large
peak in the noise is present near $x$,
thus inducing a large peak in $u^{\xi_\eps}(t,\cdot)$ near $x$ as well.
Since $u^{\xi_\eps}(t,\cdot)$ is stationary, there is no loss of generality in
studying this phenomenon at $x=0$.
In this case, we expect that we can approximate
\begin{align}
\label{Equation: Intermittency Heuristic 1}
\langle u^{\xi_\eps}(t,0)^p\rangle\approx\big\langle u^{\xi_\eps}_{R\al_\eps(pt)}(t,0)^p\big\rangle=\left\langle\mbb E_0\left[\mr e^{t(\xi_{\eps},L_{t})}\mathbbm 1_{\{L_{t}\in\ms P(Q_{R\al_\eps(pt)})\}}\right]^p\right\rangle,
\end{align}
where $R>0$ is a large constant
and $\al_\eps(pt)>0$ is a scaling function (which depends on $p,\eps$, and $t$) that describes the diameter of $u^{\xi_\eps}(t,0)$'s large peak near the origin that contributes the most to the $p^{\mr{th}}$ moment.
Once we restrict the support of the occupation measure $L_{t}$ to the box $Q_{R\al_\eps(pt)}$,
we can use the spectral expansion \eqref{Equation: Spectral Expansion}. If we apply this to
\eqref{Equation: Intermittency Heuristic 1},
neglecting all terms but the leading eigenvalue (which, asymptotically, should dominate all other terms), then
we are led to
\begin{align}
\label{Equation: Intermittency Heuristic 2}
\langle u^{\xi_\eps}(t,0)^p\rangle\approx\left\langle\mr e^{pt\la^{\xi_\eps}_{1}(Q_{R\al_\eps(pt)})}\right\rangle.
\end{align}
Thus, we now wish to understand
how the geometry of large peaks in $\xi_{\eps}$ determines the 
behavior of the leading eigenvalue's moment generating function.

In order to study the asymptotic behavior of \eqref{Equation: Intermittency Heuristic 2},
for every $p,\eps,t>0$,
we introduce a rescaling of the noise of the form
\begin{align}
\label{Equation: Scaled Noise}
\Xi_{\eps,t}^p(x):=\al_\eps(pt)^2\left(\xi_{\eps}\big(\al_\eps(pt) x\big)-\frac{H_\eps(pt)}{pt}\right),\qquad x\in\mbb R^d
\end{align}
for an appropriate choice of scaling function $H_\eps$. By a rescaling of the eigenvalue,
this yields
\begin{align}
\label{Equation: Intermittency Heuristic 3}
\left\langle\mr e^{pt\la^{\xi_\eps}_{1}(Q_{R\al_\eps(pt)})}\right\rangle=\mr e^{H_\eps(pt)}\left\langle\mr e^{\be_\eps(pt)\la^{\Xi_{\eps,t}^p}_{1}(Q_R)}\right\rangle,
\end{align}
where we define $\be_\eps(pt):=pt/\al_\eps^2(pt)$.
Suppose that we can prove a large deviation principle for the rescaled noise $\Xi_{\msf e,\msf t}^p$
as $\msf m\to\infty$,
which, informally, takes the following form:

\begin{proposition}[Informal]
\label{Proposition: Informal LDP}
For any $``$shape$"$ function
$V\geq0$ supported on $Q_R$,
\begin{multline}
\label{Equation: Informal LDP}
\mr{Prob}\left[\xi_{\msf e}\text{ has a peak with shape }\tfrac{H_\msf e(p\msf t)}{p\msf t}+\tfrac{V(\cdot/\al_\msf e(p\msf t))}{\al_\msf e(p\msf t)^2}\text{ in }Q_{R\al_\msf e(p\msf t)}\right]\\
=\mr{Prob}\left[\Xi_{\msf e,\msf t}^p\text{ has a peak with shape }V\text{ in }Q_{R}\right]
\approx\mr e^{-\be_\msf e(p\msf t)\bs I^p(V)}
\end{multline}
as $\msf m\to\infty$ for some rate function $\bs I^p$.
\end{proposition}

By an informal application of the Laplace-Varadhan method, this implies that
\begin{multline}
\label{Equation: Intermittency Heuristic 4}
\mr e^{H_\msf e(p\msf t)}\left\langle\mr e^{\be_\msf e(p\msf t)\la^{\Xi_{\msf e,\msf t}}_1(Q_R))}\right\rangle
\approx\mr e^{H_\msf e(p\msf t)}\int\mr e^{\be_\msf e(p\msf t)(\la^{V}_1(Q_R)-\bs I^p(V))}~``\dd V"\\
\approx\mr e^{H_\msf e(p\msf t)+\be_\msf e(p\msf t)\sup_{V}(\la^{V}_1(Q_R)-\bs I^p(V))}.
\end{multline}
At this point, if we take $R\to\infty$ in \eqref{Equation: Intermittency Heuristic 4}
and combine the result with \eqref{Equation: Intermittency Heuristic 2}
and \eqref{Equation: Intermittency Heuristic 3}, then we are led to
the asymptotic
\[\langle u^{\xi_\msf e}(\msf t,0)^p\rangle=H_\msf e(p\msf t)-\be_\msf e(p\msf t)\bs\chi^p\big(1+o(1)\big)\qquad\text{as }\msf m\to\infty,\]
where
\begin{multline*}
\bs\chi^p
=-\sup_{V}\Big(\la^{V}_1(\mbb R^d)-\bs I^p(V)\Big)\\
=-\sup_V\Big(\sup_f\big\{(V,f^2)-\ms S(f)\big\}-\bs I^p(V)\Big)
=\inf_f\Big(\ms S(f)+\bs J^p(f^2)\Big),
\end{multline*}
and $-\bs J^p$ denotes $\bs I^p$'s Fenchel-Legendre transform.

With this in hand, the statements of Theorems \ref{Theorem: Sub}, \ref{Theorem: Crt},
and \ref{Theorem: Sup} are obtained by noting that
a formal version of something like Proposition \ref{Proposition: Informal LDP}
(see Propositions \ref{Proposition: Smooth LDP} and \ref{Proposition: Sub-3 LDP})
can be proved using the following dual rate function $\bs J^p$ and scaling functions
$\al_\eps,\be_\eps$, and $H_\eps$:

\begin{definition}
\label{Definition: Table}
For every $\eps,t>0$, we define $\al_\eps(t),\be_\eps(t)$, and $H_\eps(t)$ as follows:
\begin{center}
\begin{tabular}{c||c|c|c}
{Case}&${\al_\eps(t)}$&${\be_\eps(t)}$&${H_\eps(t)}$\\
\hline
\hline
Sub-1&$\eps^{(2+\om)/4}t^{-1/4}$&$\eps^{-(2+\om)/2}t^{3/2}$&$\eps^{-\om}t^2\frac{\ga_1(0)}{2}$\\
Sub-2&$t^{-1/(2-\om)}$&$t^{(4-\om)/(2-\om)}$&0\\
Sub-3&$t^{-1/(2-\om)}$&$t^{(4-\om)/(2-\om)}$&0\\
\hline
Crt-1&$\eps t^{-1/4}$&$\eps^{-2}t^{3/2}$&$\eps^{-2}t^2\frac{\ga_1(0)}{2}$\\
Crt-2&$\eps$&$\eps^{-2}t$&0\\
\hline
Sup&$\eps^{(2+\om)/4}t^{-1/4}$&$\eps^{-(2+\om)/2}t^{3/2}$&$\eps^{-\om}t^2\frac{\ga_1(0)}{2}$
\end{tabular}
\end{center}

Next, define
\[J_\infty(\mu):=\frac14\iint_{(\mbb R^d)^2}(x-y)^\top\Si(x-y)\d \mu(x)\dd \mu(y),
\qquad \mu\in \ms P(\mbb R^d),\]
recalling that $-\Si$ is the Hessian matrix of $\ga_1$ at $x=0$. For every $c\in[0,\infty)$, let
\[J_c(\mu):=-\frac12\iint_{(\mbb R^d)^2}\ga_c(x-y)\d \mu(x)\dd \mu(y),\qquad\mu\in\ms P(\mbb R^d).\]
In both cases above (i.e., for $c\in[0,\infty]$), if $f\in H^1(\mbb R^d)$ is such that
$\|f\|_2=1$, then we use the convention that $J_{c}(f^2)=J_{c}(f^2(x)\dd x)$.
For every $p>0$, we define the dual rate function $\bs J^p$ as
\begin{center}
\begin{tabular}{c||ccc|cc|c}
{Case}&Sub-1&Sub-2&Sub-3&Crt-1&Crt-2&Sup\\
\hline
\hline
${\bs J^p}$&
$J_\infty$&
$J_{p^{1/(2-\om)}\mf c}$&
$J_0$&
$J_\infty$&
$p tJ_1$&
$J_\infty$
\end{tabular}
\end{center}

Finally, for every $p>1$, we let
\[\bs\chi^p:=\inf_{f\in H^1(\mbb R^d),~\|f\|_2=1}\Big(\ms S(f)+\bs J^p(f^2)\Big)=\begin{cases}
\chi&\text{(Sub-1, Crt-1, Sup)}\\
-\mbf M_{\mf c,p}&\text{(Sub-2)}\\
-\mbf M&\text{(Sub-3)}\\
-\mbf M^{\mr{crt}}_{t,p}&\text{(Crt-2)}
\end{cases}.\]
\end{definition}

In particular,
Theorems \ref{Theorem: Sub}, \ref{Theorem: Crt},
and \ref{Theorem: Sup} can be summarized as follows:

\begin{theorem}
\label{Theorem: All Moments}
Suppose that one of Assumptions \ref{Assumption: Sub}, \ref{Assumption: Crt}, or
\ref{Assumption: Sup} holds.
For every $p>0$ and $x\in\mbb R^d$, one has
\[\lim_{\msf m\to\infty}\frac{\log\big\langle u^{\xi_\msf e}(\msf t,x)^p\big\rangle-H_\msf e(p\msf t)}{\be_\msf e(p\msf t)}=-\bs\chi^p.\]
\end{theorem}

\begin{remark}
In \cite{GartnerKonig}, the rescaled noise in \eqref{Equation: Scaled Noise}
is instead denoted $\xi_t$.
This is because in \cite{GartnerKonig}, the scaling parameters $\al,\be$ and $H$
only depend on time (in contrast to depending on both time and the parameter $\eps$ in our setting).
Given that $\xi_\eps$ already has an index that keeps track of its dependence on $\eps$,
we use the different notation $\Xi_{\eps,t}^p$.
\end{remark}

\begin{remark}
\label{Remark: H is not always cumulant}
In \cite{GartnerKonig}, the scaling function denoted by $H$ corresponds to the
cumulant generating function of the noise (i.e., $H(t)=\log\langle\mr e^{t\xi(0)}\rangle$).
In our paper this is still true in cases Sub-1, Crt-1, and Sup
(i.e., $H_\eps(t)=\log\langle\mr e^{t\xi_\eps(0)}\rangle$), but not in cases Sub-2, Sub-3, and Crt-2.
Our main purpose for defining $H_\eps(t)=0$ in those latter three cases
is to set up a convenient unified notation, which simplifies some statements.
That said, the fact that the leading-order asymptotics of the PAM's moments are not determined
by the cumulant generating function in case Sub-3 has important implications for the geometry
of intermittency; see Section \ref{Section: Motivation for Maximizers}.
\end{remark}

\begin{remark}
$\bs J^p$ and $\bs\chi^p$ only depend on $p$ in cases Sub-2 and Crt-2.
We nevertheless use the superscript in all cases to
avoid different notations in different cases.
We note that $\bs J^p$ and $\bs\chi^p$ actually depend on another
parameter in those two cases, respectively $\mf c$ and $t$.
Since these
parameters form an integral part of the very definition of cases Sub-2
and Crt-2
(see Assumptions \ref{Assumption: Sub} and \ref{Assumption: Crt}) and are otherwise fixed, we leave this dependence implicit.
\end{remark}

\subsection{Geometric Interpretation}
\label{Section: Motivation for Maximizers}

Let $x\in\mbb R^d$ and $p>0$ be fixed.
The large deviations heuristics provided in equations
\eqref{Equation: Intermittency Heuristic 1}--\eqref{Equation: Intermittency Heuristic 4}
and Proposition \ref{Proposition: Informal LDP}
suggests that the event that contributes the most to the asymptotics for $\langle u_\msf e(\msf t,x)^p\rangle$ (stated in Theorem \ref{Theorem: All Moments}) is the following:
\begin{enumerate}
\item A large peak in the noise $\xi_{\msf e}$ close to $x$ of the form
\begin{align}
\label{Equation: Noise Shape}
\xi_\msf e(y)\approx\frac{H_\msf e(p\msf t)}{p\msf t}+\frac{V_\star\big(y/\al_\msf e(p\msf t)\big)}{\al_\msf e(p\msf t)^2}
\qquad\text{for }y\approx x,
\end{align}
where
\[V_\star:=\arg\sup_V\Big(\sup_f\big\{(V,f^2)-\ms S(f)\big\}-\bs I^p(V)\Big).\]
\item Equivalently, a large peak in $u_\msf e(\msf t,\cdot)^p$ close to $x$ of the form
\[u_\msf e(\msf t,y)^p\approx\mr e^{H_\msf e(p\msf t)-\be_\msf e(p\msf t)\bs \chi^p(1+o(1))}f_\star\left(\frac{y}{\al_\msf e(p\msf t)}\right)^p
\qquad\text{for }y\approx x,\]
where
\[f_\star=\arg\inf_f\Big(\ms S(f)+\bs J^p(f^2)\Big).\]
\end{enumerate}

In Cases Sub-1, Crt-2, and Sup, this geometry is similar to the one that
gives rise to the classical
asymptotics \eqref{Equation: GK 2nd Order} uncovered in \cite{GartnerKonig}.
That is, as illustrated in Figure \ref{Figure: 2 Orders},
\begin{figure}[htbp]
\begin{center}
\input{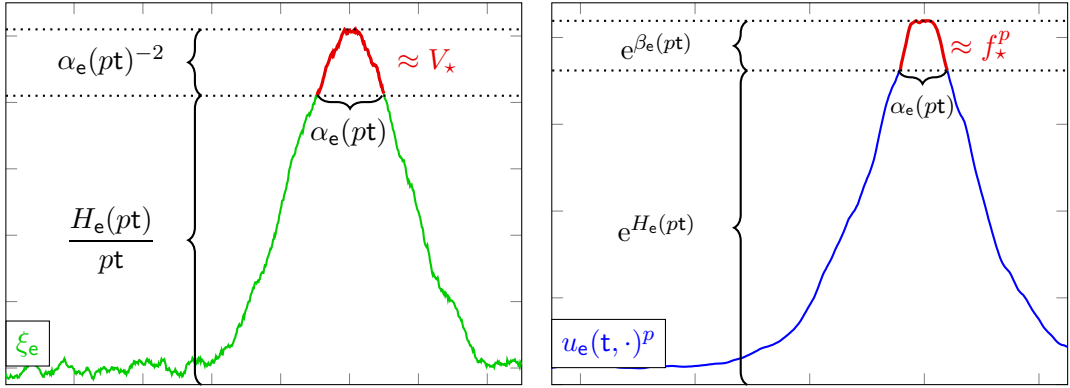}
\caption{Geometry of intermittent peaks in cases Sub-1, Crt-1, and Sup.
On the left is the configuration in the noise that contributes the most to $\langle u_\msf e(\msf t,x)^p\rangle$;
on the right is the corresponding configuration in the $p^{\mr{th}}$ power of the PAM.}
\label{Figure: 2 Orders}
\end{center}
\end{figure}
the geometry of intermittent peaks requires two distinct orders for its description: On the one hand,
the first-order scaling function $H_\eps$ (which, in those cases, is the cumulant generating function
of the noise; see Remark \ref{Remark: H is not always cumulant}) describes the height of the
intermittent peaks. On the other hand, the second-order scaling functions $\al_\eps$ and $\be_\eps$
and the minimizers $V_\star$ and $f_\star$ describe the
height, the width, and the shape of those peaks close to their summits.

As a continuation of Sections \ref{Section: Interpretation of Sub} and \ref{Section: Interpretation of Ctr},
we note that if $\msf e$ and $\msf t$ are such that $\be_\msf e(p\msf t)$ is of the
same order as the cumulant generating function of $\xi_\msf e$,
then the geometry of intermittent peaks undergoes a transition. This gives rise to
cases Sub-2 and Crt-2, which are illustrated in Figure \ref{Figure: Coalesce 1}.
\begin{figure}[htbp]
\begin{center}
\input{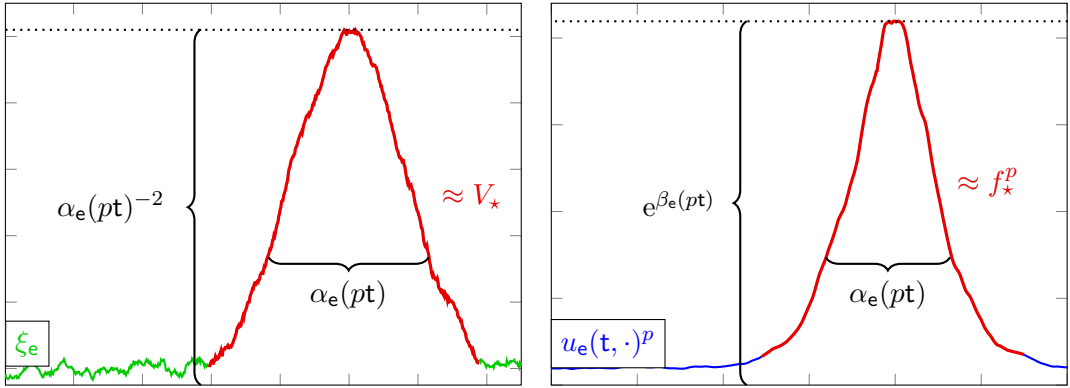}
\caption{Geometry of intermittent peaks in cases Sub-2 and Crt-2.
Once again the noise is on the left and the solution of the PAM on the right.}
\label{Figure: Coalesce 1}
\end{center}
\end{figure}
Therein, we note that we only need
$\al_\eps$ and $\be_\eps$ to describe the magnitude of the peaks, and $V_\star$
and $f_\star$ describe the shapes of the entire peaks, rather than just the summits.

\begin{remark}
In case Crt-2, if $\msf t\to t$, then we expect that Figure \ref{Figure: Coalesce 1}
only describes the geometry of intermittency when $p>\frac{2\ka}{t\mc G}$,
since this is the situation when $\mbf M^{\mr{crt}}_{t,p}\neq0$; see Proposition \ref{Proposition: Finite Time Intermittency}.
\end{remark}

Finally, if we take $\mf c\to0$ in case Sub-2, then
the geometry undergoes yet another transition, giving rise to case Sub-3.
The latter is illustrated in Figure \ref{Figure: Coalesce 2}.
\begin{figure}[htbp]
\begin{center}
\input{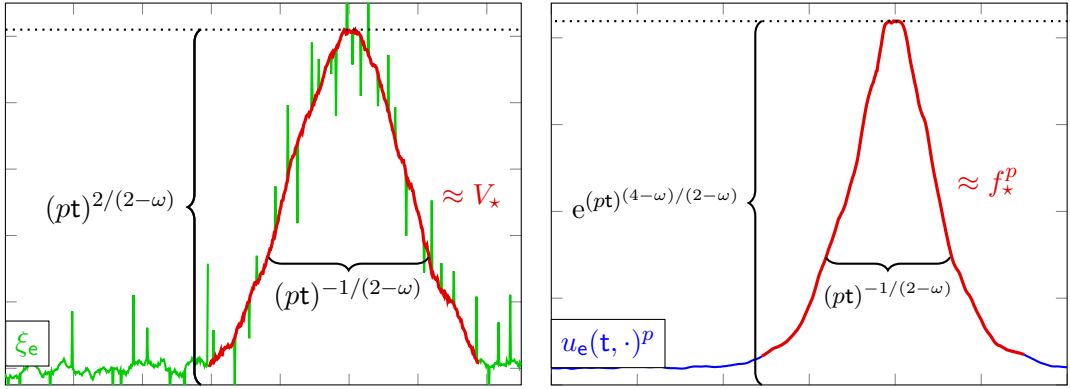}
\caption{Geometry of intermittent peaks in case Sub-3.}
\label{Figure: Coalesce 2}
\end{center}
\end{figure}
Similarly to cases Sub-2 and Crt-2, we only need $\al_\eps$ and $\be_\eps$
to describe the magnitude of intermittent peaks, and the optimizers of the variational
problems describe the shape of the entire peaks.
However, there is one meaningful difference between Figures \ref{Figure: Coalesce 1}
and \ref{Figure: Coalesce 2}: In case Sub-3, the approximate equality \eqref{Equation: Noise Shape}
should not be interpreted literally in the pointwise sense, but instead in the sense of
integration against the noise:
\[\big(\xi_\msf e,f\big)\approx
(p\msf t)^{2/(2-\om)}\Big(V_\star\big((p\msf t)^{1/(2-\om)}\cdot\big),f\Big),\qquad f\in C_0^\infty(\mbb R^d).\]
Indeed, in case Sub-3 the parameter $\msf e$ can vanish arbitrarily quickly, meaning
that the variance of $\xi_{\msf e(\msf m)}$ for large $\msf m$ can be arbitrarily large.
In particular, the random field $x\mapsto\xi_{\msf e(\msf m)}(x)$
becomes very singular as $\msf m\to\infty$ (as illustrated by the sharp oscillations in the noise
on the left-hand side of Figure \ref{Figure: Coalesce 2}, represented by green
``spikes"),
and the cumulant generating function of $\xi_\msf e$ does not appear to describe any
aspect of the geometry of intermittency. In particular, the geometry of
peaks in this case is independent of the parameter $\msf e$.

\section{Proof of Theorem \ref{Theorem: All Moments} Part 1 - Large Deviations}
\label{Section: Moments Part 1}

\subsection{Outline}

Our main purpose in this section, which consists of the
first step in the proof Theorem \ref{Theorem: All Moments},
is to formalize the informal large deviation principle
stated in Proposition \ref{Proposition: Informal LDP}.
This is done in {\bf Propositions \ref{Proposition: Smooth LDP} and \ref{Proposition: Sub-3 LDP}} below.
Much like in \cite{GartnerKonig}, the main difference between the informal statement
and Propositions \ref{Proposition: Smooth LDP} and \ref{Proposition: Sub-3 LDP} is that
the latter are proved in the dual setting, using the Varadhan lemma directly.
The immediate culmination of these results are the following two lemmas,
which are preliminary versions of Theorem \ref{Theorem: All Moments}:

\begin{definition}
For every $p,R>0$, define
\begin{align}
\label{Equation: Chi_R}
\bs\chi^p_R:=\inf_{\substack{f\in H^1(\mbb R^d),~\|f\|_2=1\\\mr{supp}(f)\subset Q_R}}\Big(\ka\ms S(f)+\bs J^p(f^2)\Big).
\end{align}
\end{definition}

\begin{lemma}
\label{Lemma: Averaged Asymptotics 1}
Let $p>0$ be fixed,
and let $\msf d=\msf d(\msf m)\geq0$ be such that $0\leq\msf d<\msf t$ and
\begin{align}
\label{Equation: Averaged Asymptotics 1 - Condition on d}
\lim_{\msf m\to\infty}\frac{\al_\msf e(p(\msf t-\msf d))}{\al_\msf e(p\msf t)}=1,
\quad
\lim_{\msf m\to\infty}\frac{\be_\msf e(p(\msf t-\msf d))}{\be_\msf e(p\msf t)}=1,
\quad
\lim_{\msf m\to\infty}\frac{H_\msf e(p(\msf t-\msf d))-H_\msf e(p\msf t)}{\be_\msf e(p\msf t)}=0.
\end{align}
For every $R>0$, one has
\begin{align}
\label{Equation: Averaged Asymptotics - Lower}
\liminf_{\msf m\to\infty}\frac{\log\Big\langle\big( u^{\xi_\msf e}_{R\al_\msf e(p\msf t)}(p(\msf t-\msf d),\cdot),\mathbbm 1\big)\Big\rangle-H_\msf e(p\msf t)}{\be_\msf e(p\msf t)}
\geq-\bs\chi^p_R.
\end{align}
\end{lemma}

\begin{lemma}
\label{Lemma: Averaged Asymptotics 2}
Let $p>0$ be fixed, and let $\msf R=\msf R(\msf m)>0$ be such that $\al_\msf e(p\msf t)=o(\msf R)$
and $\msf R=\mr e^{o(\be_\msf e(p\msf t))}$ as $\msf m\to\infty$. Then, it holds that
\begin{align}
\label{Equation: Averaged Asymptotics - Upper}
\limsup_{\msf m\to\infty}\frac{\log\Big\langle\big( u^{\xi_\msf e}_{\msf R}(p\msf t,\cdot),\mathbbm 1\big)\Big\rangle-H_\msf e(p\msf t)}{\be_\msf e(p\msf t)}
\leq-\bs\chi^p
\end{align}
and
\begin{align}
\label{Equation: Averaged Asymptotics - Trace}
\limsup_{\msf m\to\infty}\frac1{\be_\msf e(p\msf t)}\left(\log\left\langle\sum_{k=1}^\infty\mr e^{p\msf t\la^{\xi_\msf e}_{k}(Q_{\msf R})}\right\rangle-H_\msf e(p\msf t)\right)\leq-\bs\chi^p.
\end{align}
\end{lemma}

\subsubsection{Three Remarks}
\label{Section: Three Remarks}

Before proving Lemmas \ref{Lemma: Averaged Asymptotics 1}
and \ref{Lemma: Averaged Asymptotics 2}, we comment on the differences
between the latter and the statement of Theorem \ref{Theorem: All Moments}:
Indeed, instead of concerning the asymptotics of the moments $\langle u^{\xi_\msf e}(\msf t,x)^p\rangle,$
Lemmas \ref{Lemma: Averaged Asymptotics 1}
and \ref{Lemma: Averaged Asymptotics 2} concern the asymptotics of 
\[\Big\langle\big( u^{\xi_\msf e}_{\msf R}(p(\msf t-\msf d),\cdot),\mathbbm 1\big)\Big\rangle
\qquad\text{and}\qquad
\left\langle\sum_{k=1}^\infty\mr e^{p\msf t\la^{\xi_\msf e}_{k}(Q_{\msf R})}\right\rangle\]
for an appropriate choice of functions $\msf R,\msf d>0$. The reasons for this are threefold:

Firstly, since Propositions \ref{Proposition: Smooth LDP} and \ref{Proposition: Sub-3 LDP}
rely on the {\it weak} large deviation principle for Brownian occupation measures
(e.g., \cite[Section 4.2]{DeuschelStroock}), it is necessary to restrict our asymptotics
to a compact domain. In Sections \ref{Section: Moments Part 2} and \ref{Section: Moments Part 3},
we rigorously establish that up to an error of order $\mr e^{o(\be_\msf e(p\msf t))}$,
we can replace $\langle u^{\xi_\msf e}(\msf t,x)^p\rangle$
by $\langle u^{\xi_\msf e}_{\msf R}(\msf t,x)^p\rangle$, provided $\msf R$ is not
too small.

Secondly, while Propositions \ref{Proposition: Smooth LDP} and \ref{Proposition: Sub-3 LDP} can
easily be used to obtain asymptotics for the first moment $\langle u^{\xi_\msf e}_{\msf R}(\msf t,x)\rangle$
(for an appropriate choice of $\msf R$), the same is not true for $\langle u^{\xi_\msf e}_{\msf R}(\msf t,x)^p\rangle$
when $p\neq1$. In short, this is because introducing a $p^{\mr{th}}$ power on
the solution of the PAM makes it more difficult to integrate the noise in the Feynman-Kac
formula and apply the large deviation principle, as done informally in \eqref{Equation: Intermittency Heuristic 4}.
In order to get around this issue, a crucial aspect of the strategy developed in \cite{GartnerKonig}
is to show that, up to an error of order $\mr e^{o(\be_\msf e(p\msf t))}$,
we can replace $\langle u^{\xi_\msf e}_{\msf R}(\msf t,x)^p\rangle$ by $\langle u^{\xi_\msf e}_{\msf R}(p\msf t,x)\rangle$.
Heuristically, this can be justified by using a spectral expansion and only keeping the leading
eigenvalue, as in \eqref{Equation: Intermittency Heuristic 2}.
However, given that the integral of the eigenfunctions $(e^V_k(Q_r),\mathbbm 1)$ is more convenient to
control than its value $e^V_k(Q_r)(x)$ at a single point $x\in\mbb R^d$, a formal version of
this heuristic is easier to achieve if we instead look at the integral of the solution
$\big( u^{\xi_\msf e}_{\msf R}(\msf t,\cdot),\mathbbm 1\big)^p.$
In Sections \ref{Section: Moments Part 2} and \ref{Section: Moments Part 3},
we rigorously prove that up to an error of order $\mr e^{o(\be_\msf e(p\msf t))}$,
one has
\begin{align}
\label{Equation: Average the Solution}
\langle u^{\xi_\msf e}_{\msf R}(\msf t,x)^p\rangle\approx\Big\langle\big( u^{\xi_\msf e}_{\msf R}(\msf t,\cdot),\mathbbm 1\big)^p\Big\rangle\approx\Big\langle\big( u^{\xi_\msf e}_{\msf R}(p\msf t,\cdot),\mathbbm 1\big)\Big\rangle.
\end{align}

Thirdly, the replacement of $\msf t$ by a truncated time $\msf t-\msf d$
in Lemma \ref{Lemma: Averaged Asymptotics 1} amounts to a technical
concern. More specifically, in order to establish \eqref{Equation: Average the Solution},
it is in many cases necessary to ``integrate out" the contribution of the PAM's solution on the
time interval $[0,\msf d]$; we refer to Sections \ref{Section: Moments Part 2} and \ref{Section: Moments Part 3}
for the details.

The remainder of Section \ref{Section: Moments Part 1} is now
devoted to the proof of Lemmas \ref{Lemma: Averaged Asymptotics 1}
and \ref{Lemma: Averaged Asymptotics 2}.

\subsection{Proof of Lemma \ref{Lemma: Averaged Asymptotics 1}}

In addition to \eqref{Equation: Averaged Asymptotics - Lower},
in this proof we also establish
\begin{align}
\label{Equation: Averaged Asymptotics with alpha - Upper}
\limsup_{\msf m\to\infty}\frac{\log\Big\langle\big( u^{\xi_\msf e}_{R\al_\msf e(p\msf t)}(p(\msf t-\msf d),\cdot),\mathbbm 1\big)\Big\rangle-H_\msf e(p\msf t)}{\be_\msf e(p\msf t)}
\leq-\bs\chi^p
\end{align}
and
\begin{align}
\label{Equation: Averaged Asymptotics with alpha - Trace}
\limsup_{\msf m\to\infty}\frac1{\be_\msf e(p\msf t)}\left(\log\left\langle\sum_{k=1}^\infty\mr e^{p(\msf t-\msf d)\la^{\xi_\msf e}_{k}(Q_{R\al_\msf e(p\msf t)})}\right\rangle-H_\msf e(p\msf t)\right)\leq-\bs\chi^p.
\end{align}
While these two results are not directly related to Lemma \ref{Lemma: Averaged Asymptotics 1},
they are needed in the proof of Lemma \ref{Lemma: Averaged Asymptotics 2} and use the same ideas
as the proof of \eqref{Equation: Averaged Asymptotics - Lower}.

By Brownian scaling (see \cite[(1.3)]{GartnerKonig}) one has
\[u^{\xi_\msf e}_{R\al_\msf e(p\msf t)}(p(\msf t-\msf d),\cdot)=\mr e^{H_\msf e(p(\msf t-\msf d))}u^{\Xi_{\msf e,\msf t-\msf d}^p}_{R\al_\msf e(p\msf t)/\al_\msf e(p(\msf t-\msf d))}\left(\frac{p(\msf t-\msf d)}{\al_\msf e(p(\msf t-\msf d))^2},\frac\cdot{\al_\msf e(p(\msf t-\msf d))}\right)\]
for every $R>0$ and $p>0$,
where we recall that $\Xi_{\eps,t}^p$ is defined as in \eqref{Equation: Scaled Noise}.
For any $\theta>0$,
if $\msf m$ is large enough, then the first limit in \eqref{Equation: Averaged Asymptotics 1 - Condition on d}
implies that
\[R-\theta<R\al_\msf e(p\msf t)/\al_\msf e(p(\msf t-\msf d))<R+\theta.\]
Consequently (recalling that $\be_\eps(pt)=pt/\al_\eps(pt)^2$), for every $\theta>0$, one has
\begin{multline}
\label{Equation: Lower Bound Scaling 0}
\mr e^{H_\msf e(p(\msf t-\msf d))}u^{\Xi_{\msf e,\msf t-\msf d}^p}_{R-\theta}\left(\be_\msf e(p(\msf t-\msf d)),\frac\cdot{\al_\msf e(p(\msf t-\msf d))}\right)\leq u^{\xi_\msf e}_{R\al_\msf e(p\msf t)}(p(\msf t-\msf d),\cdot)\\\leq\mr e^{H_\msf e(p(\msf t-\msf d))}u^{\Xi_{\msf e,\msf t-\msf d}^p}_{R+\theta}\left(\be_\msf e(p(\msf t-\msf d)),\frac\cdot{\al_\msf e(p(\msf t-\msf d))}\right),
\end{multline}
provided $\msf m$ is large enough.
Moreover, by the Feynman-Kac formula and Tonelli's theorem,
\begin{multline}
\left\langle \Big( u^{\Xi_{\msf e,\msf t-\msf d}^p}_{R\pm\theta}\big(\be_\msf e(p(\msf t-\msf d)),\cdot/\al_\msf e(p(\msf t-\msf d))\big),\mathbbm1\Big)\right\rangle\\
\label{Equation: Lower Bound Scaling}
=\al_\msf e(p(\msf t-\msf d))^d\int_{Q_{R\pm\theta}}\mbb E_x\left[\left\langle\mr e^{\be_\msf e(p(\msf t-\msf d))(\Xi_{\msf e,\msf t-\msf d}^p,L_{\be_\msf e(p(\msf t-\msf d))})}\right\rangle\mathbbm 1_{\{L_{\be_\msf e(p(\msf t-\msf d))}\in\ms P(Q_{R\pm\theta})\}}\right]\d x,
\end{multline}
where the term $\al_\msf e(p(\msf t-\msf d))^d$ comes from a change of variables in the $\dd x$ integral.
In order to control the exponential moment of $\be_\msf e(p(\msf t-\msf d))(\Xi_{\msf e,\msf t-\msf d}^p,L_{\be_\msf e(p(\msf t-\msf d))})$,
we use the formal version of the large deviation principle stated in Proposition
\ref{Proposition: Informal LDP}, which amounts to the following:

\begin{proposition}
\label{Proposition: Smooth LDP}
Let $\msf d$ satisfy \eqref{Equation: Averaged Asymptotics 1 - Condition on d},
let $R>0$ and $p>0$, and suppose that
we are in any case except Sub-3. It holds that
\begin{align}
\label{Equation: Smooth LDP}
\lim_{\msf m\to\infty}\sup_{\mu\in \ms P(Q_R)}\left|\frac1{\be_\msf e(p(\msf t-\msf d))}\log\left\langle\mr e^{\be_\msf e(p(\msf t-\msf d))(\Xi_{\msf e,\msf t-\msf d}^p,\mu)}\right\rangle+\bs J^p(\mu)\right|=0.
\end{align}
Moreover, the map
\[\mu\mapsto \bs J^p(\mu)\]
is continuous with respect to the topology of weak convergence on $\ms P(Q_R)$.
\end{proposition}

\begin{proposition}
\label{Proposition: Sub-3 LDP}
Let $\msf d$ satisfy \eqref{Equation: Averaged Asymptotics 1 - Condition on d}.
Suppose that we are in case Sub-3,
and let $\eps>0$ be arbitrary. For $\msf m$ large enough we have that
\begin{align}
\label{Equation: Sub-3 LDP 1}
\mr e^{-\be_\msf e(p(\msf t-\msf d))J_\eps(L_{\be_\msf e(p(\msf t-\msf d))})}\leq\left\langle\mr e^{\be_\msf e(p(\msf t-\msf d))(\Xi_{\msf e,\msf t-\msf d}^p,L_{\be_\msf e(p(\msf t-\msf d))})}\right\rangle
\leq\mr e^{-\be_\msf e(p(\msf t-\msf d))J_0(L_{\be_\msf e(p(\msf t-\msf d))})}.
\end{align}
Moreover, for every $R\in(0,\infty]$ (using the convention that $Q_\infty=\mbb R^d$),
\begin{align}
\label{Equation: Sub-3 LDP 2}
\lim_{\eps\to0}\inf_{\substack{f\in H^1(\mbb R^d),~\|f\|_2=1\\\mr{supp}(f)\subset Q_R}}\Big(\ka\ms S(f)+J_\eps(f^2)\Big)
=\inf_{\substack{f\in H^1(\mbb R^d),~\|f\|_2=1\\\mr{supp}(f)\subset Q_R}}\Big(\ka\ms S(f)+J_0(f^2)\Big).
\end{align}
\end{proposition}

\begin{remark}
The reason why case Sub-3 has its own proposition is that the uniform asymptotic
\eqref{Equation: Smooth LDP} is difficult to achieve when the noise is singular.
We circumvent this with Proposition \ref{Proposition: Sub-3 LDP}.
In particular, controlling the upper bound in \eqref{Equation: Sub-3 LDP 1}
requires the use of \eqref{Equation: Sub Known} in the case
$p=1$, which is one of the ways in which the argument used here differs from that of
\cite{GartnerKonig}. Nevertheless, the fact that the lower bound in \eqref{Equation: Smooth LDP}
can be controlled using the same theory as in \cite{GartnerKonig} makes it possible to uncover
the geometric interpretation of the moment asymptotics explained in Section \ref{Section: Motivation for Maximizers}.
\end{remark}

With these two propositions in hand, \eqref{Equation: Averaged Asymptotics - Lower} and
\eqref{Equation: Averaged Asymptotics with alpha - Upper} now follow from
standard large deviations arguments: Suppose first that we are in any case except Sub-3.
If we combine \eqref{Equation: Averaged Asymptotics 1 - Condition on d}
(whereby we can replace $\msf t-\msf d$ by $\msf t$ in the asymptotics),
\eqref{Equation: Lower Bound Scaling 0}, and \eqref{Equation: Lower Bound Scaling} with the fact that
$\al_\msf e(p(\msf t-\msf d))=\mr e^{o(\be_\msf e(p(\msf t-\msf d)))}$, then
\eqref{Equation: Averaged Asymptotics - Lower} and \eqref{Equation: Averaged Asymptotics with alpha - Upper} will be proved if we show that
\[\frac1{\be_\msf e(p(\msf t-\msf d))}\log\int_{Q_{R\pm\theta}}\mbb E_x\left[\left\langle\mr e^{\be_\msf e(p(\msf t-\msf d))(\Xi_{\msf e,\msf t-\msf d}^p,L_{\be_\msf e(p(\msf t-\msf d))})}\right\rangle\mathbbm 1_{\{L_{\be_\msf e(p(\msf t-\msf d))}\in\ms P(Q_{R\pm\theta})\}}\right]\d x\to-\bs\chi^p_{R\pm\theta}\]
as $\msf m\to\infty$.
Indeed, we note that $-\bs\chi^p_{R+\theta}\leq-\bs\chi^p$, and that $\bs\chi^p_{R-\theta}\to\bs\chi^p_R$ as $\theta\to0$
by Lemma \ref{Lemma: Lower Bound Variational Limit}.
By \eqref{Equation: Smooth LDP}, this is equivalent to
\[\lim_{\msf m\to\infty}\frac1{\be_\msf e(p(\msf t-\msf d))}\log\int_{Q_{R\pm\theta}}\mbb E_x\left[\mr e^{-\be_\msf e(p(\msf t-\msf d))\bs J^p(L_{\be_\msf e(p(\msf t-\msf d))})}\mathbbm 1_{\{L_{\be_\msf e(p(\msf t-\msf d))}\in\ms P(Q_{R\pm\theta})\}}\right]\d x=-\bs\chi^p_{R\pm\theta},\] which is a consequence of Varadhan's lemma
and the weak large deviation principle for Brownian local time
with a uniform starting point on $Q_{R\pm\theta}$ (e.g., \cite[Section 4.2]{DeuschelStroock}).

Consider now case Sub-3.
On the one hand,
by the lower bound in \eqref{Equation: Sub-3 LDP 1},
for every $\eps>0$, the following holds as $\msf m\to\infty$:
\begin{align*}
&\int_{Q_{R-\theta}}\mbb E_x\left[\left\langle\mr e^{\be_\msf e(p(\msf t-\msf d))(\Xi_{\msf e,\msf t-\msf d}^p,L_{\be_\msf e(p(\msf t-\msf d))})}\right\rangle\mathbbm 1_{\{L_{\be_\msf e(p(\msf t-\msf d))}\in\ms P(Q_{R-\theta})\}}\right]\d x\\
&\hspace{1in}\geq
\int_{Q_{R-\theta}}\mbb E_x\left[\mr e^{-\be_\msf e(p(\msf t-\msf d))J_\eps(L_{\be_\msf e(p(\msf t-\msf d))})}\mathbbm 1_{\{L_{\be_\msf e(p(\msf t-\msf d))}\in\ms P(Q_{R-\theta})\}}\right]\d x\\
&\hspace{1in}=\exp\left(\be_\msf e(p(\msf t-\msf d))\inf_{\substack{f\in H^1(\mbb R^d),~\|f\|_2=1\\\mr{supp}(f)\subset Q_{R-\theta}}}\Big(\ka\ms S(f)+J_\eps(f^2)\Big)+o(\be_\msf e(p(\msf t-\msf d)))\right),
\end{align*}
where the third line follows from the same application of Varadhan's lemma as in the
previous paragraph. Taking $\eps\to0$ using \eqref{Equation: Sub-3 LDP 2} then yields
\begin{multline*}
\liminf_{\msf m\to\infty}\frac{1}{\be_\msf e(p(\msf t-\msf d))}\log\int_{Q_{R-\theta}}\left\langle\mbb E_x\left[\mr e^{\be_\msf e(p(\msf t-\msf d))(\Xi_{\msf e,\msf t-\msf d}^p,L_{\be_\msf e(p(\msf t-\msf d))})}\mathbbm 1_{\{L_{\be_\msf e(p(\msf t-\msf d))}\in\ms P(Q_{R-\theta})\}}\right]\right\rangle\d x\\\geq-\bs\chi^p_{R-\theta}.
\end{multline*}
This concludes the proof of \eqref{Equation: Averaged Asymptotics - Lower} for Sub-3 by taking $\theta\to0$ using Lemma \ref{Lemma: Lower Bound Variational Limit}.
On the other hand, by the upper bound in \eqref{Equation: Sub-3 LDP 1} and removing the indicator of $L_{\be_\msf e(p(\msf t-\msf d))}\in\ms P(Q_{R+\theta})$,
\begin{multline*}
\int_{Q_{R+\theta}}\left\langle\mbb E_x\left[\mr e^{\be_\msf e(p(\msf t-\msf d))(\Xi_{\msf e,\msf t-\msf d}^p,L_{\be_\msf e(p(\msf t-\msf d))})}\mathbbm 1_{\{L_{\be_\msf e(p(\msf t-\msf d))}\in\ms P(Q_{R+\theta})\}}\right]\right\rangle\d x\\
\leq\int_{Q_{R+\theta}}\mbb E_x\left[\mr e^{-\be_\msf e(p(\msf t-\msf d))J_0(L_{\be_\msf e(p(\msf t-\msf d))})}\right]\d x
=(2(R+\theta))^d\mbb E_0\left[\mr e^{-\be_\msf e(p(\msf t-\msf d))J_0(L_{\be_\msf e(p(\msf t-\msf d))})}\right]
\end{multline*}
for large enough $\msf m$,
where the last equality follows from the fact that $J_0$ is translation invariant.
\eqref{Equation: Averaged Asymptotics with alpha - Upper} for Sub-3 then follows from
\eqref{Equation: Sub Known} with $p=1$, which implies that
\[\lim_{t\to\infty}\frac1t\log\mbb E_0\left[\mr e^{-tJ_0(L_t)}\right]=-\bs\chi^p.\]
Indeed, by a Brownian scaling,
\begin{multline*}
-tJ_0(L_t)=\frac{t^2}{t}\iint_{(\mbb R^d)^2}\ga(x-y)\d L_t(x)\dd L_t(y)
=\frac1t\iint_{[0,t]^2}\ga(W_r-W_s)\d r\dd s\\
\deq\iint_{[0,t^{(2-\om)/(4-\om)}]^2}\ga(W_r-W_s)\d r\dd s,
\end{multline*}
and
\[\big\langle u(t^{(2-\om)/(4-\om)},0)\big\rangle=\mbb E_0\left[\exp\left(\iint_{[0,t^{(2-\om)/(4-\om)}]^2}\ga(W_r-W_s)\d r\dd s\right)\right].\]

Next, we consider
\eqref{Equation: Averaged Asymptotics with alpha - Trace}.
By combining \cite[(2.7)--(2.10)]{GartnerKonig}
with the argument used to get to
\eqref{Equation: Lower Bound Scaling}, we know that for every $\de>0$,
if $\msf m$ is large enough (it suffices that $\be_\msf e(p(\msf t-\msf d))>\de$, which is always
possible since $\be_\msf e(p\msf t)\to\infty$
as $\msf m\to\infty$ in every case), then
\begin{multline}
\label{Equation: Trace to Average}
\left\langle\sum_{k=1}^\infty\mr e^{p(\msf t-\msf d)\la^{\xi_\msf e}_{k}(Q_{R\al_\msf e(p\msf t)})}\right\rangle\leq(4\pi\ka\de)^{-d/2}\mr e^{H_\msf e(p(\msf t-\msf d))}\\\cdot\al_\msf e(p(\msf t-\msf d))^d\int_{Q_{R+\theta}}\left\langle\mbb E_x\left[\mr e^{(\be_\msf e(p(\msf t-\msf d))-\de)(\Xi_{\msf e,\msf t-\msf d}^p,L_{\be_\msf e(p(\msf t-\msf d))-\de})}\mathbbm 1_{\{L_{\be_\msf e(p(\msf t-\msf d))-\de}\in\ms P(Q_{R+\theta})\}}\right]\right\rangle\d x.
\end{multline}
As subtracting
$\delta$ from $\be_{\msf e}(p(\msf t-\msf d))$ has no effect on the asymptotics,
\eqref{Equation: Averaged Asymptotics with alpha - Trace} follows
from \eqref{Equation: Averaged Asymptotics with alpha - Upper}.

We now conclude the proof of Lemma \ref{Lemma: Averaged Asymptotics 1} by
establishing Propositions \ref{Proposition: Smooth LDP} and \ref{Proposition: Sub-3 LDP}.

\begin{proof}[Proof of Proposition \ref{Proposition: Smooth LDP}]
We begin with the continuity of $\bs J^p$. Since we are not considering case Sub-3,
we need only prove the continuity of $J_c$ for $c\in(0,\infty]$.
The continuity of $J_\infty$ is established in \cite[(0.8) and the following paragraph; see also (4.4)]{GartnerKonig}.
Now suppose that $c\in(0,\infty)$. If $\mu_n$ converges weakly to $\mu$, then $\mu_n\otimes\mu_n$
converges weakly to $\mu\otimes\mu$. The continuity of $J_c$ then follows from the fact that
the two-dimensional function $(x,y)\mapsto\ga_c(x-y)$ is bounded and continuous when $c\in(0,\infty)$.

We now prove the asymptotic \eqref{Equation: Smooth LDP}.
Recall the definition of $\Xi_{\eps,t}^p$ in \eqref{Equation: Scaled Noise},
and the fact that $\xi_\eps$ has covariance $\ga_\eps=\eps^{-\om}\ga_1(\cdot/\eps)$ for $\eps>0$.
With this, it is easy to see by
a Gaussian moment generating function calculation that
\begin{align}
\label{Equation: General Exponential Moment}
\nonumber
\left\langle\mr e^{\be_\eps(pt)(\Xi_{\eps,t}^p,\mu)}\right\rangle
&=\mr e^{-\be_\eps(pt)\al_\eps(pt)^2H_\eps(pt)/pt}
\left\langle\mr e^{\be_\eps(pt)\al_\eps(pt)^2\big(\xi_\eps(\al_\eps(pt)\cdot),\mu\big)}\right\rangle\\
&=\exp\bigg(\be_\eps(pt)\bigg(-\frac{\al_\eps(pt)^2H_\eps(pt)}{pt}\\
\nonumber
&\qquad+\frac{\be_\eps(pt)\al_\eps(pt)^4\eps^{-\om}}{2}\int_{(\mbb R^d)^2}\ga_1\left(\frac{\al_\eps(pt)(x-y)}{\eps}\right)\d\mu(x)\dd\mu(y)
\bigg)\bigg).
\end{align}
From this point on, we argue on a case-by-case basis:

\noindent {\it Case Sub-1: }For sake of readability, let us denote
\[\msf h=\msf h(\msf m):=\msf e^{(\om-2)/4}(p\msf t)^{-1/4}=\big(\msf e (p\msf t)^{1/(2-\om)}\big)^{-(2-\om)/4}.\]
By definition of case Sub-1 (recall Assumption \ref{Assumption: Sub}), we note that $\msf h\to0$ as $\msf m\to\infty$.
By definition of $\al_\eps$, $\be_\eps$, and $H_\eps$ in case Sub-1
(recall Definition \ref{Definition: Table}), we also have that
\[\frac{\al_\msf e(p\msf t)^2H_\msf e(p\msf t)}{p\msf t}=\msf h^{-2}\frac{\ga_1(0)}{2},
\quad\frac{\be_\msf e(p\msf t)\al_\msf e(p\msf t)^4\msf e^{-\om}}2=\frac{\msf h^{-2}}{2},
\quad
\frac{\al_\msf e(p\msf t)}{\msf e}=\msf h.\]
Applying this to \eqref{Equation: General Exponential Moment}, and using \eqref{Equation: Averaged Asymptotics 1 - Condition on d} to
replace $F_\msf e(p(\msf t-\msf d))=F_\msf e(p\msf t)\big(1+o(1)\big)$
for $F=\al,\be,H$, we get that for Sub-1,
\begin{multline*}
\frac1{\be_\msf e(p(\msf t-\msf d))}\log\left\langle\mr e^{\be_\msf e(p(\msf t-\msf d))(\Xi_{\msf e,\msf t-\msf d}^p,\mu)}\right\rangle\\
=\frac{\msf h^{-2}}2(1+o(1))\left(-\ga_1(0)+\int_{(\mbb R^d)^2}\ga_1\big(\msf h(1+o(1))\,(x-y)\big)\d\mu(x)\dd\mu(y)\right).
\end{multline*}
Since $\ga_1(x)$ is smooth and has a strict maximum at zero,
we can approximate its values for small $x$ using a Taylor expansion of order 2
(e.g., \cite[Theorem 1.23]{Guler}), which yields
\[\ga_1(h x)=\ga_1(0)-h^2\frac{x^\top\Si x}2+h^{3}O\big(|x|^3\big),\qquad x\in\mbb R^d,~h>0.\]
Therefore, given that
\[\sup_{\mu\in\ms P(Q_R)}\iint_{(\mbb R^d)^2}O(|x|^3)\d\mu(x)\dd\mu(y)<\infty\]
for every $R>0$, we conclude that \eqref{Equation: Smooth LDP} holds for case Sub-1.

\noindent {\it Case Sub-2: }Let us denote
\[\msf c=\msf c(\msf m):=\msf e (p\msf t)^{1/(2-\om)}.\]
Since we are in case Sub-2, we note that $\msf c\to p^{1/(2-\om)}\mf c$
as $\msf m\to\infty$; moreover,
according to Definition \ref{Definition: Table} in that same case, we have that
\[\frac{\al_\msf e(p\msf t)^2H_\msf e(p\msf t)}{p\msf t}=0,
\qquad\frac{\be_\msf e(p\msf t)\al_\msf e(p\msf t)^4\msf e^{-\om}}2=\frac{\msf c^{-\om}}{2},
\qquad\text{and}\qquad
\frac{\al_\msf e(p\msf t)}{\msf e}=\frac1{\msf c}.\]
Thus, by \eqref{Equation: Averaged Asymptotics 1 - Condition on d} and \eqref{Equation: General Exponential Moment}, we have
\begin{align}
\nonumber
&\frac{1}{\be_\msf e(p(\msf t-\msf d))}\log\left\langle\mr e^{\be_\msf e(p(\msf t-\msf d))(\Xi_{\msf e,\msf t-\msf d}^p,\mu)}\right\rangle\\
\nonumber
&=\frac{1+o(1)}2\int_{(\mbb R^d)^2}\ga_{\msf c(1+o(1))}(x-y)\d\mu(x)\dd\mu(y)\\
\label{Equation: Gamma ceps versus Gamma c}
&=-J_{p^{1/(2-\om)}\mf c}(\mu)+o(1)+\frac12\int_{(\mbb R^d)^2}\Big(\ga_{\msf c(1+o(1))}(x-y)-\ga_{p^{1/(2-\om)}\mf c}(x-y)\Big)\d\mu(x)\dd\mu(y).
\end{align}
As $\ga_1$ is smooth and bounded and $\ga_c(x)=c^{-\om}\ga_1(x/c)$, one has $\ga_{\msf c(1+o(1))}\to\ga_{p^{1/(2-\om)}\mf c}$ uniformly on compacts
as $\msf m\to\infty$.
Thus, the integral in \eqref{Equation: Gamma ceps versus Gamma c} is uniformly small
over $\mu\in \ms P(Q_R)$ as $\msf m\to\infty$, concluding the proof of \eqref{Equation: Smooth LDP} in this case.

\noindent {\it Case Crt-1: }In this case, according to Definition \ref{Definition: Table}, we have that
\[\frac{\al_\msf e(p\msf t)^2H_\msf e(p\msf t)}{p\msf t}=(p\msf t)^{1/2}\frac{\ga_1(0)}{2},
\qquad\frac{\be_\msf e(p\msf t)\al_\msf e(p\msf t)^4\msf e^{-\om}}2=\frac{(p\msf t)^{1/2}}{2},
\qquad
\frac{\al_\msf e(p\msf t)}{\msf e}=(p\msf t)^{-1/4}.\]
The proof then follows from the same argument as for case Sub-1,
with the only difference that we replace the function denoted $\msf h$ therein by $(p\msf t)^{-1/4}$.

\noindent {\it Case Crt-2: }In this case, according to Definition \ref{Definition: Table}, we have that
\[\frac{\al_\msf e(p\msf t)^2H_\msf e(p\msf t)}{p\msf t}=0,
\qquad\frac{\be_\msf e(p\msf t)\al_\msf e(p\msf t)^4\msf e^{-\om}}2=\frac{p\msf t}2,
\qquad\text{and}\qquad
\frac{\al_\msf e(p\msf t)}{\msf e}=1.\]
Thus, by \eqref{Equation: Averaged Asymptotics 1 - Condition on d} and \eqref{Equation: General Exponential Moment}, we have
\[\frac{1}{\be_\msf e(p(\msf t-\msf d))}\log\left\langle\mr e^{\be_\msf e(p(\msf t-\msf d))(\Xi_{\msf e,\msf t-\msf d}^p,\mu)}\right\rangle=\frac{-p\msf t}{2}J_{1+o(1)}(\mu)\big(1+o(1)\big).\]
The result then follows from the facts that $\msf t\to t$ as $\msf m\to\infty$
(by definition of case Crt-2; see Assumption \ref{Assumption: Crt})
and, arguing as \eqref{Equation: Gamma ceps versus Gamma c},
\[\lim_{\msf m\to\infty}\sup_{\mu\in \ms P(Q_R)}\left|J_{1+o(1)}(\mu)-J_1(\mu)\right|=0.\]

\noindent {\it Case Sup: }This follows from the same argument
as Sub-1 with
\[\msf h=\msf h(\msf m):=\msf e^{(\om-2)/4}(p\msf t)^{-1/4}=\big(\msf e (p\msf t)^{-1/(\om-2)}\big)^{(\om-2)/4},\]
which vanishes as $\msf m\to\infty$ thanks to Assumption \ref{Assumption: Sup}.
With this case in hand, the proof of \eqref{Equation: Smooth LDP}, and thus Proposition \ref{Proposition: Smooth LDP}, is now complete.
\end{proof}

\begin{proof}[Proof of Proposition \ref{Proposition: Sub-3 LDP}]
In Case Sub-3, the same calculation as in \eqref{Equation: Gamma ceps versus Gamma c}
implies that
\begin{align*}
&\frac1{\be_\msf e(p(\msf t-\msf d))}\log\left\langle\mr e^{\be_\msf e(p(\msf t-\msf d))(\Xi_{\msf e,\msf t-\msf d}^p,L_{\be_\msf e(p(\msf t-\msf d))})}\right\rangle\\
&\qquad=\frac12\int_{(\mbb R^d)^2}\ga_{\msf e (p(\msf t-\msf d))^{1/(2-\om)}}(x-y)\d L_{\be_\msf e(p(\msf t-\msf d))}(x)\dd L_{\be_\msf e(p(\msf t-\msf d))}(y)
\\&\qquad=\frac{-J_{\msf e(p(\msf t-\msf d))^{1/(2-\om)}}(L_{\be_\msf e(p(\msf t-\msf d))})}{2}.
\end{align*}
By definition of case Sub-3 and \eqref{Equation: Averaged Asymptotics 1 - Condition on d}, $\msf e (p(\msf t-\msf d))^{1/(2-\om)}\to0$ as $\msf m\to\infty$;
thus the proof of \eqref{Equation: Sub-3 LDP 1} is simply a matter of noting that
for every $\eps>\eps'>0$ and $t>0$,
one has
\begin{align}
\label{Equation: Monotonicity of Jc}
-J_\eps(L_t)\leq -J_{\eps'}(L_t)\leq -J_0(L_t).
\end{align}
Letting $\widehat{\cdot}$ denote the Fourier transform, we have by definition of $L_t$, the Parseval formula,
and the convolution theorem that
\begin{align}
\label{Equation: JepsLt Fourier Transform}
-J_\eps(L_t)=\int_{\mbb R^d}\left|\frac1t\int_0^t\mr e^{-2\pi\mr i(W_s,x)}\d s\right|^2\widehat\ga(x)\widehat{p_\eps}(x)\d x.
\end{align}
Since $\ga$ is a covariance function, $\widehat\ga\geq0$.
Computing explicitly that
\[\widehat{p_\eps}(x)=\begin{cases}
\mr e^{-2\pi ^2 |x|^2\eps^2}&\eps>0\\
1&\eps=0
\end{cases},\]
and noting that this function is decreasing in $\eps$ for all $x$,
we therefore conclude \eqref{Equation: Monotonicity of Jc}.

We now prove \eqref{Equation: Sub-3 LDP 2}. Once again using Fourier transforms,
we note that
\begin{align}
\label{Equation: f^2 Fourier Transform}
-J_\eps(f^2)=\int_{\mbb R^d}|\widehat{f^2}(x)|^2\widehat\ga(x)\widehat{p_\eps}(x)\d x.
\end{align}
Consequently, $J_\eps(f^2)\geq J_0(f^2)$
for every function $f\in H^1(\mbb R^d)$ such that $\|f\|_2=1$ and $\eps>0$. With this in hand, in order
to prove \eqref{Equation: Sub-3 LDP 2}, it suffices to show that for every
such $f$, one has
\[\lim_{\eps\to0}J_\eps(f^2)=J_0(f^2).\]
This follows from \eqref{Equation: f^2 Fourier Transform} by monotone convergence.
\end{proof}

\subsection{Proof of Lemma \ref{Lemma: Averaged Asymptotics 2}}

We begin with \eqref{Equation: Averaged Asymptotics - Upper}.
Let us define $\tilde{\msf R}=\tilde{\msf R}(\msf m):=\msf R/\al_\msf e(p\msf t)$.
By applying the same scaling identities used to obtain \eqref{Equation: Lower Bound Scaling},
we get
\[\Big\langle\big( u^{\xi_\msf e}_{\msf R}(p\msf t,\cdot),\mathbbm 1\big)\Big\rangle
=\mr e^{H_\msf e(p\msf t)+o(\be_\msf e(p\msf t))}\Big\langle\big(u^{\Xi_{\msf e,\msf t}^p}_{\tilde{\msf R}}(p\msf t,\cdot),\mathbbm 1\big)\Big\rangle.\]
Thanks to \cite[(3.29), (3.31) and (3.32)]{GartnerKonig},
there exists some constant $K>0$ (independent of $\msf m$) such that
for every fixed $r\geq2$, one has
\[\mr e^{H_\msf e(p\msf t)}\Big\langle\big(u^{\Xi_{\msf e,\msf t}^p}_{\tilde{\msf R}}(p\msf t,\cdot),\mathbbm 1\big)\Big\rangle
\leq O(\tilde{\msf R}^d)\mr e^{K\be_\msf e(p\msf t)/r}\left\langle\sum_{k=1}^\infty\mr e^{p\msf t\la^{\xi_\msf e}_k(Q_{(r+1)\al_\msf e(p\msf t)})}\right\rangle\]
as $\msf m\to\infty$.
Since $\tilde{\msf R}^d=\mr e^{o(\be_\msf e(p\msf t))}$ (as $\al_\msf e(p\msf t),\msf R=\mr e^{o(\be_\msf e(p\msf t))}$),
an application of \eqref{Equation: Averaged Asymptotics with alpha - Trace} yields
\[\limsup_{\msf m\to\infty}\frac{\log\Big\langle\big(u^{\xi_\msf e}_{\be_\msf e(p\msf t)}(p\msf t,\cdot),\mathbbm 1\big)\Big\rangle-H_\msf e(p\msf t)}{\be_\msf e(p\msf t)}\leq\frac{K}{r}-\bs\chi^p.\]
Since $r\geq2$ was arbitrary, we can take $r\to\infty$, which then yields
\eqref{Equation: Averaged Asymptotics - Upper}.
We then obtain \eqref{Equation: Averaged Asymptotics - Trace}
from \eqref{Equation: Averaged Asymptotics - Upper} by using the same argument
as in \eqref{Equation: Trace to Average}, thus concluding the proof of Lemma \ref{Lemma: Averaged Asymptotics 2}.

\section{Proof of Theorem \ref{Theorem: All Moments} Part 2 - Lower Bound}
\label{Section: Moments Part 2}

In this section, we begin implementing the program outlined in 
Section \ref{Section: Three Remarks}, which allows us to turn Lemmas
\ref{Lemma: Averaged Asymptotics 1} and \ref{Lemma: Averaged Asymptotics 2}
into Theorem \ref{Theorem: All Moments}.
More specifically,
in this section we provide the following lower bound for Theorem \ref{Theorem: All Moments}:
Suppose that one of Assumptions \ref{Assumption: Sub}, \ref{Assumption: Crt}, or
\ref{Assumption: Sup} holds.
For every $p>0$ and $x\in\mbb R^d$, one has
\begin{align}
\label{Equation: All Moments Lower}
\liminf_{\msf m\to\infty}\frac{\log\big\langle u^{\xi_\msf e}(\msf t,x)^p\big\rangle-H_\msf e(p\msf t)}{\be_\msf e(p\msf t)}\geq-\bs\chi^p.
\end{align}
The remainder of this section is structured as follows:
In Sections \ref{Section: Outline for 0<p<1} and \ref{Section: Outline for pgeq1},
we provide an outline of the proof of \eqref{Equation: All Moments Lower}
for $p\geq1$ and $0<p<1$ respectively. This outline relies
on five technical results, namely, Lemmas
\ref{Lemma: Jensen Moment Lower Bound}--\ref{Lemma: Holder Variational Convergence R}.
We then end the section by proving these technical lemmas, in Sections
\ref{Section: Lower Bound Lemmas b}--\ref{Section: Lower Bound Lemmas e}.

\subsection{Outline for $p\geq1$}
\label{Section: Outline for pgeq1}

The proof of \eqref{Equation: All Moments Lower} for $p\geq1$ relies on the following:

\begin{lemma}
\label{Lemma: Jensen Moment Lower Bound}
For every $x\in\mbb R^d$, $p\geq1$, and $R>0$,
\[\big\langle u^{\xi_\msf e}(\msf t,x)^p\big\rangle\geq\mr e^{o(\be_\msf e(p\msf t))}\Big\langle\big( u^{\xi_\msf e}_{R\al_\msf e(p\msf t)}(\msf t,\cdot),\mathbbm 1\big)^p\Big\rangle\qquad\text{as }\msf m\to\infty.\]
\end{lemma}

\begin{lemma}
\label{Lemma: p to 1 Moment Lower Bound}
For every $p\geq1$ and $R>0$,
\begin{multline*}
\liminf_{\msf m\to\infty}\frac{\log\Big\langle\big( u^{\xi_\msf e}_{R\al_\msf e(p\msf t)}(\msf t,\cdot),\mathbbm 1\big)^p\Big\rangle-H_\msf e(p\msf t)}{\be_\msf e(p\msf t)}\\
\geq
p\big(\bs\chi^p-\bs\chi^p_R\big)+\liminf_{\msf m\to\infty}\frac{\log\Big\langle\big( u^{\xi_\msf e}_{R\al_\msf e(p\msf t)}(p\msf t,\cdot),\mathbbm 1\big)\Big\rangle-H_\msf e(p\msf t)}{\be_\msf e(p\msf t)},
\end{multline*}
where we recall the definition of $\bs\chi^p_R$ in \eqref{Equation: Chi_R}.
\end{lemma}

With these results in hand, \eqref{Equation: All Moments Lower} for $p\geq1$ is proved as follows:
If we combine Lemma \ref{Lemma: Jensen Moment Lower Bound}, Lemma \ref{Lemma: p to 1 Moment Lower Bound},
and \eqref{Equation: Averaged Asymptotics - Lower}, then we obtain that
\[\liminf_{\msf m\to\infty}\frac{\log\big\langle u^{\xi_\msf e}(\msf t,x)^p\big\rangle-H_\msf e(p\msf t)}{\be_\msf e(p\msf t)}
\geq p\big(\bs\chi^p-\bs\chi^p_R\big)-\bs\chi^p_R\]
for every $x\in\mbb R^d$, $p\geq1$ and $R>0$.
We then obtain \eqref{Equation: All Moments Lower} for all $p\geq1$ by taking $R\to\infty$
using the following:

\begin{lemma}
\label{Lemma: Lower Bound Variational Limit}
For every $p>0$ and $R'\in(0,\infty]$, one has
\[\lim_{R\nearrow R'}\bs\chi^p_R=\bs\chi^p_{R'},\]
where we use the convention that $\bs\chi^p_\infty=\bs\chi^p$.
\end{lemma}

\subsection{Outline for $0<p<1$}
\label{Section: Outline for 0<p<1}

The proof of \eqref{Equation: All Moments Lower} for $0<p<1$ relies on the following:

\begin{lemma}
\label{Lemma: Moment Lower Bound for 0<p<1 - 1}
Let $x\in\mbb R^d$, $0<p<1$, and $R>0$ be arbitrary.
On the one hand,
\begin{align}
\label{Equation: Moment Lower Bound for 0<p<1 - 1 Two Orders}
\big\langle u^{\xi_\msf e}(\msf t,x)^p\big\rangle\geq\mr e^{o(\be_\msf e(p\msf t))}\left(\Big\langle\big( u^{\xi_\msf e}_{R\al_\msf e(p\msf t)}(p(\msf t-\be_\msf e(p\msf t)^{-1}),\cdot),\mathbbm 1\big)\Big\rangle-\mr e^{pH_\msf e(\msf t)-\msf e^{-\om}\msf t^{2}\cdot\msf e^{-3/4}\msf t^{1/4}}\right)
\end{align}
as $\msf m\to\infty$ in cases Sub-1, Crt-1, and Sup. On the other hand,
in Cases Sub-2, Sub-3, and Crt-2
there exists a function $\msf d=\msf d(\msf m)$ satisfying \eqref{Equation: Averaged Asymptotics 1 - Condition on d}
such that for every $\theta>1$, one has
\begin{align}
\label{Equation: Moment Lower Bound for 0<p<1 - 1 One Order}
\big\langle u^{\xi_\msf e}(\msf t,x)^p\big\rangle\geq\mr e^{o(\be_\msf e(p\msf t))}\Big\langle\big( u^{\theta^{-1}\xi_\msf e}_{R\al_\msf e(p\msf t)}(p(\msf t-\msf d),\cdot),\mathbbm 1\big)\Big\rangle^\theta\qquad\text{as }\msf m\to\infty.
\end{align}
\end{lemma}

We may now prove \eqref{Equation: All Moments Lower} for all $0<p<1$:
Consider first Cases Sub-1, Crt-1, and Sup.
Since $\be_\msf e(p\msf t)^{-1}=o(1)$,
it is clear that $\msf d=\be_\msf e(p\msf t)^{-1}$ satisfies 
the first two limits in \eqref{Equation: Averaged Asymptotics 1 - Condition on d}.
As for the third, if we combine
\eqref{Equation: e bounded and t away from zero}, Definition \ref{Definition: Table},
and the fact that we always have $\msf e\to0$ or $\msf t\to\infty$,
\[\frac{H_\msf e(p(\msf t-\be_\msf e(p\msf t)^{-1}))-H_\msf e(p\msf t)}{\be_\msf e(p\msf t)}=\frac{\ga_1(0)\msf e^{\om/2+3}}{2 p^{5/2}\msf t^{9/2}}-\frac{\ga_1(0)\msf e^2}{p\msf t^2}=o(1).\]
Thus, if we apply \eqref{Equation: Averaged Asymptotics - Lower} with $\msf d=\be_\msf e(p\msf t)^{-1}$ to
\eqref{Equation: Moment Lower Bound for 0<p<1 - 1 Two Orders}, then we obtain the lower bound
\begin{multline}
\label{Equation: Moment Lower Bound for 0<p<1 - 1 Two Orders reduced}
\big\langle u^{\xi_\msf e}(\msf t,x)^p\big\rangle\geq\mr e^{o(\be_\msf e(p\msf t))}\left(\mr e^{H_\msf e(p\msf t)-\bs\chi^p_R\be_\msf e(p\msf t)(1+o(1))}-\mr e^{pH_\msf e(\msf t)-\msf e^{-\om}\msf t^{2}\cdot\msf e^{-3/4}\msf t^{1/4}}\right)\\
=\mr e^{H_\msf e(p\msf t)-\bs\chi^p_R\be_\msf e(p\msf t)(1+o(1))}\left(1-\mr e^{pH_\msf e(\msf t)-H_\msf e(p\msf t)+\bs\chi^p_R\be_\msf e(p\msf t)(1+o(1))-\msf e^{-\om}\msf t^{2}\cdot\msf e^{-3/4}\msf t^{1/4}}\right)
\end{multline}
for large $\msf m$.
In Cases Sub-1, Crt-1, and Sup, for every $q>0$, it is the case that
\begin{align}
\label{Equation: Faster than H and beta}
\lim_{\msf m\to\infty}\frac{H_\msf e(q\msf t)}{\msf e^{-\om}\msf t^{2}\cdot\msf e^{-3/4}\msf t^{1/4}}=0
\qquad\text{and}\qquad
\lim_{\msf m\to\infty}\frac{\be_\msf e(q\msf t)}{\msf e^{-\om}\msf t^{2}\cdot\msf e^{-3/4}\msf t^{1/4}}=0;
\end{align}
we therefore obtain \eqref{Equation: All Moments Lower} for all $0<p<1$ from \eqref{Equation: Moment Lower Bound for 0<p<1 - 1 Two Orders reduced}, as well as Lemma \ref{Lemma: Lower Bound Variational Limit}
to take the limit $R\to\infty$ in the variational constant $\bs\chi^p_R$.

We now conclude the proof of \eqref{Equation: All Moments Lower} for $0<p<1$
by dealing with cases Sub-2, Sub-3, and Crt-2. In these cases, we apply
\eqref{Equation: Averaged Asymptotics - Lower} (except that we replace
the noise $\xi_\msf e$ by $\theta^{-1}\xi_\msf e$) to \eqref{Equation: Moment Lower Bound for 0<p<1 - 1 One Order}.
This yields that for every $\theta>1$, one has
\[\liminf_{\msf m\to\infty}\frac{\log\big\langle u^{\xi_\msf e}(\msf t,x)^p\big\rangle}{\be_\msf e(p\msf t)}
\geq \theta\bs\chi^p_R(\theta^{-1}),\]
where for any constant $c>0$, we define
\[\bs\chi^p_R(c):=\inf_{\substack{f\in H^1(\mbb R^d),~\|f\|_2=1\\\mr{supp}(f)\subset Q_R}}\Big(\ka\ms S(f)+c^2\bs J^p(f^2)\Big).\]
Recalling that $H_\eps=0$ in cases Sub-2, Sub-3, and Crt-2,
we then obtain \eqref{Equation: All Moments Lower} by applying the following lemma,
together with the limits $R\to\infty$ and $\theta\searrow1$ (taken in that order):

\begin{lemma}
\label{Lemma: Holder Variational Convergence R}
In cases Sub-2, Sub-3, and Crt-2, for every $p>0$, we have that
\begin{align}
\label{Equation: Convergence of theta variational}
\lim_{c\to1}\lim_{R\to\infty}\bs\chi^p_R(c)=\bs\chi^p.
\end{align}
\end{lemma}

\subsection{Proof of Lemma \ref{Lemma: Jensen Moment Lower Bound}}
\label{Section: Lower Bound Lemmas b}

Since $u^{\xi_\eps}(t,\cdot)$ is stationary for every $\eps,t>0$, we can write
\begin{align}
\label{Equation: Stationary Integral}
\big\langle u^{\xi_\msf e}(\msf t,x)^p\big\rangle=\big\langle u^{\xi_\msf e}(\msf t,0)^p\big\rangle
=\left\langle\big(2R\al_\msf e(p\msf t)\big)^{-d}
\int_{Q_{R\al_\msf e(p\msf t)}}u^{\xi_\msf e}(\msf t,y)^p\d y\right\rangle.
\end{align}
If $p\geq1$, then Jensen's inequality yields
\[\big(2R\al_\msf e(p\msf t)\big)^{-d}
\int_{Q_{R\al_\msf e(p\msf t)}}u^{\xi_\msf e}(\msf t,y)^p\d y
\geq\left(\big(2R\al_\msf e(p\msf t)\big)^{-d}
\int_{Q_{R\al_\msf e(p\msf t)}}u^{\xi_\msf e}(\msf t,y)\d y\right)^p.\]
We then get Lemma \ref{Lemma: Jensen Moment Lower Bound} by
noting that $u^{\xi_\msf e}(\msf t,\cdot)\geq u^{\xi_\msf e}_{R\al_\msf e(p\msf t)}(\msf t,\cdot)$
and $\al_\msf e(p\msf t)=\mr e^{o(\be_\msf e(p\msf t))}$.

\subsection{Proof of Lemma \ref{Lemma: p to 1 Moment Lower Bound}}

By \cite[(2.11) and (2.12)]{GartnerKonig} , if  $p\geq1$, then we have that
\[\Big\langle\big( u^{\xi_\msf e}_{R\al_\msf e(p\msf t)}(\msf t,\cdot),\mathbbm 1\big)^p\Big\rangle
\geq(2R\al_\msf e(p\msf t))^{-d}\left(\frac{\Big\langle\big( u^{\xi_\msf e}_{R\al_\msf e(p\msf t)}(p\msf t,\cdot),\mathbbm 1\big)\Big\rangle}{\left\langle\sum_{k=1}^\infty\mr e^{p\msf t\la^{\xi_\msf e}_{k}(Q_{R\al_\msf e(p\msf t)})}\right\rangle}\right)^p\Big\langle\big( u^{\xi_\msf e}_{R\al_\msf e(p\msf t)}(p\msf t,\cdot),\mathbbm 1\big)\Big\rangle.\]
We then obtain Lemma \ref{Lemma: p to 1 Moment Lower Bound} by an application of \eqref{Equation: Averaged Asymptotics - Lower}
and \eqref{Equation: Averaged Asymptotics - Trace}, and $\al_\msf e(p\msf t)=\mr e^{o(\be_\msf e(p\msf t))}$.

\subsection{Proof of Lemma \ref{Lemma: Lower Bound Variational Limit}}

Note that $\bs\chi^p_{R}\geq\bs\chi^p_{R'}$ whenever $0<R<R'$. Thus, it suffices to show that for every $f\in H^1(\mbb R^d)$ such that
$\|f\|_2=1$, $\mr{supp}(f)\subset Q_{R'}$ (using the convention that $Q_\infty=\mbb R^d$),
and $\ka\ms S(f)+\bs J^p(f^2)$ is finite, there exists a sequence $f_R$ of functions such that for each $R$,
one has $\mr{supp}(f)\subset Q_R$ and $\|f_R\|_2=1$, and moreover
\[\lim_{R\nearrow R'}=\ka\ms S(f_R)+\bs J^p(f_R^2)=\ka\ms S(f)+\bs J^p(f^2).\]
This follows from the standard proof that smooth and compactly supported functions are dense in $H^1(\mbb R^d)$
using smooth cutoff functions, together with the monotone convergence theorem for the convergence
$\bs J^p(f^2_R)\to \bs J^p(f^2)$.

\subsection{Proof of Lemma \ref{Lemma: Moment Lower Bound for 0<p<1 - 1}}

We begin with a general proposition:

\begin{proposition}
\label{Proposition: Average Lower Bound for 0<p<1}
Let $0<p<1$, $R>0$, and $\theta\geq1$. For any function $\msf d=\msf d(\msf m)$ such that $0\leq\msf d<\msf t$, one has
\[\Big\langle\big( u^{\theta^{-1}\xi_\msf e}_{R\al_\msf e(p\msf t)}(\msf t-\msf d,\cdot),\mathbbm 1\big)^p\Big\rangle
\geq\mr e^{o(\be_\msf e(p\msf t))}\Big\langle\big( u^{\theta^{-1}\xi_\msf e}_{R\al_\msf e(p\msf t)}(p(\msf t-\msf d),\cdot),\mathbbm 1\big)\Big\rangle\qquad\text{as }\msf m\to\infty.\]
\end{proposition}
\begin{proof}
For any $\theta\geq1$, the function
\[f(k):=\frac{\big(e^{\theta^{-1}\xi_\msf e}_k(Q_{R\al_\msf e(p\msf t)}),\mathbbm 1\big)^2}{(2R\al_\msf e(p\msf t))^{2d}},\qquad k\in\mbb N\]
is a probability measure on $\mbb N$.
If we use a spectral expansion and then apply Jensen's inequality
to this probability measure (since $0<p<1$), we are led to
\begin{align*}
\big(u^{\theta^{-1}\xi_\msf e}_{R\al_\msf e(p\msf t)}(\msf t-\msf d,\cdot),\mathbbm 1\big)^p
&=\left(\sum_{k=1}^\infty\mr e^{(\msf t-\msf d)\la^{\theta^{-1}\xi_\msf e}_k(Q_{R\al_\msf e(p\msf t)})}\big(e^{\theta^{-1}\xi_\msf e}_k(Q_{R\al_\msf e(p\msf t)}),\mathbbm 1\big)^2\right)^p
\\&\geq(2R\al_\msf e(p\msf t))^{2d(p-1)}\sum_{k=1}^\infty\mr e^{p(\msf t-\msf d)\la^{\theta^{-1}\xi_\msf e}_k(Q_{R\al_\msf e(p\msf t)})}\big(e^{\theta^{-1}\xi_\msf e}_k(Q_{R\al_\msf e(p\msf t)}),\mathbbm 1\big)^2\\
&=\mr e^{o(\be_\msf e(p\msf t))}\big(u^{\theta^{-1}\xi_\msf e}_{R\al_\msf e(p\msf t)}(p(\msf t-\msf d),\cdot),\mathbbm 1\big),
\end{align*}
concluding the proof.
\end{proof}

For every $0\leq\msf d<\msf t$ and $r>0$, we denote the event
\begin{align}
\label{Equation: A Event}
A_{\msf d,\msf t}(Q_{r}):=\{W_s\in Q_{r}\text{ for every }s\in[\msf d,\msf t]\}.
\end{align}
By stationarity and the Feynman-Kac formula, we have that
\[\big\langle u^{\xi_\msf e}(\msf t,x)^p\big\rangle=\big\langle u^{\xi_\msf e}(\msf t,0)^p\big\rangle\geq
\mbb E_0\left[\exp\left(\int_{0}^{\msf t}\xi_\msf e(W_s)\d s\right)\mathbbm 1_{A_{\msf d,\msf t}(Q_{R\al_\msf e(p\msf t)})}\right].\]
From this point on, we split the proof into two steps.

\subsubsection{Proof of \eqref{Equation: Moment Lower Bound for 0<p<1 - 1 Two Orders}}

Suppose that we are in one of cases Sub-1, Crt-1, or Sup.
For simplicity, throughout the proof of \eqref{Equation: Moment Lower Bound for 0<p<1 - 1 Two Orders}, we denote the functions
\begin{align}
\label{Equation: d, c and h}
\msf d=\be_\msf e(p\msf t)^{-1},\qquad
\msf c=\msf e^{1/4}\msf t^{-1/4}\be_\msf e(p\msf t)
=p^{3/2}\msf e^{-(3+2\om)/4}\msf t^{5/4},\qquad
\text{and}\qquad
\msf h=\frac{\msf e^{-\om/2}\msf t}{p^{5/2}}.
\end{align}
Then,
\begin{align*}
&\mbb E_0\left[\exp\left(\int_{0}^{\msf t}\xi_\msf e(W_s)\d s\right)\mathbbm 1_{A_{\msf d,\msf t}(Q_{R\al_\msf e(p\msf t)})}\right]\\
&\geq\mbb E_0\left[\exp\left(\int_{0}^{\msf t}\xi_\msf e(W_s)\d s\right)\mathbbm 1_{A_{\msf d,\msf t}(Q_{R\al_\msf e(p\msf t)})}\mathbbm 1_{\{\int_0^{\msf d}\xi_\msf e(W_s)\d s\geq-\msf c\}}\right]\\
&\geq\mr e^{-\msf c}\mbb E_0\left[\exp\left(\int_{\msf d}^{\msf t}\xi_\msf e(W_s)\d s\right)\mathbbm 1_{A_{\msf d,\msf t}(Q_{R\al_\msf e(p\msf t)})}\mathbbm 1_{\{\int_0^{\msf d}\xi_\msf e(W_s)\d s\geq-\msf c\}}\right]\\
&=\mr e^{-\msf c}\mbb E_0\left[\exp\left(\int_{\msf d}^{\msf t}\xi_\msf e(W_s)\d s\right)\mathbbm 1_{A_{\msf d,\msf t}(Q_{R\al_\msf e(p\msf t)})}\right]\\
&\qquad-\mr e^{-\msf c}\mbb E_0\left[\exp\left(\int_{\msf d}^{\msf t}\xi_\msf e(W_s)\d s\right)\mathbbm 1_{A_{\msf d,\msf t}(Q_{R\al_\msf e(p\msf t)})}\mathbbm 1_{\{\int_0^{\msf d}\xi_\msf e(W_s)\d s<-\msf c\}}\right].
\end{align*}
Since it is always the case that $\msf e\to0$ or $\msf t\to\infty$,
one has $\msf c=o(\be_\msf e(p\msf t))$. Therefore,
given Proposition \ref{Proposition: Average Lower Bound for 0<p<1} and
the fact that $(x-y)^p\geq x^p-y^p$ for all $0<y<x$ when $0<p<1$,
in order to prove \eqref{Equation: Moment Lower Bound for 0<p<1 - 1 Two Orders} it suffices to show that as $\msf m\to\infty$, one has
\begin{align}
\label{Equation: Moment Lower Bound for 0<p<1 - 1 Two Orders 1}
\left\langle\mbb E_0\left[\exp\left(\int_{\msf d}^{\msf t}\xi_\msf e(W_s)\d s\right)\mathbbm 1_{A_{\msf d,\msf t}(Q_{R\al_\msf e(p\msf t)})}\right]^p\right\rangle\geq
\mr e^{o(\be_\msf e(p\msf t))}\Big\langle\big( u^{\xi_\msf e}_{R\al_\msf e(p\msf t)}(\msf t-\msf d,\cdot),\mathbbm 1\big)^p\Big\rangle
\end{align}
and
\begin{align}
\label{Equation: Moment Lower Bound for 0<p<1 - 1 Two Orders 2}
\left\langle\mbb E_0\left[\exp\left(\int_{\msf d}^{\msf t}\xi_\msf e(W_s)\d s\right)\mathbbm 1_{A_{\msf d,\msf t}(Q_{R\al_\msf e(p\msf t)})}\mathbbm 1_{\{\int_0^{\msf d}\xi_\msf e(W_s)\d s<-\msf c\}}\right]^p\right\rangle\leq
\mr e^{-p\msf h\msf c+pH_\msf e(\msf t)}.
\end{align}
(Indeed, $p\msf h\msf c=\msf e^{-\om}\msf t^{2}\cdot\msf e^{-3/4}\msf t^{1/4}$.)

We begin with \eqref{Equation: Moment Lower Bound for 0<p<1 - 1 Two Orders 1}.
By the strong Markov property and definition of $A_{\msf d,\msf t}(Q_{R\al_\msf e(p\msf t)})$,
\[\mbb E_0\left[\exp\left(\int_{\msf d}^{\msf t}\xi_\msf e(W_s)\d s\right)\mathbbm 1_{A_{\msf d,\msf t}(Q_{R\al_\msf e(p\msf t)})}\right]
=\int_{R\al_\msf e(p\msf t)}\Pi_\msf d(x)u^{\xi_\msf e}_{R\al_\msf e(p\msf t)}(\msf t-\msf d,x)\d x,\]
where $(\Pi_c)_{c>0}$ is $W$'s transition density, that is,
\[\Pi_c(x):=\frac{\mr e^{-|x|^2/4\ka c}}{(4\ka c\pi)^{d/2}},\qquad c>0,~x\in\mbb R^d.\]
We then obtain \eqref{Equation: Moment Lower Bound for 0<p<1 - 1 Two Orders 1} by
noting that, since $\msf d=\mr e^{o(\be_\msf e(p\msf t))}$ and $\al_\msf e(p\msf t)^2/\msf d=o(\be_\msf e(p\msf t))$,
one has
\[\inf_{x\in Q_{R\al_\msf e(p\msf t)}}\Pi_\msf d(x)\geq\mr e^{o(\be_\msf e(p\msf t))}.\]

We now prove \eqref{Equation: Moment Lower Bound for 0<p<1 - 1 Two Orders 2}.
Given that
\[\lim_{\msf m\to\infty}(\msf h-1)\msf d=\lim_{\msf m\to\infty}\big(p^{-4}\msf e\msf t^{-1/2}-\msf e^{(2+\om)/2}(p\msf t)^{-3/2}\big)=0,\]
for large enough $\msf m$ we have that $(\msf h-1)\msf d\leq \msf t$ (recall that $\inf_{\msf m\geq0} \msf t>0$). Therefore,
\begin{align*}
&\mbb E_0\left[\exp\left(\int_{\msf d}^{\msf t}\xi_\msf e(W_s)\d s\right)\mathbbm 1_{A_{\msf d,\msf t}(Q_{R\al_\msf e(p\msf t)})}\mathbbm 1_{\{\int_0^{\msf d}\xi_\msf e(W_s)\d s<-\msf c\}}\right]\\
&\leq\mbb E_0\left[\exp\left(\int_{\msf d}^{\msf t}\xi_\msf e(W_s)\d s\right)\mathbbm 1_{\{\int_0^{\msf d}\xi_\msf e(W_s)\d s<-\msf c\}}\right]\\
&\leq\mr e^{-\msf h\msf c}\mbb E_0\left[\exp\left(\int_0^{\msf t}\big(1-(\msf h-1)\mathbbm1_{\{0\leq s\leq\msf d\}}\big)\xi_\msf e(W_s)\d s\right)\right].
\end{align*}
If we apply Jensen's inequality to the probability measure with density function
\[s\mapsto\frac{\big(1-(\msf h-1)\mathbbm1_{\{0\leq s\leq\msf d\}}\big)}{\big(\msf t-(\msf h-1)\msf d\big)},\qquad s\in[0,\msf t],\]
then we are led to
\begin{multline}
\label{Equation: Upper Bound Without Holder}
\mbb E_0\left[\exp\left(\int_{\msf d}^{\msf t}\xi_\msf e(W_s)\d s\right)\mathbbm 1_{A_{\msf d,\msf t}(Q_{R\al_\msf e(p\msf t)})}\mathbbm 1_{\{\int_0^{\msf d}\xi_\msf e(W_s)\d s<-\msf c\}}\right]\\
\leq\mr e^{-\msf h\msf c}\int_0^{\msf t}\frac{\big(1-(\msf h-1)\mathbbm1_{\{0\leq s\leq\msf d\}}\big)}{\big(\msf t-(\msf h-1)\msf d\big)}\mbb E_0\left[\mr e^{(\msf t-(\msf h-1)\msf d)\xi_\msf e(W_s)}\right]\d s.
\end{multline}
At this point, we apply Jensen's inequality to the expectation $\langle(\cdot)^p\rangle\leq\langle\cdot\rangle^p$
(as $0<p<1$) together with the fact that $\xi_\msf e$ is stationary with cumulant generating function $H_\msf e$ to obtain
\[\left\langle\mbb E_0\left[\exp\left(\int_{\msf d}^{\msf t}\xi_\msf e(W_s)\d s\right)\mathbbm 1_{A_{\msf d,\msf t}(Q_{R\al_\msf e(p\msf t)})}\mathbbm 1_{\{\int_0^{\msf d}\xi_\msf e(W_s)\d s<-\msf c\}}\right]^p\right\rangle\\
\leq\mr e^{-p\msf h\msf c+pH_\msf e(\msf t-(\msf h-1)\msf d)}.\]
The inequality \eqref{Equation: Moment Lower Bound for 0<p<1 - 1 Two Orders 2}
then follows from $H_\msf e(\msf t-(\msf h-1)\msf d)\leq H_\msf e(\msf t)$.

\subsubsection{Proof of \eqref{Equation: Moment Lower Bound for 0<p<1 - 1 One Order}}

By the reverse H\"older inequality,
for every $\theta>1$ and $0<p<1$, one has
\begin{align*}
&\left\langle\mbb E_0\left[\exp\left(\int_0^{\msf t}\xi_\msf e(W_s)\d s\right)\mathbbm 1_{A_{\msf d,\msf t}(Q_{R\al_\msf e(p\msf t)})}\right]^p\right\rangle\\
&\geq\left\langle\mbb E_0\left[\exp\left(\int_0^\msf d \theta'\xi_\msf e(W_s)\d s\right)\right]^{p/\theta'}
\mbb E_0\left[\exp\left(\int_{\msf d}^{\msf t}\theta^{-1}\xi_\msf e(W_s)\d s\right)\mathbbm 1_{A_{\msf d,\msf t}(Q_{R\al_\msf e(p\msf t)})}\right]^{\theta p}
\right\rangle\\
&\geq\left\langle\mbb E_0\left[\exp\left(\int_0^\msf d \theta'\xi_\msf e(W_s)\d s\right)\right]^p\right\rangle^{1/\theta'}\left\langle\mbb E_0\left[\exp\left(\int_{\msf d}^{\msf t}\theta^{-1}\xi_\msf e(W_s)\d s\right)\mathbbm 1_{A_{\msf d,\msf t}(Q_{R\al_\msf e(p\msf t)})}\right]^{p}
\right\rangle^{\theta},\\
&\geq\left\langle\mbb E_0\left[\exp\left(\int_0^\msf d \theta'\xi_\msf e(W_s)\d s\right)\right]\right\rangle^{p/\theta'}\left\langle\mbb E_0\left[\exp\left(\int_{\msf d}^{\msf t}\theta^{-1}\xi_\msf e(W_s)\d s\right)\mathbbm 1_{A_{\msf d,\msf t}(Q_{R\al_\msf e(p\msf t)})}\right]^{p}
\right\rangle^{\theta},
\end{align*}
where $\theta'=\frac{-1}{\theta-1}$ and the last
inequality follows from an application of Jensen's inequality to
to the expectation $\langle(\cdot)^p\rangle\leq\langle\cdot\rangle^p$.
Thus, by arguing
as in \eqref{Equation: Moment Lower Bound for 0<p<1 - 1 Two Orders 1}
(except that $\xi_\msf e$ is replaced by $\theta^{-1}\xi_\msf e$),
so long as we choose $\msf d$ in such a way that
$\msf d=\mr e^{o(\be_\msf e(p\msf t))}$ and $\al_\msf e(p\msf t)^2/\msf d=o(\be_\msf e(p\msf t))$
in addition to \eqref{Equation: Averaged Asymptotics 1 - Condition on d},
it suffices to show that
\begin{align}
\label{Equation: Moment Lower Bound for 0<p<1 - 1 One Order Bounded Part}
\left\langle\mbb E_0\left[\exp\left(\int_0^\msf d \theta'\xi_\msf e(W_s)\d s\right)\right]\right\rangle\leq\mr e^{o(\be_\msf e(p\msf t))}.
\end{align}

On the one hand, in cases Sub-2 and Sub-3, it suffices to take any
small enough constant $\msf d:=\de$ (so that $\inf_{\msf m\geq0}\msf t>\de$):
Since $\msf t\to\infty$ and $H_\msf e=0$ in cases Sub-2 and Sub-3, it is easily checked using
Definition \ref{Definition: Table} that \eqref{Equation: Averaged Asymptotics 1 - Condition on d}
is met for that choice.
Moreover, $\al_\msf e(p\msf t)^2/\de=o(1)=o(\be_\msf e(p\msf t))$.
Finally, in the subcritical regime, \eqref{Equation: Subcritical Convergence} implies that
\[\lim_{\msf m\to\infty}\left\langle\mbb E_0\left[\exp\left(\int_0^{\de}\theta'\xi_\msf e(W_s)\d s\right)\right]\right\rangle=\left\langle\mbb E_0\left[\exp\left(\int_0^{\de}\theta'\xi(W_s)\d s\right)\right]\right\rangle<\infty;\]
hence \eqref{Equation: Moment Lower Bound for 0<p<1 - 1 One Order Bounded Part} holds.

On the other hand, in case Crt-2, we note by Fubini's theorem that
\begin{multline*}
\left\langle\mbb E_0\left[\exp\left(\int_0^{\msf d} \theta'\xi_\msf e(W_s)\d s\right)\right]\right\rangle\\
=\mbb E_0\left[\exp\left(\frac{(\theta'\msf e)^2}{2}\int_{[0,\msf d]^2}\ga_1\big((W_u-W_v)/\msf e\big)\d u\dd v\right)\right]\leq\mr e^{(\theta')^2\msf d^{2}\msf e^{-2}\|\ga_1\|_\infty/2};
\end{multline*}
with this inequality in hand, we can take, for example, $\msf d=\msf e^{3/2}=\mr e^{o(\be_\msf e(p\msf t))}$.
Indeed, since $\msf e^{3/2}=o(1)$ in that case,
it is clear by Definition \ref{Definition: Table} that \eqref{Equation: Averaged Asymptotics 1 - Condition on d}
holds and that $\al_\msf e(p\msf t)^2/\msf d=\msf e^{1/2}=o(1)=o(\be_\msf e(p\msf t))$.
This concludes the proof of \eqref{Equation: Moment Lower Bound for 0<p<1 - 1 One Order},
and thus also of Lemma \ref{Lemma: Moment Lower Bound for 0<p<1 - 1}.

\subsection{Proof of Lemma \ref{Lemma: Holder Variational Convergence R}}
\label{Section: Lower Bound Lemmas e}

First, Lemma \ref{Lemma: Lower Bound Variational Limit}
(where we replace $\xi_\msf e$ by $c\xi_\msf e$) implies that
\[\lim_{c\to1}\lim_{R\to\infty}\bs\chi^p_R(c)=\lim_{c\to1}\bs\chi^p(c),\]
where we define
\[\bs\chi^p(c):=\inf_{f\in H^1(\mbb R^d),~\|f\|_2=1}\Big(\ka\ms S(f)+c^2\bs J^p(f^2)\Big).\]
Thus, it now suffices to prove that
\begin{align}
\label{Equation: Convergence of theta variational}
\lim_{c\to1}\bs\chi^p(c)=\bs\chi^p
\end{align}
for all $p>0$ in cases Sub-2, Sub-3, and Crt-2.

We begin by proving \eqref{Equation: Convergence of theta variational} in cases Sub-2 and Crt-2,
in which case we have that $\bs J^p=\ze(p)J_\eps$ for some $0<\eps<\infty$
(recall Defintion \ref{Definition: Table}) and constant $\ze(p)>0$.
Since $J_\eps\leq0$ for all $0<\eps<\infty$, we have in
those cases that $\bs\chi^p(c)\leq\bs\chi^p$ when $c\geq1$
and
$\bs\chi^p(c)\geq\bs\chi^p$ when $c\leq1$.
Given that $\ga_\eps$ is bounded
\[J_\eps(f^2)=-\frac12\iint_{(\mbb R^d)^2}f(x)^2\ga_\eps(x-y)f(y)^2\d x\dd y\geq-\frac{\|\ga_\eps\|_\infty\|f\|_2^4}2=-\frac{\|\ga_\eps\|_\infty}2\]
for every $f$ such that $\|f\|_2=1$.
Thus, if $c\geq1$, then
\[\bs\chi^p\geq\bs\chi^p(c)=\inf_{f\in H^1(\mbb R^d),~\|f\|_2=1}\Big(\ka\ms S(f)+\bs J^p(f^2)+(c^2-1)\bs J^p(f^2)\Big)\geq\bs\chi^p-(c^2-1)\ze(p)\frac{\|\ga_\eps\|_\infty}2;\]
similarly, if $c\leq1$, then we have that
\[\bs\chi^p\leq\bs\chi^p(c)\leq\bs\chi^p+(1-c^2)\ze(p)\frac{\|\ga_\eps\|_\infty}2.\]
This implies \eqref{Equation: Convergence of theta variational}.

We now consider case Sub-3, wherein we recall that
\[\bs J^p(f^2)=-\frac12\iint_{\mbb R^d}f(x)^2\ga(x-y)f(y)^2\d x\dd y.\]
By using the scaling property $\ga(cx)=c^{-\om}\ga(x)$ for all $c>0$,
it can be shown (e.g., \cite[(1.13)]{ChenHuSongSong}) that
$\bs\chi^p(c)=c^{4/(2-\om)}\bs\chi^p$ for every $c>0$, from which
\eqref{Equation: Convergence of theta variational} immediately follows.

\section{Proof of Theorem \ref{Theorem: All Moments} Part 3 - Upper Bound}
\label{Section: Moments Part 3}

In this section, we prove the following upper bound,
which, together with \eqref{Equation: All Moments Lower},
concludes the proof of Theorem \ref{Theorem: All Moments}:
Suppose that one of Assumptions \ref{Assumption: Sub}, \ref{Assumption: Crt}, or
\ref{Assumption: Sup} holds.
For every $p>0$ and $x\in\mbb R^d$, one has
\begin{align}
\label{Equation: All Moments Upper}
\limsup_{\msf m\to\infty}\frac{\log\big\langle u^{\xi_\msf e}(\msf t,x)^p\big\rangle-H_\msf e(p\msf t)}{\be_\msf e(p\msf t)}\leq-\bs\chi^p.
\end{align}
The remainder of this section is structured as follows:
In Section \ref{Section: Upper Outline},
we provide an outline of the proof of \eqref{Equation: All Moments Upper}.
This outline relies on two technical results, namely, Lemmas
\ref{Lemma: Moment Upper Bound 1} and \ref{Lemma: Moment Upper Bound 2}.
These are proved in Sections \ref{Section: Upper 1} and \ref{Section: Upper 2},
respectively.

\subsection{Outline}
\label{Section: Upper Outline}
The proof of \eqref{Equation: All Moments Upper} relies on the following two lemmas:

\begin{lemma}
\label{Lemma: Moment Upper Bound 1}
Let $x\in\mbb R^d$ and $p>0$.
There exists a function $\msf R$ that satisfies $\al_\msf e(p\msf t)=o(\msf R)$
and $\msf R=\mr e^{o(\be_\msf e(p\msf t))}$ such that
\begin{align}
\label{Equation: Box Upper Bound}
\big\langle u^{\xi_\msf e}(\msf t,x)^p\big\rangle\leq\big(1+o(1)\big)\big\langle u^{\xi_\msf e}_{\msf R}(\msf t,x)^p\big\rangle\qquad\text{as }\msf m\to\infty.
\end{align}
\end{lemma}

\begin{lemma}
\label{Lemma: Moment Upper Bound 2}
Let $x\in\mbb R^d$, $p>0$, and $\msf R=\mr e^{o(\be_\msf e(p\msf t))}$. On the one hand, as $\msf m\to\infty$,\begin{align}
\label{Equation: Averaged Noise Upper Bound - Regular}
\big\langle u^{\xi_\msf e}_{\msf R}(\msf t,x)^p\big\rangle\leq
\begin{cases}\mr e^{o(\be_\msf e(p\msf t))}\left(\Big\langle\big(u^{\xi_\msf e}_{\msf R}(p\msf t,\cdot),\mathbbm 1\big)\Big\rangle+1\right)+\mr e^{4H_\msf e(p\msf t)-\msf e^{-\om}\msf t^{2}\cdot\msf e^{-3/4}\msf t^{1/4}}&\text{if }p\geq1
\vspace{3pt}\\
\mr e^{o(\be_\msf e(p\msf t))}\left(\left\langle\sum_{k=1}^\infty\mr e^{p\msf t\la^{\xi_\msf e}_k(Q_{\msf R})}\right\rangle+1\right)+\mr e^{4pH_\msf e(\msf t)-\msf e^{-\om}\msf t^{2}\cdot\msf e^{-3/4}\msf t^{1/4}}&\text{if }0<p<1
\end{cases}
\end{align}
in cases Sub-1, Crt-1, and Sup. On the other hand, for every $\theta>1$, it holds as $\msf m\to\infty$ that
\begin{align}
\label{Equation: Averaged Noise Upper Bound - Singular}
\big\langle u^{\xi_\msf e}_{\msf R}(t,x)^p\big\rangle\leq
\begin{cases}\mr e^{o(\be_\msf e(p\msf t))}\left(\Big\langle\big(u^{\theta\xi_\msf e}_{\msf R}(p\msf t,\cdot),\mathbbm 1\big)\Big\rangle+1\right)^{1/\theta}&\text{if }p\geq1\\
\mr e^{o(\be_\msf e(p\msf t))}\left(\left\langle\sum_{k=1}^\infty\mr e^{p\msf t\la^{\theta\xi_\msf e}_k(Q_{\msf R})}\right\rangle+1\right)^{1/\theta}&\text{if }0<p<1
\end{cases}
\end{align}
in cases Sub-2, Sub-3, and Crt-2.
\end{lemma}

We may now prove \eqref{Equation: All Moments Upper}.
On the one hand, in cases Sub-1, Crt-1, and Sup, a combination of
\eqref{Equation: Averaged Asymptotics - Upper}/\eqref{Equation: Averaged Asymptotics - Trace},
\eqref{Equation: Faster than H and beta},
\eqref{Equation: Box Upper Bound}, and
\eqref{Equation: Averaged Noise Upper Bound - Regular}
implies \eqref{Equation: All Moments Upper}. On the other hand,
in cases Sub-2, Sub-3, and Crt-2 a combination of
\eqref{Equation: Averaged Asymptotics - Upper}/\eqref{Equation: Averaged Asymptotics - Trace}
(where we replace $\xi_\msf e$ by $\theta\xi_\msf e$),
\eqref{Equation: Box Upper Bound}, and
\eqref{Equation: Averaged Noise Upper Bound - Singular}
implies that for every $\theta>1$,
\begin{align*}
\limsup_{\msf m\to\infty}\log\frac{\big\langle u^{\xi_\msf e}(\msf t,x)^p\big\rangle}{\be_\msf e(p\msf t)}\leq-\frac{\bs\chi^p(\theta)}{\theta},
\end{align*}
where we recall that
\[\bs\chi^p(\theta):=\inf_{f\in H^1(\mbb R^d),~\|f\|_2=1}\Big(\ka\ms S(f)+\theta^2\bs J^p(f^2)\Big).\]
We then obtain \eqref{Equation: All Moments Upper} by taking the limit $\theta\to1$
using \eqref{Equation: Convergence of theta variational}.

\subsection{Proof of Lemma \ref{Lemma: Moment Upper Bound 1}}
\label{Section: Upper 1}

We claim that it suffices to show that
\begin{align}
\label{Equation: Moment Upper Bound 1 Reduction}
\lim_{\msf m\to\infty}\frac{\big\langle\big(u^{\xi_\msf e}(\msf t,0)-u^{\xi_\msf e}_{\msf R}(\msf t,0)\big)^p\big\rangle}{\langle u^{\xi_\msf e}(\msf t,0)^p\rangle}=0.
\end{align}
On the one hand, for $p\geq1$, as remarked in \cite[(3.16)]{GartnerKonig}, combining the triangle inequality with \eqref{Equation: Moment Upper Bound 1 Reduction}
and the fact that $u^{\xi_\eps}(t,x)\deq u^{\xi_\eps}(t,0)\geq u^{\xi_\eps}_{R}(t,0)\geq0$
implies that
\[0\leq\big\langle u^{\xi_\msf e}(\msf t,0)^p\big\rangle^{1/p}-\big\langle u^{\xi_\msf e}_{\msf R}(\msf t,0)^p\big\rangle^{1/p}
\leq\big\langle\big(u^{\xi_\msf e}(\msf t,0)-u^{\xi_\msf e}_{\msf R}(\msf t,0)\big)^p\big\rangle^{1/p}=o\big(\langle u^{\xi_\msf e}(\msf t,0)^p\rangle^{1/p}\big).\]
If we divide all terms in the above inequality by $\langle u^{\xi_\msf e}(\msf t,0)^p\rangle^{1/p}$, then we obtain \eqref{Equation: Box Upper Bound}.
On the other hand, for $0<p<1$ we have that
$(x+y)^p\leq x^p+y^p$ for all $x,y\geq0$, hence
\[0\leq\big\langle u^{\xi_\msf e}(\msf t,x)^p\big\rangle-\big\langle u^{\xi_\msf e}_{\msf R}(\msf t,0)^p\big\rangle
\leq\big\langle\big(u^{\xi_\msf e}(\msf t,0)-u^{\xi_\msf e}_{\msf R}(\msf t,0)\big)^p\big\rangle=o\big(\langle u^{\xi_\msf e}(\msf t,0)^p\rangle\big);\]
this also implies \eqref{Equation: Box Upper Bound}.

In order to establish \eqref{Equation: Moment Upper Bound 1 Reduction}, we use two tools:
Let $\tau_r:=\inf\{t\geq0:W_t\not\in Q_r\}$
denote the first hitting time of $Q_r$'s boundary by $W$.
By the Feynman-Kac formula, one has
\begin{align}
\label{Equation: Moment Upper Bound 1 Reduction - FK}
u^{\xi_\msf e}(\msf t,0)-u^{\xi_\msf e}_{\msf R}(\msf t,0)=\mbb E_0\left[\exp\left(\int_0^{\msf t}\xi_\msf e(W_s)\d s\right)\mathbbm1_{\{\tau_{\msf R}\leq \msf t\}}\right],
\end{align}
and by the reflection principle,
\begin{align}
\label{Equation: Moment Upper Bound 1 Reduction - Reflection}
\mbb P_0[\tau_{\msf R}\leq \msf t]\leq\mr e^{-C\msf R^2/\msf t}
\end{align}
for some constant $C>0$ independent of $\msf m$.
The details of the proof of \eqref{Equation: Moment Upper Bound 1 Reduction}
using these depend on which case we are considering,
and thus we split the remainder of the proof into two steps.

\subsubsection{Sub-1, Crt-1, and Sup}

Suppose first that $p\geq1$. In that case,
Thanks to \eqref{Equation: Moment Upper Bound 1 Reduction - FK} and \cite[(3.17)]{GartnerKonig}
(note that, in the notation of \cite{GartnerKonig}, $R(pt)$ corresponds to $\msf R$ in our setting),
we have that for $p\geq1$ and large $\msf m$,
\begin{align}
\label{Equation: Moment Upper Bound 1 Reduction - p geq 1}
\big\langle\big(u^{\xi_\msf e}(\msf t,0)-u^{\xi_\msf e}_{\msf R}(\msf t,0)\big)^p\big\rangle\leq\mr e^{H_\msf e(p\msf t)-C\msf R^2/\msf t}.
\end{align}
If we combine this with the fact that
\begin{align}
\label{Equation: Moment Upper Bound 1 Reduction Lower - p geq 1}
\langle u^{\xi_\msf e}(\msf t,0)^p\rangle\geq\mr e^{-H_\msf e(p\msf t)+\be_\msf e(p\msf t)\bs\chi^p(1+o(1))}\qquad\text{as }\msf m\to\infty
\end{align}
thanks to \eqref{Equation: All Moments Lower}, we have that
\[\frac{\big\langle\big(u^{\xi_\msf e}(\msf t,0)-u^{\xi_\msf e}_{\msf R}(\msf t,0)\big)^p\big\rangle}{\langle u^{\xi_\msf e}(\msf t,0)^p\rangle}\leq\mr e^{\be_\msf e(p\msf t)\bs\chi^p(1+o(1))-C\msf R^2/\msf t}\qquad\text{as }\msf m\to\infty.\]
When $0<p<1$, we use Jensen's inequality and \eqref{Equation: Moment Upper Bound 1 Reduction - p geq 1}
in the case $p=1$ to get
\[\big\langle\big(u^{\xi_\msf e}(\msf t,0)-u^{\xi_\msf e}_{\msf R}(\msf t,0)\big)^p\big\rangle
\leq\big\langle\big(u^{\xi_\msf e}(\msf t,0)-u^{\xi_\msf e}_{\msf R}(\msf t,0)\big)\big\rangle^p
\leq\mr e^{pH_\msf e(\msf t)-pC\msf R^2/\msf t}.\]
Once again combining this with \eqref{Equation: All Moments Lower} yields
\[\frac{\big\langle\big(u^{\xi_\msf e}(\msf t,0)-u^{\xi_\msf e}_{\msf R}(\msf t,0)\big)^p\big\rangle}{\langle u^{\xi_\msf e}(\msf t,0)^p\rangle}\leq\mr e^{pH_\msf e(\msf t)-H_\msf e(p\msf t)+\be_\msf e(p\msf t)\bs\chi^p(1+o(1))-pC\msf R^2/\msf t}\qquad\text{as }\msf m\to\infty.\]
Since $\be_\msf e(p\msf t)=o(H_\msf e(p\msf t))=o(H_\msf e(\msf t))$ in cases Sub-1, Crt-1, and Sup,
in order to prove \eqref{Equation: Moment Upper Bound 1 Reduction}
it suffices to find a function $\msf R$ that satisfies $\al_\msf e(p\msf t)=o(\msf R)$,
$\msf R=\mr e^{o(\be_\msf e(p\msf t))}$, and $H_\msf e(p\msf t)=o(\msf R^2/\msf t)$.
It is easily seen to be the case that $\msf R=H_\msf e(p\msf t)$ satisfies all three of these conditions.

\subsubsection{Sub-2, Sub-3, and Crt-2}

If $p\geq1$, then by Jensen's inequality and \eqref{Equation: Moment Upper Bound 1 Reduction - FK}, we have that
\[\big\langle\big(u^{\xi_\msf e}(\msf t,0)-u^{\xi_\msf e}_{\msf R}(\msf t,0)\big)^p\big\rangle\leq\left\langle\mbb E_0\left[\exp\left(\int_0^{\msf t}p\xi_\msf e(W_s)\d s\right)\mathbbm1_{\{\tau_{\msf R}\leq \msf t\}}\right]\right\rangle.\]
If we then apply H\"older's inequality and \eqref{Equation: Moment Upper Bound 1 Reduction - Reflection}, we get that
\[\big\langle\big(u^{\xi_\msf e}(\msf t,0)-u^{\xi_\msf e}_{\msf R}(\msf t,0)\big)^p\big\rangle\leq\left\langle\mbb E_0\left[\exp\left(\int_0^{\msf t}2p\xi_\msf e(W_s)\d s\right)\right]\right\rangle^{1/2}
\mr e^{-C\msf R^2/2\msf t}\]
as $\msf m\to\infty$.
Similarly, when $0<p<1$ we have that
\[\big\langle\big(u^{\xi_\msf e}(\msf t,0)-u^{\xi_\msf e}_{\msf R}(\msf t,0)\big)^p\big\rangle\leq\left\langle\mbb E_0\left[\exp\left(\int_0^{\msf t}2\xi_\msf e(W_s)\d s\right)\right]\right\rangle^{p/2}
\mr e^{-pC\msf R^2/2\msf t}.\]

In cases Sub-2, Sub-3, and Crt-2, it is always the case that
$\al_\msf e(p\msf t)=o(\be_\msf e(p\msf t))$,
$\be_\msf e(p\msf t)=\mr e^{o(\be_\msf e(p\msf t))}$, and
$\be_\msf e(p\msf t)=o(\be_\msf e(p\msf t)^2/\msf t)$.
Moreover, since $H_\eps=0$ in those cases, we know from 
\eqref{Equation: All Moments Lower}
that $\langle u^{\xi_\msf e}(\msf t,0)^p\rangle^{-1}=\mr e^{O(\be_\msf e(p\msf t))}$.
Thus, by choosing $\msf R=\be_\msf e(p\msf t)$, in order to prove \eqref{Equation: Moment Upper Bound 1 Reduction}, it is enough to show that
\[\left\langle\mbb E_0\left[\exp\left(\int_0^{\msf t}\theta\xi_\msf e(W_s)\d s\right)\right]\right\rangle\leq\mr e^{O(\be_\msf e(p\msf t))}\]
for every $\theta>0$.
On the one hand, in cases Sub-1 and Sub-2, this follows from the fact that
\[\left\langle\mbb E_0\left[\exp\left(\int_0^{\msf t}\theta\xi_\msf e(W_s)\d s\right)\right]\right\rangle\leq\left\langle\mbb E_0\left[\exp\left(\int_0^{\msf t}\theta\xi(W_s)\d s\right)\right]\right\rangle\leq\mr e^{O(\be_\msf e(p\msf t))},\]
where the first inequality follows from the same argument as in the proof of \eqref{Equation: Monotonicity of Jc},
and the second inequality follows from \eqref{Equation: Sub Known} (except that we replace $\ga$ by $\theta^2\ga$
in the variational problem $\mbf M$).
On the other hand, in case Crt-2, we simply note that
\[\left\langle\mbb E_0\left[\exp\left(\int_0^{\msf t}\theta\xi_\msf e(W_s)\d s\right)\right]\right\rangle
=\mbb E_0\left[\exp\left(\int_{[0,\msf t]^2}\theta^2\ga_\msf e(W_u-W_v)\d u\dd v\right)\right]
\leq\mr e^{\theta^2\msf t^2\msf e^{-2}\|\ga_1\|_\infty},\]
recalling that $\msf e^{-2}=O(\be_\msf e(p\msf t))$ and that $\msf t$ is bounded in that case.
With this, the proof of Lemma \ref{Lemma: Moment Upper Bound 1} is now complete.

\subsection{Proof of Lemma \ref{Lemma: Moment Upper Bound 2}}
\label{Section: Upper 2}

The proof of this result is very similar
to that of Lemma \ref{Lemma: Moment Lower Bound for 0<p<1 - 1}.
We begin with the following general bound:

\begin{proposition}
\label{Proposition: Averaged Upper Bound}
Let $p>0$, $\theta\geq1$ and $\msf R=\mr e^{o(\be_\msf e(p\msf t))}$.
For any function $\msf d$ such that $0\leq\msf d<\msf t$,
\[\Big\langle\big(u^{\theta\xi_\msf e}_{\msf R}(\msf t-\msf d,\cdot),\mathbbm 1\big)^p\Big\rangle\leq
\begin{cases}
\mr e^{o(\be_\msf e(p\msf t))}\left(\Big\langle\big(u^{\theta\xi_\msf e}_{\msf R}(p\msf t,\cdot),\mathbbm 1\big)\Big\rangle+1\right)&\text{if }p\geq1
\vspace{3pt}\\
\mr e^{o(\be_\msf e(p\msf t))}\left(\sum_{k=1}^\infty\mr e^{p\msf t\la^{\theta\xi_\msf e}_k(Q_{\msf R})}+1\right)
&\text{if }0<p<1
\end{cases}.
\]
\end{proposition}
\begin{proof}
By arguing as in \cite[(3.37)]{GartnerKonig},
for every $p>0$, we have that
\begin{align}
\label{Equation: Averaged Upper Bound Sum}
\big(u^{\theta\xi_\msf e}_{\msf R}(\msf t-\msf d,\cdot),\mathbbm 1\big)^p\leq\left(\sum_{k=1}^\infty\mr e^{\msf t\la^{\theta\xi_\msf e}_k(Q_{\msf R})}\big(e^{\theta\xi_\msf e}_k(Q_{\msf R}),\mathbbm 1\big)^2+\big(2\msf R\big)^{2d}\right)^p.
\end{align}
Suppose first that $p\geq1$. Since
\[f(k):=\frac{\big(e^{\theta\xi_\msf e}_k(Q_{\msf R}),\mathbbm 1\big)^2}{(2\msf R)^{2d}},\qquad k\in\mbb N\]
is a probability measure on $\mbb N$, an application of Jensen's inequality yields
\begin{align*}
\left(\sum_{k=1}^\infty\mr e^{\msf t\la^{\theta\xi_\msf e}_k(Q_{\msf R})}\big(e^{\theta\xi_\msf e}_k(Q_{\msf R}),\mathbbm 1\big)^2\right)^p
&\leq(2\msf R)^{2d(p-1)}\sum_{k=1}^\infty\mr e^{p\msf t\la^{\theta\xi_\msf e}_k(Q_{\msf R})}\big(e^{\theta\xi_\msf e}_k(Q_{\msf R}),\mathbbm 1\big)^2\\
&=(2\msf R)^{2d(p-1)}\big(u^{\theta\xi_\msf e}_{\msf R}(p\msf t,\cdot),\mathbbm 1\big).
\end{align*}
If we combine this with
$\msf R=\mr e^{o(\be_\msf e(p\msf t))}$, the inequality $(x+y)^p\leq2^p(x^p+y^p)$
for all $x,y>0$ and $p\geq1$, and \eqref{Equation: Averaged Upper Bound Sum},
then we obtain the claimed result for $p\geq1$.

Now suppose that $0<p<1$.
Since $(\sum_kx_k)^p\leq\sum_kx_k^p$ for every $0<p<1$ and $x_k\geq0$, we get
from \eqref{Equation: Averaged Upper Bound Sum} that
\[\big(u^{\theta\xi_\msf e}_{\msf R}(\msf t-\msf d,\cdot),\mathbbm 1\big)^p\leq\sum_{k=1}^\infty\mr e^{p\msf t\la^{\theta\xi_\msf e}_k(Q_{\msf R})}\big(e^{\theta\xi_\msf e}_k(Q_{\msf R}),\mathbbm 1\big)^{2p}+\big(2\msf R\big)^{2dp}.\]
Given that $\msf R=\mr e^{o(\be_\msf e(p\msf t))}$ and
$\big(e^{\theta\xi_\msf e}_k(Q_{\msf R}),\mathbbm 1\big)\leq(2\msf R)^d$, the result follows.
\end{proof}
With this in hand, we may now prove \eqref{Equation: Averaged Noise Upper Bound - Regular} and \eqref{Equation: Averaged Noise Upper Bound - Singular}:

\subsubsection{Proof of \eqref{Equation: Averaged Noise Upper Bound - Regular}}

Let $A_{\msf d,\msf t}(Q_{r})$ be as in \eqref{Equation: A Event},
and let $\msf d$, $\msf c$, and $\msf h$ be as in \eqref{Equation: d, c and h}.
By the Feynman-Kac formula,
\begin{multline*}
u_{\msf R}^{\xi_\msf e}(\msf t,0)\leq\mbb E_0\left[\exp\left(\int_{0}^{\msf t}\xi_\msf e(W_s)\d s\right)\mathbbm 1_{\{A_{\msf d,\msf t}(Q_{\msf R})\}}(\mathbbm 1_{\{\int_0^{\msf d}\xi_\msf e(W_s)\d s\leq\msf c\}}+\mathbbm 1_{\{\int_0^{\msf d}\xi_\msf e(W_s)\d s>\msf c\}})\right]\\
\leq\mr e^{\msf c}\mbb E_0\left[\exp\left(\int_{\msf d}^{\msf t}\xi_\msf e(W_s)\d s\right)\mathbbm 1_{\{A_{\msf d,\msf t}(Q_{\msf R})\}}\right]+\mbb E_0\left[\exp\left(\int_{0}^{\msf t}\xi_\msf e(W_s)\d s\right)\mathbbm 1_{\{\int_0^{\msf d}\xi_\msf e(W_s)\d s>\msf c\}}\right].
\end{multline*}
Since $\msf c=o(\be_\msf e(p\msf t))$ and $(x+y)^p\leq\ze(p)(x^p+y^p)$
for some constant $\ze(p)>0$
whenever $x,y\geq0$, thanks to Proposition \ref{Proposition: Averaged Upper Bound},
\eqref{Equation: Averaged Noise Upper Bound - Regular} will be proved if we show that,
for large enough $\msf m$, one has
 \begin{align}
\label{Equation: Moment Upper Bound - Two Orders 1}
\left\langle\mbb E_0\left[\exp\left(\int_{\msf d}^{\msf t}\xi_\msf e(W_s)\d s\right)\mathbbm 1_{\{A_{\msf d,\msf t}(Q_{\msf R})\}}\right]^p\right\rangle\leq
\mr e^{o(\be_\msf e(p\msf t))}\Big\langle\big( u^{\xi_\msf e}_{\msf R}(\msf t-\msf d,\cdot),\mathbbm 1\big)^p\Big\rangle
\end{align}
and
\begin{align}
\label{Equation: Moment Upper Bound - Two Orders 2}
\left\langle\mbb E_0\left[\exp\left(\int_{0}^{\msf t}\xi_\msf e(W_s)\d s\right)\mathbbm 1_{\{\int_0^{\msf d}\xi_\msf e(W_s)\d s>\msf c\}}\right]^p\right\rangle\leq
\begin{cases}
\mr e^{-p\msf h\msf c+4H_\msf e(p\msf t)}&\text{if }p\geq1\\
\mr e^{-p\msf h\msf c+4pH_\msf e(\msf t)}&\text{if }0<p<1
\end{cases}.
\end{align}

\eqref{Equation: Moment Upper Bound - Two Orders 1} can be proved using essentially the same argument as \eqref{Equation: Moment Lower Bound for 0<p<1 - 1 Two Orders 1}, except that this
time we use the upper bound $\sup_{x\in\mbb R^d}\Pi_\msf d(x)=O(\msf d^{-d/2})=\mr e^{o(\be_\msf e(p\msf t))}$.
As for \eqref{Equation: Moment Upper Bound - Two Orders 2}, essentially the same argument used
to arrive at \eqref{Equation: Upper Bound Without Holder} together with Jensen's inequality yields
\begin{multline*}
\left\langle\mbb E_0\left[\exp\left(\int_{0}^{\msf t}\xi_\msf e(W_s)\d s\right)\mathbbm 1_{\{\int_0^{\msf d}\xi_\msf e(W_s)\d s>\msf c\}}\right]^p\right\rangle\\
\leq\mr e^{-p\msf h\msf c}\left\langle\left(\int_0^{\msf t}\frac{1+\msf h\mathbbm1_{\{0\leq s\leq\msf d\}}}{\msf t+\msf h\msf d}\mbb E_0\left[\mr e^{(\msf t+\msf h\msf d)\xi_\msf e(W_s)}\right]\d s\right)^p\right\rangle
=\begin{cases}
\mr e^{-p\msf h\msf c+H_\msf e(p(\msf t+\msf h\msf d))}&\text{if }p\geq1\\
\mr e^{-p\msf h\msf c+pH_\msf e(\msf t+\msf h\msf d)}&\text{if }0<p<1
\end{cases}.
\end{multline*}
Then, we obtain \eqref{Equation: Moment Upper Bound - Two Orders 2} by
noting that since $\msf h\msf d\leq\msf t$ for large enough $\msf m$, we have the inequality
$(\msf t+\msf h\msf d)^2\leq2(\msf t^2+(\msf h\msf d)^2)\leq 4\msf t^2$ for large enough $\msf m$;
hence $H_\msf e(q(\msf t+\msf h\msf d))\leq\msf 4H_\msf e(q\msf t)$ for all $q>0$.

\subsubsection{Proof of \eqref{Equation: Averaged Noise Upper Bound - Singular}}

Let $\theta>1$ be arbitrary, and let $\theta'>1$ be such that $1/\theta'+1/\theta=1$.
For every function $\msf d=\msf d(\msf m)$ such that $0\leq\msf d<\msf t$,
the Feynman-Kac formula followed by two applications of H\"older's inequality yields
\begin{align*}
&\big\langle u^{\xi_\msf e}_{\msf R}(t,x)\big\rangle^p
=\left\langle\mbb E_0\left[\exp\left(\int_0^{\msf t}\xi_\msf e(W_s)\d s\right)\mathbbm 1_{\{L_{\msf t}\in\ms P(Q_{\msf R})\}}\right]^p\right\rangle\\
&\leq\left\langle\mbb E_0\left[\exp\left(\int_0^\msf d \theta'\xi_\msf e(W_s)\d s\right)\right]^{p/\theta'}\mbb E_0\left[\exp\left(\int_{\msf d}^{\msf t}\theta\xi_\msf e(W_s)\d s\right)\mathbbm 1_{\{A_{\msf d,\msf t}(Q_{\msf R})\}}\right]^{p/\theta}
\right\rangle\\
&\leq\left\langle\mbb E_0\left[\exp\left(\int_0^\msf d \theta'\xi_\msf e(W_s)\d s\right)\right]^p\right\rangle^{1/\theta'}\left\langle\mbb E_0\left[\exp\left(\int_{\msf d}^{\msf t}\theta\xi_\msf e(W_s)\d s\right)\mathbbm 1_{\{A_{\msf d,\msf t}(Q_{\msf R})\}}\right]^{p}
\right\rangle^{1/\theta}.
\end{align*}
By the Markov property,
\[\mbb E_0\left[\exp\left(\int_{\msf d}^{\msf t}\theta\xi_\msf e(W_s)\d s\right)\mathbbm 1_{\{A_{\msf d,\msf t}(Q_{\msf R})\}}\right]^p
\leq O(\msf d^{-d/2})\big(u^{\theta\xi_\msf e}_{\msf R}(\msf t-\msf d,\cdot),\mathbbm 1\big)^p.\]
Therefore, by Proposition \ref{Proposition: Averaged Upper Bound}, it suffices to
make sure that it is possible to choose $\msf d$ such that $\msf d=\mr e^{o(\be_\msf e(p\msf t))}$
and
\[\left\langle\mbb E_0\left[\exp\left(\int_0^{\msf d} \theta'\xi_\msf e(W_s)\d s\right)\right]^p\right\rangle
\leq\mr e^{o(\be_\msf e(p\msf t))}.\]
This can be proved using the
same argument as in \eqref{Equation: Moment Lower Bound for 0<p<1 - 1 One Order Bounded Part}
(up to applying Jensen's inequality to move $p$ inside of $\mbb E_0$ or outside of $\langle\cdot\rangle$,
depending on whether $p\geq1$ or $0<p<1$).

With this, the proof of Lemma \ref{Lemma: Moment Upper Bound 2}---and
therefore also \eqref{Equation: All Moments Upper}---is now complete.
In particular, combining this with \eqref{Equation: All Moments Lower}, this
concludes the proof of Theorem \ref{Theorem: All Moments}.

\section{Variational Problems}
\label{Section: Variational}

In this section, we argue
that the constant $\mc G$ in \eqref{Equation: GNS} is finite for the Riesz noise
(see Section \ref{Section: G is Finite}),
we prove Proposition \ref{Proposition: Finite Time Intermittency} (see Section \ref{Section: Finite Time Intermittency}),
and we prove Theorem \ref{Theorem: Fractional Minimizers} and Proposition \ref{Proposition: Convergence of Minimizers}
(see Sections \ref{Section: Variational Outline}--\ref{Section: Geometry of fcoord}). Since some of these
proofs involve a large number of cumbersome integrals we use two shorthands
throughout this section: On the one hand,
the notation
\eqref{Equation: Dirichlet Form} for the Dirichlet form, and on the other hand, the form
\[\mc J_c(f):=-J_c(f^2)=\frac12\iint f(x)^2\ga_c(x-y)f(y)^2\d x\dd y\]
for $c\in[0,\infty)$, recalling that $\ga_0=\ga$. Moreover, whenever we omit
the domain of integration in some integral, it can be assumed that the integral is over
all of $\mbb R^d$.

\subsection{The Constant $\mc G$ for Riesz Noise}
\label{Section: G is Finite}

We begin with the following statement:

\begin{proposition}
\label{Proposition: Riesz Best Constant}
Let $d>2$ and $\si>0$, and
take $\ga(x)=\si^2|x|^{-2}$.
There exists a finite constant $C>0$ such that
for every $f\in H^1(\mbb R^d)$ with $\|f\|_2=1$, one has
$\mc J_0(f)\leq C\ms S(f)$.
\end{proposition}

While we did not find a proof for this exact statement in the literature, it follows
from standard scaling arguments (e.g., \cite[Lemma A.4]{Chen14}) and the
fact that
\[\sup_{f\in H^1(\mbb R^d),~\|f\|_2=1}\left(\mc J_0(f)^{1/2}
-\ms S(f)\right)<\infty\]
(e.g., apply a trivial change of variables to \cite[(1.19)]{BassChenRosen}).

\subsection{Proof of Proposition \ref{Proposition: Finite Time Intermittency}}
\label{Section: Finite Time Intermittency}

The fact that $\mbf M^{\mr{crt}}_{t,p}<\infty$ follows from
its definition in \eqref{Equation: Crt-2 Variational}, and noting that if $\|f\|_2=1$, then
\begin{align}
\label{Equation: Finiteness of M crt}
\mc J_1(f)\leq\frac{\|\ga_1\|_\infty\|f\|_2^4}2=\frac{\|\ga_1\|_\infty}2<\infty.
\end{align}
Next, we show that $\mbf M^{\mr{crt}}_{t,p}$ is positive if $p>\frac{2\kappa}{t\mc G}$.
	The proof is based on scaling arguments.
	Given $f\in H^{1}(\mathbb{R}^d)$ with $\Vert f\Vert_{2}=1$, we denote
	$f_\eps(x):=\eps^{d/2}f(\eps x)$.
By definition of $\gamma_1$, and since $\om=2$, we have 
\begin{align*}
	\gamma_1(\varepsilon^{-1}x)= \gamma*p_1(\varepsilon^{-1}x)= \varepsilon^{2}\gamma_\varepsilon(x).
\end{align*}
Consequently, for any $\eps>0$ and $f\in H^{1}(\mathbb{R}^d)$ with $\Vert f\Vert_{2}=1$, one has (by a change
of variables)
\begin{align}
	pt\mc J_1(f_\eps)-\ms S(f_\eps)
	= pt\varepsilon^{2d}\mc J_1\big(f(\eps\cdot)\big)-\varepsilon^2\ms S(f)
	= \varepsilon^{2}\big(pt\mc J_\eps(f)-\ms S(f)\big).
	\label{Equation: relatingtwoVP0}
\end{align}
Thus, an application of monotone convergence (using Fourier transforms as in \eqref{Equation: f^2 Fourier Transform})
yields
\begin{align}\label{Equation: relatingtwoVP}
		\lim_{\varepsilon\to 0}\varepsilon^{-2}\big(pt\mc J_1(f_\eps)-\ms S(f_\eps)\big)
		= pt\mc J_0(f)-\ms S(f).
\end{align}
This relates the critical variational problem \eqref{Equation: Crt-2 Variational} to the following new variational problem:
\begin{equation}
	\overline{\mbf M}^{\mr{crt}}_{t,p}:= \sup_{f\in H^1(\mbb R^d),~\|f\|_2=1}\big(pt\mc J_0(f)-\ms S(f)\big).
\end{equation}
Indeed, if $p>\frac{2\kappa}{t\mc G}$, then by definition of $\mc G$
we have that $\overline{\mbf M}^{\mr{crt}}_{t,p}>0$ (in fact, we actually have that $\overline{\mbf M}^{\mr{crt}}_{t,p}=\infty$ by a similar scaling argument as above). Therefore there exists $f\in H^1(\mbb R^d)$ with $\|f\|_2=1$ such that 
\begin{align*}
	pt\mc J_0(f)-\ms S(f)>0.
\end{align*}
Hence by \eqref{Equation: relatingtwoVP} we have for $\varepsilon>0$ small enough 
\begin{align*}
	pt\mc J_1(f_\eps)-\ms S(f_\eps)>0,
\end{align*}
which implies that $\mbf M^{\mr{crt}}_{t,p}$ is positive since $f_\eps\in H^1(\mbb R^d)$
with $\|f_\eps\|_2=1$ for every $f\in H^{1}(\mathbb{R}^d)$ with $\Vert f\Vert_{2}=1$.

We now argue that if $p<\frac{2\kappa}{t\mc G}$, then $\mbf M^{\mr{crt}}_{t,p}=0$.
In this case by definition of $\mc G$, for any $f\in H^1(\mathbb{R}^d)$ with $\Vert f\Vert_2=1$ we have 
$pt\mc J_0(f)-\ms S(f)\leq 0.$
Since
$\mc J_1(f)\leq\mc J_0(f)$
(this is easily checked using a Fourier transform as in \eqref{Equation: f^2 Fourier Transform}),
this implies that $\mbf M^{\mr{crt}}_{t,p}\leq0$.
The fact that $\mbf M^{\mr{crt}}_{t,p}=0$ then immediately follows from
\eqref{Equation: relatingtwoVP}.
	
We now finish the proof of Proposition \ref{Proposition: Finite Time Intermittency} by showing that $\mbf M^{\mr{crt}}_{t,p}$ is strictly increasing in $p$
whenever $\mbf M^{\mr{crt}}_{t,p}>0$.
Take $p>q>\frac{2\kappa}{t\mc G}$.
For every $\de>0$, we can find some $g\in H^1(\mbb R^d)$ such that
$\|g\|_2=1$ and
\[qt\mc J_1(g)-\ms S(g)\geq\mbf M^{\mr{crt}}_{t,q}-\de.\]
Since $\ms S(g)\geq0$, this also implies that
\[qt\mc J_1(g)\geq\mbf M^{\mr{crt}}_{t,q}-\de.\]
In particular,
\begin{multline*}
\mbf M^{\mr{crt}}_{t,p}-\mbf M^{\mr{crt}}_{t,q}
\geq pt\mc J_1(g)-\ms S(g)-\big(qt\mc J_1(g)-\ms S(g)+\de\big)\\
=(p-q)t\mc J_1(g)-\de\geq\frac{(p-q)(\mbf M^{\mr{crt}}_{t,q}-\de)}{q}-\de.
\end{multline*}
If we take $\de\to0$ in the above, this yields
\[\mbf M^{\mr{crt}}_{t,p}-\mbf M^{\mr{crt}}_{t,q}\geq\frac{p-q}{q}\mbf M^{\mr{crt}}_{t,q}>0,\]
concluding the proof.

\subsection{Outline of Proof of Theorem \ref{Theorem: Fractional Minimizers} and Proposition \ref{Proposition: Convergence of Minimizers}}
\label{Section: Variational Outline}

The remainder of Section \ref{Section: Variational} is devoted to the proof of
Theorem \ref{Theorem: Fractional Minimizers} and Proposition \ref{Proposition: Convergence of Minimizers}.
In Section \ref{Section: Variational Outline}, we provide an outline of the proof,
which relies on a number of technical lemmas (i.e., Lemmas \ref{Lemma: M c p is finite and positive}--\ref{Lemma: Geometry of fcoord}). Then, in Sections
\ref{Section: M c p is finite and positive}--\ref{Section: Geometry of fcoord}, we prove these technical lemmas.

\subsubsection{Step 1. Finiteness, Existence, and Convergence}
\label{Section: Variational Outline Step 1}

We begin by proving that $\mbf M$ and $\mbf M_{\mf c,p}$ are positive, finite,
and have maximizers, as well as the two limits \eqref{Equation: Convergence to M}
and \eqref{Equation: Convergence to fstar}.
The fact that $\mbf M$ is positive and finite can be proved by using the same argument as in
 \cite[Lemma A.2 with $\al_0=0$]{ChenHuSongXing}. As for $\mbf M_{\mf c,p}$, we have the
 following:

\begin{lemma}
\label{Lemma: M c p is finite and positive}
Let Assumption \ref{Assumption: Noise} hold
with $0<\om<2$.
$\mbf M_{\mf c,p}\in(0,\infty)$ for every $\mf c,p>0$.
\end{lemma}

The proof of this lemma, which we provide in Section \ref{Section: M c p is finite and positive}, uses a scaling argument
similar to \eqref{Equation: relatingtwoVP0} and \eqref{Equation: relatingtwoVP}.
Next, since $\ga_{p^{1/(2-\om)}\mf c}\in L^\infty(\mbb R^d)$ for every $\mf c,p\in(0,\infty)$,
a standard application of Lions' concentration-compactness principle
(e.g., combine Lemma \ref{Lemma: M c p is finite and positive} with \cite[Theorem III.2]{Lions})
implies that maximizers of $\mbf M_{\mf c,p}$ exist. Regarding the existence of maximizers of $\mbf M$,
we need the following two lemmas:

\begin{lemma}
\label{Lemma: M is continuous}
Let Assumption \ref{Assumption: Noise} hold
with $0<\om<2$. For every $R>0$, the map $f\mapsto\mc J_0(f)$
is continuous with respect to the $L^2$ norm on the set $\{f\in L^2(\mbb R^d):\|f\|_{H^1(\mbb R^d)}<R\}$.
\end{lemma}

\begin{lemma}
\label{Lemma: M compactness}
Let Assumption \ref{Assumption: Noise} hold
with $0<\om<2$. Let $(f_n)_{n\in\mbb N}\subset H^1(\mbb R^d)$ be a sequence
of functions such that $\|f_n\|_2=1$ for all $n\in\mbb N$ and
\begin{align}
\label{Equation: M compactness hyp}
\lim_{n\to\infty}\big(\mc J_0(f_n)-\ms S(f_n)\big)=\mbf M.
\end{align}
There exists a subsequence $(n_k)_{k\in\mbb N}$,
a function $f_\star\in H^1(\mbb R^d)$, and a sequence $(z_k)_{k\in\mbb N}\subset\mbb R^d$ such that
\begin{align}
\label{Equation: M compactness}
\lim_{k\to\infty}\|f_{n_k}(\cdot-z_n)-f_\star\|_{H^1(\mbb R^d)}=0.
\end{align}
\end{lemma}

Lemma \ref{Lemma: M is continuous}, which we prove in Section \ref{Section: M is continuous},
is straightforward. The statement of Lemma \ref{Lemma: M compactness} is only new in the case of fractional noise;
in the case of white and Riesz noises, see \cite[Theorem III.2]{Lions} (note that
\cite[Theorem III.2]{Lions} does not apply to fractional noise because of condition \cite[(24)]{Lions}). Thus, our proof of Lemma \ref{Lemma: M compactness},
which we provide in Section \ref{Section: M compactness}, only covers the fractional noise.
As explained in Section \ref{Section: M compactness}, the proof relies on the version of
Lions' concentration-compactness principle stated in \cite[Lemma III.1]{Lions}.

The existence of Maximizers of $\mbf M$ can now be argued as follows:
Let $f_{n_k}$ and $f_\star$ be as in the statement of Lemma \ref{Lemma: M compactness}.
Since $\mbf M$ is positive and $\om<2$, it is clear by \eqref{Equation: M compactness hyp} and \eqref{Equation: Chen's bound} that
there exists some $R>0$ large enough such that $f_{n_k}\in\{f\in L^2(\mbb R^d):\|f\|_{H^1(\mbb R^d)}<R\}$
for all $k\in\mbb N.$
In particular, by combining Lemma \ref{Lemma: M is continuous} with \eqref{Equation: M compactness hyp} and \eqref{Equation: M compactness},
we get that
\[\mbf M=\lim_{k\to\infty}\big(\mc J_0(f_{n_k})-\ms S(f_{n_k})\big)
=\big(\mc J_0(f_\star)-\ms S(f_\star)\big);\]
hence $f_\star$ is a maximizer of $\mbf M$.

Finally, \eqref{Equation: Convergence to M} was proved earlier in \eqref{Equation: Sub-3 LDP 2},
and the limit \eqref{Equation: Convergence to fstar} follows
from a combination of Lemmas \ref{Lemma: M is continuous} and \ref{Lemma: M compactness}
with the fact that for every $\mf c,p>0$, one has
\[\mbf M\geq\Big(\mc J_0(f^{(\mf c)}_\star)-\ms S(f^{(\mf c)}_\star)\Big)
\geq\Big(\mc J_{p^{1/(2-\om)}\mf c}(f^{(\mf c)}_\star)-\ms S(f^{(\mf c)}_\star)\Big)=\mbf M_{\mf c,p},\]
the latter of which implies by \eqref{Equation: Convergence to M} that the sequence
$f^{(\mf c)}_\star$ satisfies \eqref{Equation: M compactness hyp} as $\mf c\to0$.

\subsubsection{Step 2. Geometric Properties}

Following-up on Section \ref{Section: Variational Outline Step 1},
in order to prove Theorem \ref{Theorem: Fractional Minimizers} and
Proposition \ref{Proposition: Convergence of Minimizers},
it only remains to establish properties (1)--(3) in the statement of Theorem \ref{Theorem: Fractional Minimizers}.

We begin with the proof of properties (1) and (2),
which follows the outline provided in \cite[Theorem 2]{FrohlichLenzmann}.
Let $f_\star$ be a maximizer of $\mbf M$ with $\ga(x)=\si^2\prod_{i=1}^d|x_i|^{-\om_i}$.
Define the function $W_\star:=\ga*f_\star^2$. By definition of $\mbf M$, the function $f_\star$
is the ground state (i.e., the eigenvector of the smallest eigenvalue) of the Schr\"odinger operator $-\ka\De-W_\star$,
with eigenvalue $-\mbf M<0$. Since $\ga$ is locally integrable and $\|f_\star\|_2=1$,
$W_\star$ is locally integrable. Therefore,
by \cite[Theorem XIII.48 (a)]{ReedSimonIV},
$f_\star$ is either positive or negative; and by \cite[Theorem C.3.3]{SimonSemigroup},
$f_\star$ has exponential decay.
Since we can write $f_\star=\mr e^{t\mbf M}\mr e^{-t(-\ka\De-W_\star)}f_\star$,
\cite[Theorem B.3.2]{SimonSemigroup} implies that $f_\star$ is continuous.
Finally, since $f_\star$ is continuous and decays exponentially, the
fact that it is smooth follows from a standard elliptic bootstraping argument,
such as \cite[Theorem 8 (iii)]{Lieb} (more precisely, we can replace replace $V_\phi$ in \cite{Lieb}
by $W_\star$, and replace $Y_e$ in \cite{Lieb} by the Green function of $(\ka\De-\mbf M)^{-1}$,
and then the same argument used to prove \cite[Theorem 8 (iii)]{Lieb} can be applied).

It now only remains to prove Theorem \ref{Theorem: Fractional Minimizers} (3).
We recall some standard definitions and terminology:
Given a measurable set $A\subset\mbb R^d$ and a unit vector $u\in\mbb R^d$,
we use $S_u(A)$ to denote the Steiner symmetrization of $A$ with respect
to the plane perpendicular to $u$ (see, e.g., \cite[Definition 2.3]{EvansGariepy}).
We recall from \cite[Theorem 1.13]{LiebLoss} that any nonnegative measurable function
$f:\mbb R^d\to[0,\infty)$ can be written using the layer cake representation
\[f(x)=\int_0^\infty\mathbbm1_{\{y\in\mbb R^d:f(y)>\ell\}}(x)\d \ell,\]
and that the Steiner symmetrization of $f$ along any unit vector $u$ is defined as
\[S_u(f)(x):=\int_0^\infty\mathbbm1_{S_u(\{y\in\mbb R^d:f(y)>\ell\})}(x)\d \ell\]
(e.g., \cite[Page 87]{LiebLoss}).
Finally,
let $e_1,e_2,\ldots,e_d$ denote the standard basis vectors in
$\mbb R^d$. Given a nonnegative measurable $f$, we denote
$f_{\mr{coord}}:=S_{e_d}(S_{e_{d-1}}(\cdots (S_{e_1}(f))).$
We have the following two results regarding $f_{\mr{coord}}$,
which are proved in Sections \ref{Section: Minimizer is fcoord} and \ref{Section: Geometry of fcoord}, respectively:

\begin{lemma}
\label{Lemma: Minimizer is fcoord}
Let $\ga(x)=\si^2\prod_{i=1}^d|x_i|^{-\om_i}$ with $0<\om<2$.
Let $f\in H^1(\mbb R^d)$ be nonnegative, smooth and such that $\|f\|_2=1$.
If there is no $z\in\mbb R^d$ such that $f=f_{\mr{coord}}(\cdot-z)$, then
\[\mc J_0(f)-\ms S(f)
<\mc J_0(f_{\mr{coord}})-\ms S(f_{\mr{coord}}).\]
\end{lemma}

\begin{lemma}
\label{Lemma: Geometry of fcoord}
Let $f:\mbb R^d\to[0,\infty)$ be measurable.
For every fixed $1\leq i\leq d$ and $x_j\in\mbb R$ (for all $j\neq i$),
there exists a nonincreasing function $\rho:[0,\infty)\to[0,\infty)$ such that
\[f_{\mr{coord}}(x_1,\ldots,x_{i-1},r,x_{i+1},\ldots,x_d)=\rho(|r|)\qquad\text{for every }r\in\mbb R.\]
\end{lemma}

We are now in a position to prove Theorem \ref{Theorem: Fractional Minimizers} (3):
By Lemma \ref{Lemma: Geometry of fcoord}, it suffices to prove that if $f_\star$
is a maximizer of $\mbf M$ for the Riesz noise, then $f_\star=(f_\star)_{\mr{coord}}(\cdot-z)$
for some $z\in\mbb R^d$. This follows from Lemma
\ref{Lemma: Minimizer is fcoord}. With this, we have now completed the proof
of Theorem \ref{Theorem: Fractional Minimizers} and Proposition \ref{Proposition: Convergence of Minimizers}
(up to proving Lemmas \ref{Lemma: M c p is finite and positive}--\ref{Lemma: Geometry of fcoord},
which we now carry out).

\subsection{Proof of Lemma \ref{Lemma: M c p is finite and positive}}
\label{Section: M c p is finite and positive}

The finiteness of $\mbf M_{\mf c,p}$ follows from the same argument as in
\eqref{Equation: Finiteness of M crt}. As for positivity,
let $f$ be such that $\|f\|_2=1$, and let $f_\eps(x):=\eps^{d/2}f(\eps x)$.
Arguing as in \eqref{Equation: relatingtwoVP0} and \eqref{Equation: relatingtwoVP},
except that in the subcritical case one has
\[\ga_{p^{1/(2-\om)}\mf c}(\eps^{-1} x)=\eps^{\om}\ga_{\eps p^{1/(2-\om)}\mf c}(\eps^{-1} x),\]
we see that as $\eps\to0$,
\begin{align*}
\mc J_1(f_\eps)-\ms S(f_\eps)
=\varepsilon^{2}\Big(\eps^{\om-2}\mc J_{\eps p^{1/(2-\om)}\mf c}(f)-\ms S(f)\Big)
=\varepsilon^{2}\Big(\eps^{\om-2}\big(1+o(1)\big)\mc J_0(f)-\ms S(f)\Big).
\end{align*}
Since $\eps^{\om-2}\to\infty$ when $\eps\to0$, the positivity of $\mbf M_{\mf c,p}$
follows from the existence of a function $f\in H^1(\mbb R^d)$ such that
$\mc J_0(f)>0$, which is trivial.

\subsection{Proof of Lemma \ref{Lemma: M is continuous}}
\label{Section: M is continuous}

Since $\ga$ is a covariance function, the bilinear map
\begin{align}
\label{Equation: Bilinear Map}
\langle f,g\rangle_\ga:=\frac12\iint f(x)\ga(x-y)g(y)\d x\dd y
\end{align}
is a semi-inner product. Thus, its induced semi-norm satisfies the (reverse) triangle inequality, whence
for every functions $f_1$ and $f_2$ (noting that $\mc J_0(f)=\langle f^2,f^2\rangle_\ga$), one has
\begin{align}
\left|\sqrt{\mc J_0(f_1)}-
\sqrt{\mc J_0(f_2)}\right|
\leq\sqrt{\langle f_1^2-f_2^2,f_1^2-f_2^2\rangle_\ga}
\label{Equation: Continuity of Semi-Inner-Product}
\leq\sqrt{\sum_{i=1}^2\left|\langle f_1^2-f_2^2,f_i^2\rangle_\ga\right|}.
\end{align}
The claimed result then follows from the fact that in the subcritical regime, there exists a constant $C>0$ such that for every $f\in H^1(\mbb R^d)$
and $x\in\mbb R^d$, one has
\begin{align}
\label{Equation: Chen's bound}
\int\ga(x-y)f(y)^2\d y\leq C\|f\|_2^{2-\om}\|\nabla f\|_2^{\om}.
\end{align}
Indeed, 
if we combine \eqref{Equation: Chen's bound} with \eqref{Equation: Continuity of Semi-Inner-Product},
we get that \eqref{Equation: Continuity of Semi-Inner-Product} is bounded above by
\begin{multline}
\label{Equation: Continuity of Semi-Inner-Product 2-pre}
\sqrt{C\left(\int|f_1(x)^2-f_2(x)^2|\d x\right)\sum_{i=1}^2\|f_i\|_2^{2-\om}\|\nabla f_i\|_2^\om}\\
=\sqrt{C\left(\int|f_1(x)-f_2(x)||f_1(x)+f_2(x)|\d x\right)\sum_{i=1}^2\|f_i\|_2^{2-\om}\|\nabla f_i\|_2^\om},
\end{multline}
If we now apply Cauchy-Schwarz in the $\dd x$ integral in \eqref{Equation: Continuity of Semi-Inner-Product 2-pre},
then we get that \eqref{Equation: Continuity of Semi-Inner-Product} is bounded above by
\begin{align}
\label{Equation: Continuity of Semi-Inner-Product 2}
\sqrt{C\|f_1-f_2\|_2\|f_1+f_2\|_2\sum_{i=1}^2\|f_i\|_2^{2-\om}\|\nabla f_i\|_2^\om}
\end{align}
which immediately implies the result.
To conclude, we note that
the inequality \eqref{Equation: Chen's bound} in the case of Riesz and fractional noise is proved in \cite[(A.17) and (A.29)]{Chen14}.
In the case of one-dimensional white noise, we remark that every $f\in H^1(\mbb R)$
is uniformly H\"older-$1/2$ (since $|f(y)-f(x)|\leq\left|\int_x^yf'(z)\d z\right|\leq\|f'\|_2\sqrt{y-x}$
for all $y>x$).
Therefore, since $f$ is square-integrable, $f(x)\to0$ as $x\to\pm\infty$, which then implies by Cauchy-Schwarz that
\[\int\de_0(x-y)f(y)^2\d y= f(x)^2=-\int_x^\infty(f(y)^2)'=-2\int_x^\infty f(y)f'(y)\d y\leq2\|f\|_2\|f'\|_2.\]

\subsection{Proof of Lemma \ref{Lemma: M compactness}}
\label{Section: M compactness}

As mentioned in Section \ref{Section: Variational Outline Step 1},
we only need to prove the result in the case of subcritical fractional
noise; we thus henceforth assume that $\ga(x)=\si^2\prod_{i=1}^d|x_i|^{-\om_i}$ with $0<\om<2$.
Our proof relies on the version of the concentration-compactness
principle stated in \cite[Lemma III.1]{Lions}, which is as follows:

\begin{lemma}[\cite{Lions}]
\label{Lemma: Lions}
Let $(f_n)_{n\in\mbb N}$ be such that $f_n\geq0$ and $\|f_n\|_2=1$
for every $n$. There exists a subsequence $(n_k)_{k\in\mbb N}$
that satisfies one of the following three possibilities:
\begin{enumerate}
\item ({\bf Compactness}) There exists a sequence $(z_k)_{k\in\mbb N}\subset\mbb R^d$ such that for every $\eps>0$,
there exists $R>0$ such that
\begin{align}
\label{Equation: Lions Compactness}
\int_{\{|x|<R\}}f_{n_k}(x-z_k)^2\d x\geq1-\eps,\qquad k\in\mbb N.
\end{align}
\item ({\bf Vanishing}) For every $R>0$,
\begin{align}
\label{Equation: Lions Vanishing}
\lim_{k\to\infty}\sup_{z\in\mbb R^d}\int_{\{|x|<R\}}f_{n_k}(x-z)^2\d x=0.
\end{align}
\item ({\bf Dichotomy}) There exists $0<a<1$ such that for every $\eps>0$,
there is some $k_0\in\mbb N$ and nonnegative $f_k^{(1)},f_k^{(2)}\in L^2(\mbb R^d)$ such that
for every $k\geq k_0$:
\begin{itemize}
\item $\|f_{n_k}-(f_k^{(1)}+f_k^{(2)})\|_{p}\leq\de_p(\eps)$ for $2\leq p\leq 6$, where $\de_p(\eps)\to0$ as $\eps\to0$; and
\item $\big| \|f_k^{(1)}\|_2^2 - a \big|\leq\eps$ and $\big| \|f_k^{(2)}\|_2^2 - (1-a) \big|\leq\eps$.
\end{itemize}
Moreover,
\begin{itemize}
\item $\displaystyle\sup_{\eps\in(0,1)}\sup_{k\in\mbb N}\|f_k^{(i)}\|_{H^1(\mbb R^d)}<\infty$ for $i=1,2$;
\item for every $\eps>0$, $\displaystyle\lim_{k\to\infty}\inf\left\{|x-y|:x\in\mr{supp}(f_k^{(1)})\text{ and }y\in\mr{supp}(f_k^{(2)})\right\}=\infty$; and
\item for every $\eps>0$, $\displaystyle\liminf_{k\to\infty}\int|\nabla f_{n_k}(x)|^2-|\nabla f_k^{(1)}(x)|^2-|\nabla f_k^{(2)}(x)|^2\d x\geq0$.
\end{itemize}
({\bf Remark.} Though this is not explicit in the notation,
$k_0$ and $f_k^{(i)}$ depend on $\eps$.)
\end{enumerate}
\end{lemma}

Our objective is to show that any sequence $(f_n)_{n\in\mbb N}$ that satisfies
\eqref{Equation: M compactness hyp} must also satisfy the compactness
condition stated in \eqref{Equation: Lions Compactness}. Indeed, if such is
the case, then a direct application of the argument contained in the two
paragraphs following \cite[Remark III.3]{Lions} yields \eqref{Equation: M compactness}.
Thanks to Lemma \ref{Lemma: Lions}, for this it suffices to prove that
any sequence satisfying \eqref{Equation: M compactness hyp} cannot also
satisfy the vanishing or dichotomy conditions stated in Lemma \ref{Lemma: Lions}.

\subsubsection{Vanishing}

We begin by proving that vanishing does not occur.
Suppose that the sequence $(f_n)_{n\in\mbb N}$ satisfies \eqref{Equation: M compactness hyp}, and
that we can find a subsequence such that \eqref{Equation: Lions Vanishing} occurs.
We claim that this implies that $\mbf M\leq0$, which is a contradiction.
Thanks to \eqref{Equation: M compactness hyp}, in order to prove $\mbf M\leq 0$,
it is enough to show that the vanishing condition \eqref{Equation: Lions Vanishing}
implies that
\[\limsup_{k\to\infty}\mc J_0(f_{n_k})=0.\]
For this purpose, for every $\theta>0$, this limsup is bounded above by
\begin{align}
\label{Equation: Vanishing 1}
&\limsup_{k\to\infty}\frac{\si^2}{2}\iint_{\{|x_i-y_i|\leq \theta~\forall1\leq i\leq d\}}f_{n_k}(x)^2\prod_{i=1}^d|x_i-y_i|^{-\om_i}f_{n_k}(y)^2\d x\dd y\\
\label{Equation: Vanishing 2}
&\qquad+\limsup_{k\to\infty}\sum_{j=1}^d\frac{\si^2}{2}\iint_{\{|x_j-y_j|>\theta\}}f_{n_k}(x)^2\prod_{i=1}^d|x_i-y_i|^{-\om_i}f_{n_k}(y)^2\d x\dd y.
\end{align}
We now prove that \eqref{Equation: Vanishing 1} vanishes for all $\theta>0$,
and that \eqref{Equation: Vanishing 2} can be made arbitrarily small by taking $\theta\to\infty$.

We begin with the claim regarding \eqref{Equation: Vanishing 1}.
By a straightforward change of variables and the fact that $\|f_{n_k}\|_2=1$, we have that
\begin{multline*}
\iint_{\{|x_i-y_i|\leq \theta~\forall1\leq i\leq d\}}f_{n_k}(x)^2\prod_{i=1}^d|x_i-y_i|^{-\om_i}f_{n_k}(y)^2\d x\dd y\\
\leq\sup_{x\in\mbb R^d}\int_{\{|y_i|\leq \theta~\forall1\leq i\leq d\}}\prod_{i=1}^d|y_i|^{-\om_i}f_{n_k}(x+y)^2\d y.
\end{multline*}
For every $\theta>0$, we can find a smooth and compactly supported function $g_\theta$ and
a large enough $R_\theta>0$ such that
\[\mathbbm 1_{\{|y_i|\leq \theta~\forall1\leq i\leq d\}}\leq g_\theta(y)\leq\mathbbm 1_{\{|y|<R_\theta\}},\qquad y\in\mbb R^d.\]
In particular, an application of \eqref{Equation: Chen's bound} with the function
\[f(y)^2=f_{n_k}(x+y)^2g_\theta(y)\leq f_{n_k}(x+y)^2\mathbbm 1_{\{|y|<R_\theta\}}\]
yields
\begin{multline}
\label{Equation: Vanishing 1 Bound}
\sup_{x\in\mbb R^d}\int_{\{|y_i|\leq \theta~\forall1\leq i\leq d\}}\prod_{i=1}^d|y_i|^{-\om_i}f_{n_k}(x+y)^2\d y\\
\leq C\sup_{x\in\mbb R^d}\|f_{n_k}(x+\cdot)\mathbbm 1_{\{|\cdot|<R_\theta\}}\|_2^{2-\om}\cdot\sup_{x\in\mbb R^d}\Big\|\nabla \big(f_{n_k}(x+\cdot)\sqrt{g_\theta}\big)\Big\|_2^\om.
\end{multline}
The first supremum on the right-hand side of \eqref{Equation: Vanishing 1 Bound} 
goes to zero as $k\to\infty$ for every $R_\theta>0$ thanks to the vanishing condition \eqref{Equation: Lions Vanishing}.
Thus, in order to show that the lim sup in \eqref{Equation: Vanishing 1} is zero, it suffices to show
that the second supremum on the right-hand side of \eqref{Equation: Vanishing 1 Bound} remains
bounded as $k\to\infty$. For this purpose, we apply the product rule and $(x+y)^2\leq 2(x^2+y^2)$, which yields
\begin{align*}
&\sup_{k\geq0}\Big\|\nabla \big(f_{n_k}(x+\cdot)\sqrt{g_\theta}\big)\Big\|_2^2\\
&\leq\sup_{k\geq0}2\int\sum_{i=1}^d\left(\frac{\partial}{\partial y_i}f_{n_k}(x+y)\right)^{2}g_\theta(y)
+\sum_{i=1}^d\left(\frac{\partial}{\partial y_i}\sqrt{g_\theta(y)}\right)^{2}f_{n_k}(x+y)^2\d y\\
&\leq \sup_{k\geq0}2\|\nabla f_{n_k}\|_2^2+2d\max_{1\leq i\leq d}\left\|\left(\tfrac{\partial}{\partial y_i}\sqrt{g_\theta}\right)^{2}\right\|_\infty.
\end{align*}
By combining \eqref{Equation: Chen's bound} with $\mbf M>0$ and $\om<2$,
the fact that $(f_n)_{n\in\mbb N}$ satisfies \eqref{Equation: M compactness hyp} implies that
\begin{align}
\label{Equation: Uniformly Bounded Gradient}
\sup_n\|\nabla f_n\|_2<\infty,
\end{align}
as desired.

We now conclude the proof that vanishing does not occur by
showing that \eqref{Equation: Vanishing 2} can be made arbitrarily small by taking $\theta\to\infty$.
Note that, for any $\theta>0$, one has
\begin{multline}
\label{eq:pre-Chen}
\iint_{\{|x_1-y_1|>\theta\}}f_{n_k}(x)^2\prod_{i=1}^d|x_i-y_i|^{-\om_i}f_{n_k}(y)^2\d x\dd y\\
\leq\theta^{-\om_1}\iint f_{n_k}(x)^2\prod_{i=2}^d|x_i-y_i|^{-\om_i}f_{n_k}(y)^2\d x\dd y.
\end{multline}
Let us denote $\tilde x:=(x_2,\ldots,x_d)$ and similarly for $\tilde y$, as well as $\tilde\om:=\sum_{i=2}^d\om_i$. 
If we apply \eqref{Equation: Chen's bound} with $\ga(x)=\prod_{i=2}^d|x_i|^{-\om_i}$, then we get that for every fixed $x_1\in\mbb R$
\begin{align}\label{eq:Chen15}
\sup_{\tilde y\in\mbb R^{d-1}}\int f_{n_k}(x_1,\tilde x)^2\prod_{i=2}^d|x_i-y_i|^{-\om_i}\d\tilde x
\leq C\|f_{n_k}(x_1,\cdot)\|_2^{2-\tilde\om}\|\nabla f_{n_k}(x_1,\cdot)\|_2^{\tilde\om}.
\end{align}
for some constant $C>0$ independent of $x_1$, $k$, and $\theta$.
In particular,
the right-hand side of \eqref{eq:pre-Chen} is bounded above by
\[C\theta^{-\om_1}\iint\|f_{n_k}(x_1,\cdot)\|_2^{2-\tilde\om}\|\nabla f_{n_k}(x_1,\cdot)\|_2^{\tilde\om}f_{n_k}(y)^2\d x_1\dd y.\]
Now integrate out $\dd y$ from the above display, which yields
\[C\theta^{-\om_1}\int\|f_{n_k}(x_1,\cdot)\|_2^{2-\tilde\om}\|\nabla f_{n_k}(x_1,\cdot)\|_2^{\tilde\om}\d x_1.\]
Next, applying H\"older's inequality with $p=\frac1{1-\tilde\om/2}$ and $q=\frac1{\tilde\om/2}$ gives the upper bound
\[C\theta^{-\om_1}\int\|f_{n_k}(x_1,\cdot)\|_2^{2-\tilde\om}\|\nabla f_{n_k}(x_1,\cdot)\|_2^{\tilde\om}\d x_1\leq C\theta^{-\om_1}\|\nabla f_{n_k}\|_2^{\tilde\om}.\]
Combining this with \eqref{Equation: Uniformly Bounded Gradient} and \eqref{eq:pre-Chen}, we conclude that
\[\limsup_{k\to\infty}\iint_{\{|x_1-y_1|>\theta\}}f_{n_k}(x)^2\prod_{i=1}^d|x_i-y_i|^{-\om_i}f_{n_k}(y)^2\d x\dd y=O(\theta^{-\om_1}).\]
Using the same argument (but interchanging the roles of $x_1$ and $y_1$ with $x_j$ and $y_j$), one has
\begin{align}
\label{Equation: Controlling One Component}
\limsup_{k\to\infty}\iint_{\{|x_j-y_j|>\theta\}}f_{n_k}(x)^2\prod_{i=1}^d|x_i-y_i|^{-\om_i}f_{n_k}(y)^2\d x\dd y=O(\theta^{-\om_j})
\end{align}
for all $2\leq j\leq d$.
With this, we obtain that
\[\lim_{\theta\searrow 0}\limsup_{k\to\infty}\sum_{j=1}^d\iint_{\{|x_j-y_j|>\theta\}}f_{n_k}(x)^2\prod_{i=1}^d|x_i-y_i|^{-\om_i}f_{n_k}(y)^2\d x\dd y=0,\]
thus concluding the proof that vanishing cannot occur at the same time as \eqref{Equation: M compactness hyp}.

\subsubsection{Dichotomy}

We now conclude the proof of Lemma \ref{Lemma: M compactness}
by showing that the dichotomy condition stated in Lemma \ref{Lemma: Lions}
cannot occur at the same time as \eqref{Equation: M compactness hyp}.
Suppose then that $(f_n)_{n\in\mbb N}$ satisfies both \eqref{Equation: M compactness hyp}
and the dichotomy condition. The contradiction that we obtain from this assumption
is based on the following quantities: For every $0<a<1$, we define
\[\mbf M(a):=\sup_{f\in H^1(\mbb R^d),~\|f\|_2^2=a}\big(\mc J_0(f)-\ms S(f)\big).\]
By the straightforward change of variables $f\mapsto f/\sqrt a$, we note that for any $a<1$, one has
\[\mbf M(a):=\sup_{f\in H^1(\mbb R^d),~\|f\|_2=1}a\big(a\mc J_0(f)-\ms S(f)\big)<a\mbf M.\]
Therefore,
\begin{align}
\label{Equation: Dichotomy Contradiction}
\mbf M(a)+\mbf M(1-a)<a\mbf M+(1-a)\mbf M=\mbf M
\qquad\text{for every }0<a<1.
\end{align}
We now prove that the assumption that $(f_n)_{n\in\mbb N}$ satisfies both
\eqref{Equation: M compactness hyp}
and the dichotomy condition
contradicts \eqref{Equation: Dichotomy Contradiction}.

%

Let $0<a<1$ be as in the dichotomy statement in Lemma \ref{Lemma: Lions},
and for every $\eps>0$, recall the definitions of $f^{(i)}_k$ and $\de_p(\eps)$
in the same statement.
By \eqref{Equation: M compactness hyp} and the definition of the dichotomy condition, we can write
\begin{align}
\nonumber
\mbf M&=\limsup_{k\to\infty}\Big(\mc J_0(f_{n_k})-\ms S(f_{n_k})\Big)\\
\nonumber
&=\lim_{\eps\searrow0}\limsup_{k\to\infty}\Bigg(\sum_{i=1}^2\Big(\mc J_0(f_k^{(i)})-\ms S(f_k^{(i)})\Big)
+\Big(\ms S(f_{k}^{(1)})+\ms S(f_{k}^{(2)})-\ms S(f_{n_k})\Big)\\
\nonumber
&\qquad+\Big(\mc J_0(f_{n_k})-\mc J_0(f_k^{(1)})-\mc J_0(f_k^{(2)})\Big)\Bigg)\\
\label{Equation: Dichotomy Condition}
&\leq\mbf M(a)+\mbf M(1-a)+\lim_{\eps\searrow0}\limsup_{k\to\infty}\Big(\mc J_0(f_{n_k})-\mc J_0(f_k^{(1)})-\mc J_0(f_k^{(2)})\Big).
\end{align}
Thus, to get a contradiction, it suffices to prove that the remaining limit in \eqref{Equation: Dichotomy Condition} is zero.

Recall the bilinear map notation introduced in \eqref{Equation: Bilinear Map}.
For any $\eps>0$ and $k\geq0$, we can write
\begin{multline}
\label{Equation: Dichotomy Condition 2}
\mc J_0(f_{n_k})-\mc J_0(f_k^{(1)})-\mc J_0(f_k^{(2)})
=\Big(\mc J_0(f_{n_k})-\Big\langle(f_k^{(1)}+f_k^{(2)})^2,(f_k^{(1)}+f_k^{(2)})^2\Big\rangle_\ga\Big)\\
+\Big(\Big\langle(f_k^{(1)}+f_k^{(2)})^2,(f_k^{(1)}+f_k^{(2)})^2\Big\rangle_\ga-\mc J_0(f_k^{(1)})-\mc J_0(f_k^{(2)})\Big).
\end{multline}
We begin by controlling the first term on the right-hand side of \eqref{Equation: Dichotomy Condition 2}.
If we combine \eqref{Equation: Continuity of Semi-Inner-Product} and \eqref{Equation: Continuity of Semi-Inner-Product 2}
with the fact that $\|f_{n_k}\|_{H^1(\mbb R^d)}$ and $\|f_k^{(i)}\|_{H^1(\mbb R^d)}$ are uniformly
bounded in $\eps\in(0,1)$ and $k\geq0$, then we get that
\[\Big|\sqrt{\mc J_0(f_{n_k})}-\sqrt{\Big\langle(f_k^{(1)}+f_k^{(2)})^2,(f_k^{(1)}+f_k^{(2)})^2\Big\rangle_\ga}\Big|
\leq C\|f_{n_k}-(f_k^{(1)}+f_k^{(2)})\|_2\]
for every $\eps\in(0,1)$ and $k\geq0$, where the constant $C>0$ is independent of $k$ and $\eps$.
By definition of the dichotomy condition, for every $\eps>0$, there is some $k_0\geq0$ such that
\[\|f_{n_k}-(f_k^{(1)}+f_k^{(2)})\|_2\leq\de_2(\eps)\] whenever $k\geq k_0$.
Recalling that $\de_2(\eps)\to0$ as $\eps\to0$, we therefore conclude that
\begin{align}
\label{Equation: Dichotomy Condition 3}
\lim_{\eps\searrow0}\limsup_{k\to\infty}\Big(\mc J_0(f_{n_k})-\Big\langle(f_k^{(1)}+f_k^{(2)})^2,(f_k^{(1)}+f_k^{(2)})^2\Big\rangle_\ga\Big)=0.
\end{align}

It now only remains to control the second term on the right-hand side of \eqref{Equation: Dichotomy Condition 2}.
For every $\eps>0$, we have that
\begin{align}
\label{Equation: Dichotomy Distance Condition}
\lim_{k\to\infty}\inf\left\{|x-y|:x\in\mr{supp}(f_k^{(1)})\text{ and }y\in\mr{supp}(f_k^{(2)})\right\}=\infty.
\end{align}
In particular, if $k$ is large enough, then $f_k^{(1)}f_k^{(2)}=0$;
hence $(f_k^{(1)}+f_k^{(2)})^2=[f_k^{(1)}]^2+[f_k^{(2)}]^2$.
Therefore, since $\mc J_0(f)=\langle f^2,f^2\rangle_\ga$, we can expand
\[\Big\langle(f_k^{(1)}+f_k^{(2)})^2,(f_k^{(1)}+f_k^{(2)})^2\Big\rangle_\ga-\mc J_0(f_k^{(1)})-\mc J_0(f_k^{(2)})
=2\Big\langle[f_k^{(1)}]^2,[f_k^{(1)}]^2\Big\rangle_\ga.\]
Since the $\ell_2$ and $\ell_\infty$ norms are equivalent on $\mbb R^d$, \eqref{Equation: Dichotomy Distance Condition}
implies that for every $\eps>0$, there exists a sequence of indices $1\leq i(k)\leq d$ and numbers $d_k>0$ such that
\begin{itemize}
\item $|x_{i(k)}-y_{i(k)}|>d_k$ for every $x\in\mr{supp}(f_k^{(1)})$ and $y\in\mr{supp}(f_k^{(2)})$; and
\item $d_k\to\infty$ as $k\to\infty$.
\end{itemize}
In particular, we can write
\[\Big\langle[f_k^{(1)}]^2,[f_k^{(1)}]^2\Big\rangle_\ga\leq\frac{\si^2}{2}\iint_{\{|x_{i(k)}-y_{i(k)}|>d_k\}}[f_k^{(1)}(x)]^2\prod_{i=1}^d|x_i-y_i|^{-\om_i}[f_k^{(2)}(y)]^2\d x\dd y.\]
As $\sup_{k\in\mbb N}\|f_k^{(i)}\|_{H^1(\mbb R^d)}<\infty$ for $i=1,2$, by replicating the argument leading up to
\eqref{Equation: Controlling One Component}, we conclude that $\Big\langle[f_k^{(1)}]^2,[f_k^{(1)}]^2\Big\rangle_\ga\leq Cd_k^{-\om_{i(k)}}$
for some constant $C>0$ independent of $k$. Therefore, for every $\eps>0$,
\[\limsup_{k\to\infty}\Big\langle[f_k^{(1)}]^2,[f_k^{(1)}]^2\Big\rangle_\ga=0.\]
If we combine this with \eqref{Equation: Dichotomy Condition}, \eqref{Equation: Dichotomy Condition 2}, and \eqref{Equation: Dichotomy Condition 3},
then we finally obtain that dichotomy cannot occur simultaneously with \eqref{Equation: M compactness hyp},
thus concluding the proof of Lemma \ref{Lemma: M compactness}.

\subsection{Proof of Lemma \ref{Lemma: Minimizer is fcoord}}
\label{Section: Minimizer is fcoord}

For every fixed $x_2,\ldots,x_d\in\mbb R$ and nonnegative measurable function $g$,
we note that the one-dimensional function $x_1\mapsto S_{e_1}(g)(x_1,x_2,\ldots,x_d)$ is the symmetric
decreasing rearrangement of the function $x_1\mapsto g(x_1,x_2,\ldots,x_d)$
(see, e.g., \cite[Page 80]{LiebLoss}). Therefore, by the P\'olya-Szeg\H{o} inequality (e.g., \cite[Lemma 7.17]{LiebLoss}),
\[\int \left(\frac{\partial f(x)}{\partial x_1}\right)^2\d x\geq\int \left(\frac{\partial S_{e_1}(f)(x)}{\partial x_1}\right)^2\d x,\]
and since the symmetric decreasing rearrangement preserves the $L^p$-norm
(\cite[Page 81]{LiebLoss}),
\[\int \sum_{i=2}^d\left(\frac{\partial f(x)}{\partial x_1}\right)^2\d x=\int \sum_{i=2}^d\left(\frac{\partial S_{e_1}(f)(x)}{\partial x_1}\right)^2\d x.\]
In particular,
\[-\ms S(f)\leq-\ka\int |\nabla S_{e_1}(f)(x)|_2^2\d x.\]

Next, by Fubini's theorem, we can write
\begin{multline*}
\iint f(x)^2\prod_{i=1}^d|x_i-y_i|^{-\om_i}f(y)^2\d x\dd y\\
=\iint \prod_{i=2}^d|x_i-y_i|^{-\om_i}\left(\int f(x_1,\tilde x)^2|x_1-y_1|^{-\om_1}f(y_1,\tilde y)^2\d x_1\dd y_1\right)\d\tilde x\dd\tilde y,
\end{multline*}
where we denote $\tilde x:=(x_2,\ldots,x_d)$ and similarly for $\tilde y$.
Let $\tilde x,\tilde y\in\mbb R^{d-1}$ be fixed.
By the one-dimensional Riesz rearrangement inequality (e.g., \cite[Lemma 3.6]{LiebLoss}),
we have that
\begin{multline*}
\int f(x_1,\tilde x)^2|x_1-y_1|^{-\om_1}f(y_1,\tilde y)^2\d x_1\dd y_1\\
\leq
\int [S_{e_1}(f)(x_1-z_1,\tilde x)]^2|x_1-y_1|^{-\om_1}[S_{e_1}(f)(y_1-z_1,\tilde y)]^2\d x_1\dd y_1.
\end{multline*}
for any $z_1\in\mbb R$.
Moreover, since $x_1\mapsto |x_1|^{-\om_1}$ is strictly decreasing in $|x_1|$,
it follows from \cite[Theorem 3.9]{LiebLoss} that the above inequality is an equality only if
there exists $z_1\in\mbb R$ such that
\[f(x_1,\tilde x)=S_{e_1}(f)(x_1-z_1,\tilde x)
\qquad\text{and}\qquad
f(y_1,\tilde y)=S_{e_1}(f)(y_1-z_1,\tilde y)\]
for almost every $x_1,y_1\in\mbb R$.
If we choose the same $z_1$ for all
$\tilde x$ and $\tilde y$, this implies that
\begin{multline*}
\iint f(x)^2\prod_{i=1}^d|x_i-y_i|^{-\om_i}f(y)^2\d x\dd y\\
\leq\iint [S_{e_1}(f)(x_1-z_1,\tilde x)]^2\prod_{i=1}^d|x_i-y_i|^{-\om_i}[S_{e_2}(f)(y_1-z_1,\tilde y)]^2\d x\dd y,
\end{multline*}
for any $z_1\in\mbb R$, and if there is no $z_1\in\mbb R$ such that
$f=S_{e_1}(f)\big(\cdot-(z_1,0,\ldots,0)\big)$ almost everywhere, then the inequality is strict.

If we then iterate the above argument with the Steiner symmetrizations with respect to the axes $e_2,\ldots,e_d$,
we conclude the proof of Lemma \ref{Lemma: Minimizer is fcoord}.

\subsection{Proof of Lemma \ref{Lemma: Geometry of fcoord}}
\label{Section: Geometry of fcoord}

By definition,
\[S_{\mr{coord}}(f)(x):=\int_0^\infty\mathbbm1_{S_{e_d}(S_{e_{d-1}}(\cdots (S_{e_1}(\{y\in\mbb R^d:f(y)>\ell\}))))}(x)\d \ell.\]
Thus, the lemma follows from the fact that for any measurable set $A\subset\mbb R$, the set
\[S_{e_d}(S_{e_{d-1}}(\cdots (S_{e_1}(A))))\]
is symmetric with respect to every coordinate axis
(e.g., \cite[Claim \#1 in Theorem 2.4]{EvansGariepy}).

\section*{Acknowledgements}

P.Y.G.L. gratefully acknowledges the Centre de recherches math\'ematiques (CRM) for
providing a productive research environment during the Probability and PDEs thematic semester,
during which much of the writing of this paper took place. 
The work of Y.L. was supported by EPSRC through grant EP/R024456/1.

The authors thank Xia Chen and Wolfgang K\"onig
for helpful references and insightful conversations on the contents of this paper.

Last but not least, the authors thank anonymous referees for carefully reading the manuscript,
pointing out misprints and errors, and providing suggestions that improved the presentation of the paper.

\bibliographystyle{plain}
\bibliography{Bibliography}

\end{document}